\documentclass[12pt,reqno]{amsart}

\usepackage{eurosym}
\usepackage{amsfonts}
\usepackage{graphicx}
\usepackage{amsmath,amssymb,lineno,amsthm,epic,amsthm}
\usepackage{natbib}
\usepackage[usenames,dvipsnames]{color}

\usepackage{textcomp}

\usepackage{graphicx}
\usepackage[dvips,bottom=1.4in,right=1in,top=1in, left=1in]{geometry}

\newtheorem{theorem}{Theorem}[section]

\newtheorem{definition}[theorem]{Definition}

\newtheorem{lemma}[theorem]{Lemma}

\newtheorem{remark}[theorem]{Remark}

\numberwithin{equation}{section}



\begin{document}

\title[Superdiffusive fractional dynamics]{Superdiffusive fractional dynamics: Unveiling regularity results in systems with general positive self-adjoint operators}

\author{Edgardo Alvarez}
\address{E. Alvarez, Departamento de Matem\'aticas y Estadistica Universidad
del Norte, Barranquilla (COLOMBIA)}
\email{ealvareze@uninorte.edu.co}

\author{Ciprian G. Gal}
\address{C. G. Gal, Department of Mathematics, Florida International
University, Miami, FL 33199 (USA)}
\email{cgal@fiu.edu}

\author{Valentin Keyantuo}
\address{V. Keyantuo, University of Puerto Rico, Rio Piedras Campus,
Department of Mathematics, Faculty of Natural Sciences, 17 University AVE.
STE 1701 San Juan PR 00925-2537 (USA)}
\email{valentin.keyantuo1@upr.edu}

\author{Mahamadi Warma}
\address{M.  Warma, Department of Mathematical Sciences and the Center for Mathematics and Artificial Intelligence (CMAI), George Mason University,  Fairfax, VA 22030 (USA)}
\email{mwarma@gmu.edu}


\thanks{The work of V. Keyantuo and M. Warma is partially supported by the US Army Research Office (ARO) under Award NO: W911NF-20-1-0115.}

\begin{abstract}
We investigate the following fractional order in time Cauchy
problem 
\begin{equation*}
\begin{cases}
\mathbb{D}_{t}^{\alpha }u(t)+Au(t)=f(u(t)), & 1<\alpha <2, \\ 
u(0)=u_{0},\,\,\,u^{\prime }(0)=u_{1}. & 
\end{cases}%
\end{equation*}%
where $\mathbb{D}_{t}^{\alpha }u(\cdot )$ is the Caputo time-fractional
derivative of order $\alpha\in (1, 2)$ of the function $u$. Such problems are increasingly used in concrete models in
applied sciences, notably phenomena with memory effects. We obtain results on existence and regularity of weak and strong
energy solutions assuming that $A$ is any positive self-adjoint operator
in a Hilbert space, when the nonlinearity $f\in C^{1}({\mathbb{R}}) $
satisfies suitable growth conditions. 
Our aim is to obtain regularity
results {without} assuming that the operator $A$ has compact
resolvent readily extending our recent results from our previous paper \cite{AGKW}. Examples of operators $A$ are considered, mainly differential operators such as Schr\"odinger operators, as well as various nonlocal operators.
\end{abstract}

\maketitle

\section{Introduction}

The paper deals with the problem of existence of energy solutions for a
class of semilinear \emph{super-diffusive fractional equations}. More
precisely, our aim is to investigate the following initial-value problem 
\begin{equation}
\begin{cases}
\mathbb{D}_{t}^{\alpha }u\left( x,t\right) +Au\left( x,t\right) =f(u\left(
x,t\right) ), & \mbox{ in }\;\Omega \times \left( 0,T\right) , \\ 
u(\cdot ,0)=u_{0},\;\;\partial _{t}u(\cdot ,0)=u_{1} & \text{ in }\Omega .%
\end{cases}
\label{EQ-NL0}
\end{equation}%
Here $\Omega $ is a locally compact Hausdorff space, $T>0$ and $1<\alpha <2$
are real numbers and $\mathbb{D}_{t}^{\alpha }u$ denotes the Caputo
fractional derivative of order $\alpha$ with respect to $t$, which is defined by 
\begin{equation}
\mathbb{D}_{t}^{\alpha }u(t,x):=\int_{0}^{t}g_{2-\alpha }\left( t-s\right)
\partial _{s}^{2}u(s,x)ds,\ (t,x)\in \left( 0,T\right) \times \Omega ,
\label{fra}
\end{equation}%
where we have set for $\gamma> 0$:%
\begin{equation*}
g_{\gamma}\left( t\right) =\left\{ 
\begin{array}{ll}
\frac{t^{\gamma-1}}{\Gamma (\gamma )}, & \text{if }t>0, \\ 
0, & \text{if }t\leq 0.%
\end{array}%
\right.
\end{equation*}%
It is convenient to define $g_0$  as $\delta_0,$ the Dirac measure concentrated at $0.$

When the function involved in \eqref{fra} is sufficiently smooth, then %
\eqref{fra} is equivalent to the following weaker form\footnote{%
Strictly speaking, the notion of Riemann-Liouville derivative requires a
"lesser" degree of smoothness of the function involved.} (see \cite{Po99};
cf. also \cite{Ba01,KLW2, KW}): 
\begin{equation}
\mathbb{D}_{t}^{\alpha }u(t,x)=\partial _{t}^{2}\int_{0}^{t}g_{2-\alpha
}\left( t-s\right) [u(s,x)-u(x,0)-s\partial _{s}u(x,0)]ds.  \label{fra-RL}
\end{equation}%
The nonlinearity $f\in C^{1}({\mathbb{R}})$ satisfies some suitable growth
conditions as $\left\vert u\right\vert \rightarrow \infty $ (see Section 3 and Section 4).

The theory of fractional differential equations, presented in monographs
such as \cite{Go-Ma-Ki-Ro14,Mi-Ro,Po99,SKM}, is has turned out to be critical for understanding nonlocal aspects of phenomena in various areas of science, engineering and beyond.  Applications range from memory effects to
material science \cite{GalWarma2020}, mostly viscoelasticity.  Both deterministic and stochastic problems are being intensively studied. In fact, evidence has shown  that they provide a better alternative for the modeling closely describe the phenomena under consideration.  Several monographs have been dedicated to such phenomena (see again  \cite{Go-Ma-Ki-Ro14,Mi-Ro,Po99,SKM}).

It is important to note that the fractional approach may include fractional order operators as well.  We should mention here anomalous diffusion, Levy flights, fractional Brownian motion  and random walks.  Indeed,  in \eqref{EQ-NL0}, the operator $A$ may be fractional in the space variable (see e.g.  \cite{DuSpUch,GoMaVi,MaBeNe}   Thus, one may consider equations that are fractional in time and in space.  Such problems have been considered e.g. in control theory (\cite{KW,Wa19}). Thus, the importance of the fractional model encompasses both the time and the space variable.   Investigations on both aspects have undergone an almost explosive development in recent years.   
Linear and nonlinear models have been explored, with the former often providing tools for analyzing the latter. Notably, results for linear equations with Lipschitz nonlinearity \cite{Gor,Tr-Webb78} offer
insight into semilinear and quasilinear phenomena. Existence and uniqueness
theory for linear nonhomogeneous equations related to \eqref{EQ-NL0} have
been addressed in \cite{KLW2}, while integral solutions have been studied in 
\cite{Ki-Ya}.

We note that a complete study of locally and/or globally defined weak and
strong solutions and their fine regularities in the case $0<\alpha \leq 1$
has been already performed in \cite{GalWarma2020}. Here, the authors have
established some precise and optimal conditions on the nonlinearity in order
to have existence and the precise regularity of local and global weak and
strong solutions assuming that the operator governing the problem is self-adjoint, positive and has compact resolvent. In particular, these results show that case $\alpha =1$
can be recovered in a natural way. We refer the interested reader for an
extensive comparison of our work in \cite{GalWarma2020} with other investigations
for the problem when $0<\alpha \leq 1$, which lie outside the scope of the
present paper. 
Alvarez et al. \cite{AGKW} investigated the existence, uniqueness and regularity
results for the semi-linear equation \eqref{EQ-NL0} when $1<\alpha <2$,
under appropriate conditions on the data.  In this article, the authors assumed $A$ as a self-adjoint operator in $L^{2}(X)$, that is associated with a bilinear
symmetric and closed form $\mathcal{E}_{A}$, whose domain $D(\mathcal{E}%
_{A}) $ is compactly embedded into $L^{2}(X)$ and they work with a new concept of weak and strong (energy)  solutions.  Recently,  Caicedo and Gal  considered mathematical models for the behavior of superdiffusive in a physical regime which interpolates continuously between the Schr\"odinger equation and the wave equation, respectively \cite{GaTo23}.  In  addition the same authors studied approximate controllability  of a fractional (in time) Schr\"odinger  equation involving weak and strong Caputo derivatives \cite{GaTo24}.  On the other hand, in \cite{GoPrTr17} the authors investigated the fractional Schr\"odinger  equation for $0<\alpha<1$.  They using the spectral theorem to prove the existence and uniqueness  of strong solutions which are governed by an operator solution family $\{U_{\alpha}(t)\}_{t\geq 0}$.  Asymptotic properties of this solution operator are studied in \cite{GoPrPo20}.
On the other side, when the operator $A$ is almost sectorial,  Cuesta and Ponce \cite{CuPo25} showed that is possible to obtain a variation of constants formula to an abstract fractional fractional wave equation.


The present work extends this analysis to the semi-linear case %
\eqref{EQ-NL0} with $1<\alpha<2$, offering novel results applicable to a
broad class of self-adjoint operators whose associated resolvent operator may not necessarily  be compact.  In fact, in addition to well-known classical operators, our framework includes examples of differential operators whose spectrum on a bounded domain or bounded interval is not necessarily discrete.  This is the reason why the outcomes presented herein naturally expand upon the findings of \cite%
{AGKW} in which the authors considered self-adjoint operators with compact resolvent.

The present work applies to more general operators. Those considered in \cite%
{AGKW}  which included elliptic differential operators with the Dirichlet boundary condition an a bounded open subset $\Omega\subset\mathbb R^n$, the fractional Laplacian with exterior Dirichlet condition an a bounded domain in $\mathbb R^n$ and the Dirichlet-to-Neumann operators.  Many operators on unbounded subsets of $\mathbb R^n$ fit into the present framework. An important class is that of Schr\"odinger operators. Those operators from mathematical physics constitute a large area of the theory of partial differential equations and their applications. We  present conditions from the paper \cite{MaSh} by V. G. Maz'ya and M. Schubin (see also \cite[Chapter 18]{Ma11}) which ensure that the corresponding Schr\"odinger operator fits into the framework considered in the present work. Several related results have appeared in the literature dealing with strict positivity and discreteness of the spectrum for Schr\"odinger operators. The case of magnetic Schr\"odinger operators with magnetic field is treated in \cite{KoMaSh04, LeSi81, Ma11}.  Such operators are important in the theory of liquid crystals and superconductivity.

Our study focuses on energy solutions, which  exhibit many interesting properties. The
main results guarantee existence, uniqueness, and regularity of solutions,
with explicit representations provided in terms of Mittag-Leffler functions.   More precisely,  let  $A$ be a positive (i.e., nonnegative and injective)
self-adjoint operator on a separable Hilbert space $\mathcal{H}$.  By the spectral theorem, there is an $L^{2}$-space $%
L^{2}(\Omega ,\Sigma ,\mu )$ (where $(\Omega ,\Sigma ,\mu )$ a finite
measure space), a unitary operator $U:\mathcal{H}\rightarrow L^{2}(\Omega
,\Sigma ,\mu )$ and a $\Sigma $%
-measurable function $m:\Sigma \rightarrow \mathbb{R}$, unique (modulo
changes on sets of $\mu $-measure zero) 
such that $UAU^{-1}$ is the operator of multiplication by $%
m:\Omega \rightarrow \mathbb{R}$.  In fact, 
\begin{equation*}
UAU^{-1}f=mf,\quad f\in D(UAU^{-1})=\{f\in L^{2}(\Omega ,\Sigma ,\mu
):\,mf\in L^{2}(\Omega ,\Sigma ,\mu )\}.
\end{equation*}
Then, we can characterize the solutions (weak or strong energy solutions) as 
\begin{align*}
v(t,\cdot)&= U\big(E_{\alpha
,1}(-m(\cdot )t^{\alpha })u_{0}(\cdot )+tE_{\alpha ,2}(-m(\cdot )t^{\alpha
})u_{1}(\cdot )   \\
& +\int_{0}^{t}f(v\left( \cdot ,\tau \right) )(t-\tau )^{\alpha -1}E_{\alpha
,\alpha }(-m(\cdot )(t-\tau )^{\alpha })d\tau\big)U^{-1},
\end{align*} 
where $E_{\alpha,\beta}(z)$ are the Mittag-Leffler functions (with the respect change in the linear case).  
Actually, we are able to achieve more regularity of solutions compared to our previous work \cite{AGKW}.   
More precisely, in this paper, for $\sigma\in\mathbb R,$ we denote by $V_\sigma$ the domain of the operator $A^\sigma;$ and when necessary, we endow it with the graph norm. Regarding  the initial datum (in the context of weak energy solutions) we assume that is $u_0 \in V_{\widetilde{\gamma}}$, where $1/2 \leq \widetilde{\gamma} \leq 1/\alpha$.  On the other hand, in \cite[Theorems 3.2 and 4.4]{AGKW}, the assumption was that  $u_0 \in V_{1/\alpha}$. Additionally, we introduce a new notion of  strong solution for the linear case, which includes as a special case the definition given in \cite[Definition 3.3]{AGKW}.




The manuscript is organized as follows. In Section 2, we recall some basic facts about the spectral theory and fractional calculus, mainly the Mittag-Leffler functions. The fundamental estimates that will be used throughout the document are presented in Section 3. In Section 4, we provide results on local energy weak solutions, extensions, global energy weak solutions, and blow-up in the semilinear case. Section 4 is devoted to the study of strong energy solutions in the linear case. In Section 5, we prove  existence of  strong solutions in the linear case.  The next Section 6 concerns strong energy solutions in the semilinear case: local existence, their extension to larger intervals, global energy strong solutions, and a blow-up result. Finally, some examples of relevant operators $A$ are presented in Section 6. Here we mention Schr\"odinger operators, including hose with a magnetic field, and some nonlocal operators.

\section{The functional framework}

\label{sec-main}

We first introduce some background. Let $Y,Z$ be two Banach spaces endowed
with norms $\left\Vert \cdot \right\Vert _{Y}$ and $\left\Vert \cdot
\right\Vert _{Z}$, respectively. We denote by $Y\hookrightarrow Z$ if $%
Y\subseteq Z$ and there exists a constant $C>0$ such that $\left\Vert
u\right\Vert _{Z}\leq C\left\Vert u\right\Vert _{Y},$ for $u\in Y. $ This means that the injection of $Y$ into $Z$ is continuous. In
addition, if the injection is also compact we shall denote it by $Y\overset{c%
}{\hookrightarrow }Z$. By the dual $Y^{\ast }$\ of $Y$, we think of $Y^{\ast
}$ as the set of all (continuous) linear functionals on $Y$. When equipped
with the operator norm $\left\Vert \cdot \right\Vert _{Y^{\ast }}$, $Y^{\ast
}$ is also a Banach space. In addition, we shall denote by $\left\langle
\cdot ,\cdot \right\rangle _{Y^{\star },Y}$ their duality bracket.

It is well-known that any non-negative (definite) self-adjoint operator $A,$ that is in one-to-one
correspondence with the Dirichlet form $\mathcal{E}_{A}$ associated to $A$,
turns out to possess a number of good properties provided a certain Sobolev
embedding theorem holds for $V_{1/2}:=D(\mathcal{E}_{A})$ (see, for
instance, \cite[Theorem 2.9]{G-DCDS}, cf. also \cite[Chapter 1]{Fuk}).
Similar results in abstract form can be also found in the monographs \cite%
{Dav,Fuk}. Let $A$ be a positive (i.e., nonnegative and injective)
self-adjoint operator on a separable Hilbert space $\mathcal{H}$.  In fact, by the spectral theorem, there is an $L^{2}$-space $%
L^{2}(\Omega ,\Sigma ,\mu )$ (where $(\Omega ,\Sigma ,\mu )$ a finite
measure space), a unitary operator $U:\mathcal{H}\rightarrow L^{2}(\Omega
,\Sigma ,\mu )$ and a $\Sigma $%
-measurable function $m:\Sigma \rightarrow \mathbb{R}$, unique (modulo
changes on sets of $\mu $-measure zero) 
such that $UAU^{-1}$ is the operator of multiplication by $%
m:\Omega \rightarrow \mathbb{R}$. More precisely, 
\begin{equation*}
UAU^{-1}f=mf,\quad f\in D(UAU^{-1})=\{f\in L^{2}(\Omega ,\Sigma ,\mu
):\,mf\in L^{2}(\Omega ,\Sigma ,\mu )\}.
\end{equation*}

We identify $m$ with the operator $M_{m}$ of multiplication by $m$ on $%
L^{2}(\Omega ,\Sigma ,\mu )$. Then for any Borel function $G:(0,\infty
)\rightarrow \mathbb{C}$ we can define $G(A)$ by $G(A)=U^{-1}M_{G(m)}U$.
From the point of view of differential equations, this effectively allows us
to use the functional calculus associated with $A$. Let us write $%
L^{2}(\Omega )$ instead of $L^{2}(\Omega ,\Sigma ,\mu )$. Without loss of
generality, we assume that $Af(\xi )=m(\xi )f(\xi )$, $\xi \in \Omega $, for 
$f\in D(A)=\{f\in L^{2}(\Omega ):\,m\cdot f\in L^{2}(\Omega )\}$. Thus, for $%
\gamma >0$, the fractional power $A^{\gamma }$ of the operator $A$ has the
following representation: 
\begin{equation}
A^{\gamma }f(\xi )=(m(\xi ))^{\gamma }f(\xi ),\quad \xi \in \Omega ;\quad
f\in D(A^{\gamma }):=\{f\in L^{2}(\Omega ):\,m^{\gamma }\cdot f\in
L^{2}(\Omega )\}.  \label{frac-pow}
\end{equation}%
We endow $D(A^{\gamma })$ with the norm 
\begin{equation*}
\Vert f\Vert _{\gamma }=\Vert f\Vert _{L^{2}(\Omega )}+\Vert (m(\cdot
))^{\gamma }f\Vert _{L^{2}(\Omega )}.
\end{equation*}%
Note that the previous norm is the graph norm, and it is equivalent to 
\begin{equation*}
\Vert |f|\Vert _{\gamma }=\left\{ \Vert f\Vert _{L^{2}(\Omega )}^{2}+\Vert
(m(\cdot ))^{\gamma }f\Vert _{L^{2}(\Omega )}^{2}\right\} ^{1/2}
\end{equation*}%
which comes from an inner product in an obvious way.

Since $A$ is positive and self-adjoint, we have 
\begin{align*}
(UAU^{-1}f|f)_{L^2(\Omega)}=(mf|f)_{L^2(\Omega)}=\int_{\Omega}\vert
m(\xi)\vert^2|f(\xi)|^{2}\,d\mu\geq 0,\quad \forall f\in D(A).
\end{align*}
This implies that $m \geq 0$ $\mu$-almost everywhere. We further assume that 
$0\in\rho(A)$, so that the graph norm of $A$, that is $\Vert\cdot\Vert_A$
with $\Vert x\Vert_A=\Vert x\Vert+\Vert Ax\Vert$ for $x\in D(A)$ is
equivalent to $\Vert Ax\Vert$.

In fact, the graph norm $\Vert\cdot\Vert_A$ is equivalent to any of the
expressions $\Vert x\Vert_{A,p}=(\Vert x\Vert^p+\Vert Ax\Vert^p)^{1/p}$
where $1\le p<\infty$ or $\Vert x\Vert_{A,\infty}=\max\{\Vert x \Vert,\,
\Vert Ax\Vert\}.$ In the case of a Hilbert space, the case $p=2$ is
interesting since the corresponding expression derives from an inner
product. Hence if in the above we assume that for some $m_0>0,$ $m(x)\ge m_0$
for $\mu-$almost every $x\in\Omega$ then $\|\vert f\vert \|_{\gamma}$ is
equivalent to 
\begin{equation*}
\|f\|_{D(A^{\gamma})}=\|(m(\cdot))^{\gamma}f\|_{L^2(\Omega)},\quad f\in
D(A^{\gamma}).
\end{equation*}

In making this assumption, we do not actually restrict the range of
applications of the results obtained. It suffices to observe that the
equation 
\begin{equation*}
\mathbb{D}_{t}^{\alpha }u(x,t)+Au(x,t)=F(u(x,t)),\quad t\in \lbrack 0,T]
\end{equation*}%
is replaced by 
\begin{equation*}
\mathbb{D}_{t}^{\alpha }u(x,t)+(A+m_{0}I)u(x,t)=F(u(x,t))+m_{0}u(x,t),\quad
t\in \lbrack 0,T],
\end{equation*}%
and the function $G(\xi )=F(\xi )+m_{0}\xi ,\,\xi \in \mathbb{R}$ satisfies
the conditions imposed on $F.$ We are therefore able to handle domains $%
\Omega $ that are not necessarily bounded. By duality, we can also set $%
D(A^{-\gamma })=(D(A^{\gamma }))^{\ast }$ by identifying $(L^{2}(\Omega
))^{\ast }=L^{2}\left( \Omega \right) ,$ and using the so called Gelfand
triple (see e.g. \cite{ATW}). Then $D(A^{-\gamma })$ is a Hilbert space with
the norm 
\begin{equation*}
\Vert h\Vert _{D(A^{-\gamma })}=\Vert (m(\cdot ))^{-\gamma }h\Vert
_{L^{2}(\Omega )}.
\end{equation*}%
Since $D(A^{1/2})=V_{1/2}$, we identify $V_{-1/2}$ with $D(A^{-1/2})$.

\begin{remark}
{\em 
Throughout the remainder of the paper, we shall assume that $A$ satisfies
the above assumptions, i.e., $A$ \textbf{need not} possess a compact
resolvent over $\mathcal{H}$. In addition to well-known classical examples,
one may also consider examples of differential operators whose spectrum on a
bounded domain or bounded interval is not necessarily discrete, that is,
eigenvalues of infinite multiplicity, continuous spectrum, and eigenvalues
embedded in the continuous spectrum may be present. Explicit constructions
related to Schr\"{o}dinger operators with prescribed (non-negative) essential~spectrum can
be found for instance in \cite{BeKh} (and references
therein).}
\end{remark}

The Mittag-Leffler function (see e.g. \cite{Go-Lu-Ma99,Po99,SKM}) is defined
as follows: 
\begin{equation*}
E_{\alpha ,\beta }(z):=\sum_{n=0}^{\infty }\frac{z^{n}}{\Gamma (\alpha
n+\beta )}=\frac{1}{2\pi i}\int_{Ha}e^{\mu }\frac{\mu ^{\alpha -\beta }}{\mu
^{\alpha }-z}d\mu ,\quad \alpha >0,\;\beta \in {\mathbb{C}},\quad z\in {%
\mathbb{C}},
\end{equation*}%
where $\Gamma $ is the usual Gamma function and $Ha$ is a Hankel path, i.e.
a contour which starts and ends at $-\infty $ and encircles the disc $|\mu
|\leq |z|^{\frac{1}{\alpha }}$ counterclockwise. It is well-known that $%
E_{\alpha ,\beta }(z)$ is an entire function. The following estimate of the
Mittag-Leffler function will be useful. Let $0<\alpha <2$, $\beta \in {%
\mathbb{R}}$ and $\mu $ be such that $\frac{\alpha \pi }{2}<\mu <\min \{\pi
,\alpha \pi \}$. Then there is a constant $C=C(\alpha ,\beta ,\mu )>0$ such
that 
\begin{equation}
|E_{\alpha ,\beta }(z)|\leq \frac{C}{1+|z|},\;\;\;\mu \leq |\mbox{arg}%
(z)|\leq \pi .  \label{Est-MLF}
\end{equation}%
We have that for $\alpha >0$, $\lambda >0$, $t>0$ and $m\in {\mathbb{N}}$, 
\begin{equation}
\frac{d^{m}}{dt^{m}}\left[ E_{\alpha ,1}(-\lambda t^{\alpha })\right]
=-\lambda t^{\alpha -m}E_{\alpha ,\alpha -m+1}(-\lambda t^{\alpha }),
\label{Est-MLF2}
\end{equation}%
and 
\begin{equation}
\frac{d}{dt}\left[ tE_{\alpha ,2}(-\lambda t^{\alpha })\right] =E_{\alpha
,1}(-\lambda t^{\alpha }),  \label{2}
\end{equation}%
and 
\begin{equation}
\frac{d}{dt}\left[ t^{\alpha -1}E_{\alpha ,\alpha }(-\lambda t^{\alpha })%
\right] =t^{\alpha -2}E_{\alpha ,\alpha -1}(-\lambda t^{\alpha }).  \label{3}
\end{equation}%
The proof of \eqref{Est-MLF} is contained in \cite[Theorem 1.6, page 35]%
{Po99}.  The proofs of \eqref{Est-MLF2}-\eqref{3} are contained in \cite[%
Section 1.2.3 Formula (1.83)]{Po99}. For more details on the Mittag-Leffler
functions we refer the reader to \cite{Go-Ma-Ki-Ro14, Kilbas, Po99,SKM} and
the references therein.

In what follows, we will also exploit the following estimates. 
They follow from \eqref{Est-MLF} and \cite[Lemma 3.3]{KW}.

\begin{lemma}
\label{lem-INE} Let $1<\alpha<2$ and $\alpha^{\prime }>0$. Then the
following assertions hold.

\begin{enumerate}
\item Let $0\leq \beta \le 1$, $0<\gamma <\alpha $ and $\lambda >0$. Then
there is a constant $C>0$ such that for every $t>0$, 
\begin{equation}
|\lambda ^{\beta }t^{\gamma }E_{\alpha ,\alpha ^{\prime }}(-\lambda
t^{\alpha })|\leq Ct^{\gamma -\alpha \beta }.  \label{IN-L1}
\end{equation}

\item Let $0\leq \gamma \leq 1$ and $\lambda >0$. Then there is a constant $%
C>0$ such that for every $t>0$, 
\begin{equation}
|\lambda ^{1-\gamma }t^{\alpha -2}E_{\alpha ,\alpha ^{\prime }}(-\lambda
t^{\alpha })|\leq Ct^{\alpha \gamma -2}.  \label{IN-L2}
\end{equation}
\end{enumerate}
\end{lemma}

\section{Basic energy estimates for the inhomogeneous linear problem}

\label{lin-pro}

Throughout the remainder of the article, without any mention, by a.e. on $%
\Omega $, we shall mean $\mu $-a.e. on $\Omega $. Let $1<\alpha <2$ and
consider the following fractional in time wave equation 
\begin{equation}
\begin{cases}
\mathbb{D}_{t}^{\alpha }u\left( x,t\right) +Au\left( x,t\right) =f\left(
x,t\right) \;\; & \mbox{ in }\;\Omega \times (0,T), \\ 
u(\cdot ,0)=u_{0},\;\;\partial _{t}u(\cdot ,0)=u_{1} & \mbox{ in }\;\Omega ,%
\end{cases}
\label{EQ-LI}
\end{equation}%
where $u_{0},u_{1}$ and $f$ are given functions. Our notion of weak
solutions to the system \eqref{EQ-LI} is as follows.

\begin{definition}
\label{def-weak} Let $0\leq \gamma\leq 1/2$, $1/2\leq \widetilde{\gamma}\leq 1$ and recall that $V_{\gamma }:=D\left( A^{\gamma
}\right) $. A function $u$ is said to be a weak solution of \eqref{EQ-LI} on 
$(0,T)$, for some $T>0$, if the following assertions hold.

\begin{itemize}
\item Regularity:%
\begin{equation}
u\in C\left( \left[ 0,T\right] ;V_{\widetilde{\gamma }}\right),\text{ }\partial_tu\in L^2(0,T;L^2(\Omega)) ,\text{ }%
\mathbb{D}_{t}^{\alpha }u\in C((0,T];V_{-\gamma }).  \label{reg-lin}
\end{equation}

\item Initial conditions: 
\begin{equation}
\lim_{t\rightarrow 0^{+}}\left\Vert u(\cdot ,t)-u_{0}\right\Vert _{V_{\sigma
}}=0,\lim_{t\rightarrow 0^{+}}\left\Vert \partial _{t}u(\cdot
,t)-u_{1}\right\Vert _{V_{-\beta }}=0,  \label{ini}
\end{equation}%
for some 
\begin{equation*}
\min \left\{ \widetilde{\gamma },\alpha ^{-1}\right\} >\sigma \geq 0\text{
and }\beta >\max \left\{ 0,\alpha ^{-1}-\widetilde{\gamma }\right\} .
\end{equation*}

\item Variational identity: for every $\varphi \in V_{1/2}$ and for a.e. $%
t\in (0,T)$, we have 
\begin{equation}
\langle \mathbb{D}_{t}^{\alpha }u(\cdot ,t),\varphi \rangle +\mathcal{E}%
_{A}(u(\cdot ,t),\varphi )=\left\langle f\left( \cdot ,t\right) ,\varphi
\right\rangle .  \label{Var-I}
\end{equation}%
The bracket $\left\langle \cdot ,\cdot \right\rangle $ denotes the
duality pairing between $V_{-1/2}$ and $V_{1/2},$ respectively.
\end{itemize}
\end{definition}
\begin{remark}
{\em Observe that the parameters $\sigma$ and $\beta$ in the initial conditions allow a wider range of  that those in \cite[Definition 2.3]{AGKW}.}
\end{remark}

We first prove well-posedness for the linear problem (recall that $1<\alpha
<2$).

\begin{theorem}
\label{theo-weak} Let $1/2\leq \widetilde{\gamma }\leq \alpha ^{-1},$ $0\leq
\gamma \leq 1/2$ be such that $0\leq \gamma +\widetilde{\gamma }\leq 1$. Let $%
u_{0}\in V_{\widetilde{\gamma }}$ and $u_{1}\in L^{2}\left( \Omega \right) $%
. Next, assume that%
\begin{equation*}
f\in C\left( (0,T];V_{-\gamma }\right) \cap L^{q}(0,T;L^{2}(\Omega )),\text{ 
}q:=\frac{p}{p-1}\text{,}
\end{equation*}%
where $p\in \left[ 1,\infty \right) $ (and $q=1$ if $p=\infty$) is such that 
\begin{equation*}
\alpha -1-\alpha \widetilde{\gamma }+\frac{1}{p}>0\text{ and }\alpha -2+%
\frac{1}{p}>0.
\end{equation*}%
Then the system \eqref{EQ-LI} has a unique weak solution given by%
\begin{align}
u(\cdot ,t)=& E_{\alpha ,1}(-m(\cdot )t^{\alpha })u_{0}(\cdot )+tE_{\alpha
,2}(-m(\cdot )t^{\alpha })u_{1}(\cdot ) \notag \\
& +\int_{0}^{t}f(\cdot ,\tau )(t-\tau )^{\alpha -1}E_{\alpha ,\alpha
}(-m(\cdot )(t-\tau )^{\alpha })d\tau .   \label{sol-spec}
\end{align}%
Moreover, there is a constant $C>0$ such that for all $t\in (0,T]$,%
\begin{align}
\Vert u(\cdot ,t)\Vert _{V_{\widetilde{\gamma }}}& \leq C\left( \Vert
u_{0}\Vert _{V_{\widetilde{\gamma }}}+t^{1-\alpha \widetilde{\gamma }}\Vert
u_{1}\Vert _{L^{2}(\Omega )}+t^{\alpha -1-\alpha \widetilde{\gamma }+\frac{1%
}{p}}\Vert f\Vert _{L^{q}((0,T);L^{2}(\Omega ))}\right) ,  \label{EST-1} \\
\Vert \partial _{t}u(\cdot ,t)\Vert _{L^{2}(\Omega )}& \leq C\left(
t^{\alpha \widetilde{\gamma }-1}\Vert u_{0}\Vert _{V_{\widetilde{\gamma }%
}}+\Vert u_{1}\Vert _{L^{2}(\Omega )}+t^{\frac{1}{p}+\alpha -2}\Vert f\Vert
_{L^{q}((0,T);L^{2}(\Omega ))}\right)  \label{EST-1-2}
\end{align}%
and%
\begin{align}
\Vert \mathbb{D}_{t}^{\alpha }u(\cdot ,t)\Vert _{V_{-\gamma }}& \leq
Ct^{\alpha \left( \gamma +\widetilde{\gamma }-1\right) }\Vert u_{0}\Vert
_{V_{\widetilde{\gamma }}}+Ct^{1-\alpha +\alpha \gamma }\Vert u_{1}\Vert
_{L^{2}(\Omega )} \notag \\
& +Ct^{\alpha \gamma -1+1/p}\Vert f\Vert _{L^{q}((0,T);L^{2}(\Omega
))}+\left\Vert f\left( \cdot ,t\right) \right\Vert _{V_{-\gamma }}.   \label{EST-1-3}
\end{align}
\end{theorem}

\begin{proof}
We prove the result in several steps.

\textbf{Step 1}. We show that $u\in C([0,T];V_{\widetilde{\gamma }})$. Let $%
t\in \lbrack 0,T]$ and set 
\begin{align*}
S_{1}(t)u_{0}(\xi ):=&E_{\alpha ,1}(-m(\xi )t^{\alpha })u_{0}(\xi )\\
S_{2}(t)u_{1}(\xi ):=&tE_{\alpha ,2}(-m(\xi )t^{\alpha })u_{1}(\xi )\\ S_3(t)x:=&t^{\alpha-1}E_{\alpha,\alpha}(-m(\xi )t^{\alpha }),
\end{align*}%
and 
\begin{equation*}
(S_{3}\ast f(\xi ,\cdot ))(t):=\int_{0}^{t}f(\cdot ,\tau )(t-\tau )^{\alpha
-1}E_{\alpha ,\alpha }((-m(\cdot )(t-\tau )^{\alpha })d\tau
\end{equation*}%
so that 
\begin{equation*}
u(\xi ,t)=S_{1}(t)u_{0}(\xi )+S_{2}(t)u_{1}(\xi )+(S_{3}\ast f(\xi ,\cdot
))(t),\quad t\in \lbrack 0,T],\,\,\xi \in \Omega .
\end{equation*}%
Using \eqref{Est-MLF} we get that there is a constant $C>0$ such that for
every $t\in \lbrack 0,T]$, 
\begin{align*}
\Vert S_{1}(t)u_{0}(\cdot )\Vert _{V_{\widetilde{\gamma }}}^{2} & \leq
C^{2}\int_{\Omega }|(m(\xi ))^{\widetilde{\gamma }}u_{0}(\xi )|^{2}d\mu
=C^{2}\Vert u_{0}\Vert _{V_{\widetilde{\gamma }}}^{2},
\end{align*}%
which implies that 
\begin{equation}
\Vert S_{1}(t)u_{0}(\cdot )\Vert _{V_{\widetilde{\gamma }}}\leq C\Vert
u_{0}\Vert _{V_{\widetilde{\gamma }}}.  \label{S1}
\end{equation}%
Taking into account \eqref{IN-L1}, we obtain that there is a constant $C>0$
such that for every $t\in \lbrack 0,T]$, 
\begin{align*}
\Vert S_{2}(t)u_{1}(\cdot )\Vert _{V_{\widetilde{\gamma }}}^{2} & \leq
\int_{\Omega }C^{2}t^{2\left( 1-\alpha \widetilde{\gamma }\right)
}|u_{1}(\xi )|^{2}d\mu =C^{2}t^{2\left( 1-\alpha \widetilde{\gamma }\right)
}\Vert u_{1}\Vert _{L^{2}(\Omega )}^{2},
\end{align*}%
so that%
\begin{equation}
\Vert S_{2}(t)u_{1}(\cdot )\Vert _{V_{\widetilde{\gamma }}}\leq Ct^{\left(
1-\alpha \widetilde{\gamma }\right) }\Vert u_{1}\Vert _{L^{2}(\Omega )}.
\label{S2}
\end{equation}%
Using \eqref{IN-L1} again and the H\"{o}lder inequality, we get that there
is a constant $C>0$ such that for every $t\in \lbrack 0,T]$, 
\begin{align*}
\Vert (S_{3}\ast f(\xi ,\cdot ))(t)\Vert _{V_{\widetilde{\gamma }}} & \leq
\int_{0}^{t}\left( \int_{\Omega }C^{2}\left( t-\tau \right) ^{2\left( \alpha
-1-\alpha \widetilde{\gamma }\right) }|f(\xi ,\tau )|^{2}d\mu \right)
^{1/2}d\tau \\
& \leq C\left( \int_{0}^{t}\left( t-\tau \right) ^{p\left( \alpha -1-\alpha 
\widetilde{\gamma }\right) }d\tau \right) ^{1/p}\left( \int_{0}^{t}\Vert
f(\xi ,\tau )\Vert ^{q}d\tau \right) ^{1/q} \\
& \leq Ct^{\alpha -1-\alpha \widetilde{\gamma }+\frac{1}{p}}\Vert f\Vert
_{L^{q}((0,T);L^{2}(\Omega ))},
\end{align*}%
provided that $p\geq 1$ is such that $\alpha -1-\alpha \widetilde{\gamma }+%
\frac{1}{p}>0.$ It follows that 
\begin{equation}
\Vert (S_{3}\ast f(\xi ,\cdot ))(t)\Vert _{V_{\widetilde{\gamma }}}\leq
Ct^{\alpha -1-\alpha \widetilde{\gamma }+\frac{1}{p}}\Vert f\Vert
_{L^{q}((0,T);L^{2}(\Omega ))}.  \label{S3}
\end{equation}%
Let us see that $u\in C([0,T];V_{\widetilde{\gamma }})$. We start proving
that $S_{1}(t)u_{0}(\xi )$ is continuous on $[0,T]$. Indeed, for $h>0$, 
\begin{align*}
\Vert S_{1}(t+h)u_{0}(\cdot )-S_{1}(t)u_{0}(\cdot )\Vert _{V_{\widetilde{%
\gamma }}}^2  =&\int_{\Omega }|(E_{\alpha ,1}(-m(\xi )(t+h)^{\alpha
})\\
&-E_{\alpha ,1}(-m(\xi )t^{\alpha })) (m(\xi ))^{\widetilde{\gamma }%
}u_{0}(\xi )|^{2}d\mu.
\end{align*}%
Since $t\mapsto E_{\alpha ,1}(-m(\xi )t^{\alpha })u_{0}(\xi )$ is continuous
for every $t\in \lbrack 0,T]$,%
\begin{equation*}
|(E_{\alpha ,1}(-m(\xi )(t+h)^{\alpha })-E_{\alpha ,1}(-m(\xi )t^{\alpha
}))(m(\xi ))^{\widetilde{\gamma }}u_{0}(\xi )|^{2}\leq C^{2}|(m(\xi ))^{%
\widetilde{\gamma }}u_{0}(\xi )|^{2}
\end{equation*}%
and $u_{0}\in V_{\widetilde{\gamma }}$, it follows from Dominated
Convergence Theorem that $S_{1}(t)u_{0}(\xi )$ is continuous for all $t\in
\lbrack 0,T]$. The proof of each of the facts that $S_{2}(t)u_{1}(\xi )$ and 
$(S_{3}\ast f(\xi ,\cdot ))(t)$ are continuous on $[0,T]$ is similar. We
have shown that $u\in C([0,T];V_{\widetilde{\gamma }})$. It also follows
from the estimates \eqref{S1}, \eqref{S2} and \eqref{S3} that there is a
constant $C>0$ such that for every $t\in \lbrack 0,T]$, 
\begin{equation}
\Vert u(\cdot ,t)\Vert _{V_{\widetilde{\gamma }}}\leq C_{1}\left( \Vert
u_{0}\Vert _{V_{\widetilde{\gamma }}}+t^{1-\alpha \widetilde{\gamma }}\Vert
u_{1}\Vert _{L^{2}(\Omega )}+t^{\alpha -1-\alpha \widetilde{\gamma }+\frac{1%
}{p}}\Vert f\Vert _{L^{q}((0,T);L^{2}(\Omega ))}\right) .  \label{B-1}
\end{equation}

\textbf{Step 2}. Next, we show that  $\partial _{t}u\in
L^{2}(0,T;L^{2}(\Omega ))$. We notice that a simple calculation (using \eqref{2}) gives for
a.e. $t\in (0,T)$, 
\begin{align}
\partial _{t}u(\cdot ,t)=& -m(\cdot )t^{\alpha -1}E_{\alpha ,1}(-m(\cdot
)t^{\alpha })u_{0}(\cdot )+E_{\alpha ,1}(-m(\cdot )t^{\alpha })u_{1}(\cdot )
\label{D1-1} \\
& +\int_{0}^{t}f(\cdot ,\tau )(t-\tau )^{\alpha -2}E_{\alpha ,\alpha
-1}(-m(\cdot )(t-\tau )^{\alpha })d\tau .  \notag \\
=:& S_{1}^{\prime }(t)u_{0}(\cdot )+S_{2}^{\prime }(t)u_{1}(\cdot
)+(S_{3}^{\prime }\ast f(\xi ,\cdot ))(t).  \notag
\end{align}%
Proceeding as in \textbf{Step 1}, on account of (\ref{IN-L2}), we get the
following estimates: 
\begin{align*}
\Vert S_{1}^{\prime }(t)u_{0}(\cdot )\Vert _{L^{2}(\Omega )}^{2} & \leq
C^{2}t^{2\left( \alpha \widetilde{\gamma }-1\right) }\int_{\Omega }|(m(\xi
))^{\widetilde{\gamma }}u_{0}(\xi )|^{2}d\mu \\
& \leq C^{2}t^{2\left( \alpha \widetilde{\gamma }-1\right) }\Vert u_{0}\Vert
_{V_{\widetilde{\gamma }}}^{2},
\end{align*}%
as well as%
\begin{align*}
\Vert S_{2}^{\prime }(t)u_{1}(\cdot )\Vert _{L^{2}(\Omega )}^{2}&
=\int_{\Omega }|E_{\alpha ,1}(-m(\xi )t^{\alpha })u_{1}(\xi )|^{2}d\mu \\
& \leq C^{2}\Vert u_{1}\Vert _{L^{2}(\Omega )}^{2},
\end{align*}%
and%
\begin{align*}
\Vert (S_{3}^{\prime }\ast f(\xi ,\cdot ))(t)\Vert _{L^{2}(\Omega )} & \leq
C\int_{0}^{t}(t-\tau )^{\alpha -2}\Vert f(\cdot ,\tau )\Vert _{L^{2}(\Omega
)}d\tau \\
& \leq Ct^{\frac{1}{p}+\alpha -2}\Vert f\Vert _{L^{q}((0,T);L^{2}(\Omega ))},
\end{align*}%
for as long as $1/p+\alpha -2>0$. In summary, we get%
\begin{equation*}
\Vert S_{1}^{\prime }(t)u_{0}(\cdot )\Vert _{L^{2}(\Omega )}\leq Ct^{\alpha 
\widetilde{\gamma }-1}\Vert u_{0}\Vert _{V_{\widetilde{\gamma }}}\;%
\mbox{
and }\;\Vert S_{2}^{\prime }(t)u_{1}(\cdot )\Vert _{L^{2}(\Omega )}\leq
C\Vert u_{1}\Vert _{L^{2}(\Omega )},
\end{equation*}%
and%
\begin{equation}
\Vert (S_{3}^{\prime }\ast f(\xi ,\cdot ))(t)\Vert _{L^{2}(\Omega )}\leq Ct^{%
\frac{1}{p}+\alpha -2}\Vert f\Vert _{L^{q}((0,T);L^{2}(\Omega ))},
\label{B1-1}
\end{equation}%
for all $t\in (0,T]$. It thus follows from these estimates that there is a
constant $C>0$ such that%
\begin{equation}
\Vert \partial _{t}u(\cdot ,t)\Vert _{L^{2}(\Omega )}\leq C_{1}\left(
t^{\alpha \widetilde{\gamma }-1}\Vert u_{0}\Vert _{V_{\widetilde{\gamma }%
}}+\Vert u_{1}\Vert _{L^{2}(\Omega)}+t^{\frac{1}{p}+\alpha -2}\Vert
f\Vert _{L^{q}((0,T);L^{2}(\Omega ))}\right) .  \label{B-2}
\end{equation}%
In addition, \eqref{EST-1}-(\ref{EST-1-2}) follows from \eqref{B-1} and %
\eqref{B-2}. Also, observe that $\frac{1}{2\alpha}<\frac{1}{2}\leq \tilde{\gamma}$ implies that $\partial _{t}u\in
L^{2}(0,T;L^{2}(\Omega ))$.

\textbf{Step 3}. Next, we prove that $\mathbb{D}_{t}^{\alpha }u\in
C((0,T];V_{-\gamma })$. It follows from \eqref{sol-spec} that 
\begin{align}
\mathbb{D}_{t}^{\alpha }u(\cdot ,t)=& -m(\cdot )E_{\alpha ,1}(-m(\cdot
)t^{\alpha })u_{0}(\cdot )-m(\cdot )tE_{\alpha ,2}(-m(\cdot )t^{\alpha
})u_{1}(\cdot )  \label{dt-al} \\
& -m(\cdot )\int_{0}^{t}f(\cdot ,\tau )(t-\tau )^{\alpha -1}E_{\alpha
,\alpha }(-m(\cdot )(t-\tau )^{\alpha })\;d\tau +f(\cdot ,t)  \notag \\
& =-Au(\cdot ,t)+f(\cdot ,t).  \notag
\end{align}%
Using \eqref{Est-MLF} and \eqref{IN-L1}, we get the following estimates, for
every $0\leq \gamma \leq \frac{1}{2},$%
\begin{align}  \label{D1}
\left\Vert m(\cdot )E_{\alpha ,1}(-m(\cdot )t^{\alpha })u_{0}(\cdot
)\right\Vert _{V_{-\gamma }}^{2} & \leq C^{2}\int_{\Omega }t^{2\alpha \left(
\gamma +\widetilde{\gamma }-1\right) }|(m(\xi ))^{\widetilde{\gamma }%
}u_{0}(\xi )|^{2}d\mu  \notag \\
& \leq C^{2}t^{2\alpha \left( \gamma +\widetilde{\gamma }-1\right) }\Vert
u_{0}\Vert _{V_{\widetilde{\gamma }}}^{2},
\end{align}%
provided that $0\leq \gamma +\widetilde{\gamma }\leq 1$ and%
\begin{align}  \label{D2}
\left\Vert m(\cdot )tE_{\alpha ,2}(-m(\cdot )t^{\alpha })u_{1}(\cdot
)\right\Vert _{V_{-\gamma }}^{2} & \leq C^{2}t^{2\left( 1-\alpha +\alpha
\gamma \right) }\Vert u_{1}\Vert _{L^{2}(\Omega )}^{2}.
\end{align}

Similarly on account of (\ref{IN-L1}), we deduce%
\begin{align*}
& \left\Vert m(\cdot )\int_{0}^{t}f(\cdot ,\tau )(t-\tau )^{\alpha
-1}E_{\alpha ,\alpha }(-m(\cdot )(t-\tau )^{\alpha })d\tau \right\Vert
_{V_{-\gamma }} \\
& \leq \int_{0}^{t}\left( \int_{\Omega }C^{2}(t-\tau )^{2(\alpha -1-\alpha
(1-\gamma ))}|f(\cdot ,\tau )|^{2}d\mu \right) ^{1/2}d\tau \\
& \leq Ct^{\alpha \gamma -1+1/p}\Vert f\Vert _{L^{q}((0,T);L^{2}(\Omega ))},
\end{align*}%
provided that $1\leq p<\frac{1}{1-\alpha \gamma }\Leftrightarrow \alpha
\gamma -1+1/p>0.$ Consequently, we obtain%
\begin{equation}
\left\Vert m(\cdot )\int_{0}^{t}f(\cdot ,\tau )(t-\tau )^{\alpha
-1}E_{\alpha ,\alpha }(-m(\cdot )(t-\tau )^{\alpha })d\tau \right\Vert
_{V_{-\gamma }}\leq Ct^{\alpha \gamma -1+1/p}\Vert f\Vert
_{L^{q}((0,T);L^{2}(\Omega ))}.  \label{D3}
\end{equation}%
Since the functions $t\mapsto m(\xi )E_{\alpha ,1}(-m(\xi )t^{\alpha
})u_{0}(\xi )$, $t\mapsto m(\xi )tE_{\alpha ,2}(-m(\xi )t^{\alpha
})u_{1}(\xi )$ and 
\begin{equation*}
t\mapsto m(\xi )\int_{0}^{t}\tau ^{\alpha -1}E_{\alpha ,\alpha }(-m(\xi
)\tau ^{\alpha })f(\xi ,\cdot )d\tau
\end{equation*}%
in \eqref{dt-al} are continuous for every $t\in \lbrack 0,T]$, we can
conclude (using the Dominated Convergence Theorem) that $\mathbb{D}%
_{t}^{\alpha }u\in C((0,T];V_{-\gamma })$ since $f\in C\left(
(0,T];V_{-\gamma }\right) $. The estimate \eqref{EST-1-3} then follows from %
\eqref{D1}, \eqref{D2}, \eqref{D3} on account of (\ref{dt-al}).

\textbf{Step 4}. First, note that $u\left( t\right) \in V_{\widetilde{\gamma 
}}\subseteq V_{1/2},$ for $1/2\leq \widetilde{\gamma}\leq \alpha^{-1},$
for as long as $\gamma \in \left[ 0,1/2\right] $ and $\widetilde{\gamma }%
+\gamma \in \left[ 0,1\right] $. By definition, this says that $(m(\cdot
))^{1/2}u\left( t\right) \in L^{2}(\Omega )$. Then 
\begin{equation*}
(m(\cdot ))^{-1/2}Au\left( t\right) =(m(\cdot ))^{-1/2}(m(\cdot )u\left(
t\right) )=(m(\cdot ))^{1/2}u\left( t\right) \in L^{2}(\Omega),
\end{equation*}%
which implies that $Au(\cdot ,t)\in V_{-1/2}$, for all $t\in \left[ 0,T%
\right] .$ Since $\mathbb{D}_{t}^{\alpha }u(\cdot ,t)\in V_{-\gamma
}\subseteq V_{-1/2},$ and $f(\cdot ,t)\in L^{2}(\Omega ),$ for a.e. $t\in
(0,T)$, then taking the duality product in (\ref{dt-al}) we get the
variational identity \eqref{Var-I}.

\textbf{Step 5}. 
By \eqref{sol-spec}, for $t<1$ small enough we have%
\begin{equation}
u\left( \cdot ,t\right) -u_{0}(\cdot )=E_{\alpha ,1}(-m(\cdot )t^{\alpha
})u_{0}(\cdot )-u_{0}(\cdot )+S_{2}\left( t\right) u_{1}(\cdot )+(S_{3}\ast
f(\xi ,\cdot ))\left( t\right) .  \label{diff-u0}
\end{equation}%
The estimates \eqref{S2} and \eqref{S3} yield 
\begin{equation}
\left\Vert (S_{3}\ast f(\xi ,\cdot ))\left( t\right) \right\Vert _{V_{\sigma
}}\leq \left\Vert (S_{3}\ast f(\xi ,\cdot ))\left( t\right) \right\Vert _{V_{%
\widetilde{\gamma }}}\leq Ct^{\alpha -1-\alpha \widetilde{\gamma }+\frac{1}{p%
}}\Vert f\Vert _{L^{q}((0,T);L^{2}(\Omega ))}\rightarrow 0,  \label{diff-s3}
\end{equation}%
as $t\rightarrow 0^{+}$ since $V_{\widetilde{\gamma }}\hookrightarrow
V_{\sigma }$ for $0\leq \sigma <\widetilde{\gamma }$ and $\alpha -1-\alpha 
\widetilde{\gamma }+\frac{1}{p}>0$. Since $\sigma \alpha <1$, it also
follows that 
\begin{equation*}
\left\Vert S_{2}\left( t\right) u_{1}(\cdot )\right\Vert _{V_{\sigma }}\leq
\left( \int_{\Omega }|(m(\xi ))^{\sigma }tE_{\alpha ,2}(-m(\xi )t^{\alpha
})u_{1}(\xi )|^{2}d\mu \right) ^{1/2}\leq Ct^{1-\alpha \sigma }\Vert
u_{1}\Vert _{L^{2}(\Omega )}\rightarrow 0,
\end{equation*}%
as $t\rightarrow 0^{+}$. Finally, since $E_{\alpha ,1}(m(\xi )t^{\alpha
})u_{0}(\xi )\rightarrow u_{0}(\xi )$ as $t\rightarrow 0^{+}$, and 
\begin{align*}
& |(m(\xi ))^{\sigma }(E_{\alpha ,1}(-m(\xi )t^{\alpha })u_{0}(\xi
)-u_{0}(\xi ))|^{2} \leq 4C^{2}m_{0}^{2(\sigma -\widetilde{\gamma })}|(m(\xi
))^{\widetilde{\gamma }}u_{0}(\xi ))|^{2}\in L^{1}(\Omega ),
\end{align*}%
it follows from the Dominated Convergence Theorem that
\begin{equation*}
\left\Vert E_{\alpha ,1}(-m(\cdot )t^{\alpha })u_{0}(\cdot )-u_{0}(\cdot
)\right\Vert _{V_{\sigma }}=\left( \int_{\Omega }|(m(\xi ))^{\sigma
}(E_{\alpha ,1}(-m(\xi )t^{\alpha })u_{0}(\xi )-u_{0}(\xi ))|^{2}d\mu
\right) ^{1/2}\rightarrow 0,
\end{equation*}%
as $t\rightarrow 0^{+}$. Thus, the first statement of (\ref{ini}) is
verified. On account of (\ref{D1-1}), we have%
\begin{equation}
\partial _{t}u\left( \cdot ,t\right) -u_{1}(\cdot )=S_{1}^{\prime
}(t)u_{0}(\cdot )+E_{\alpha ,1}(-m(\cdot )t^{\alpha })u_{1}(\cdot
)-u_{1}(\cdot )+(S_{3}^{\prime }\ast f(\xi ,\cdot ))(t)  \label{diff-u1}
\end{equation}%
and once again, by (\ref{B1-1}), it is easy to see, on account of the
condition $1/p+\alpha -2>0,$ that%
\begin{equation}
\left\Vert (S_{3}\ast f(\xi ,\cdot ))^{\prime }(t)\right\Vert _{L^{2}\left(
\Omega \right) }\rightarrow 0\text{ as }t\rightarrow 0^{+}.
\label{diff-u1-1}
\end{equation}%
A simple calculation also yields from (\ref{IN-L1}), that%
\begin{align}
\left\Vert S_{1}^{\prime }(t)u_{0}(\cdot )\right\Vert _{V_{-\beta }} & \leq
Ct^{\alpha -1-\alpha \left( 1-\widetilde{\gamma }-\beta \right) }\left\Vert
u_{0}\right\Vert _{V_{\widetilde{\gamma }}}  \notag \\
& =Ct^{\alpha \widetilde{\gamma }+\alpha \beta -1}\left\Vert
u_{0}\right\Vert _{V_{\widetilde{\gamma }}}\rightarrow 0,  \label{diff-u1-2}
\end{align}%
as $t\rightarrow 0^{+}$, owing to the condition $\alpha \left( \widetilde{%
\gamma }+\beta \right) >1$. Finally, since $E_{\alpha ,1}(m(\xi )t^{\alpha
})u_{1}(\xi )\rightarrow u_{1}(\xi )$ as $t\rightarrow 0^{+}$, and 
\begin{equation*}
|(m(\xi ))^{-\beta }(E_{\alpha ,1}(-m(\xi )t^{\alpha })u_{1}(\xi )-u_{1}(\xi
))|^{2}\leq 4C^{2}m_{0}^{-2\beta }|(m(\xi ))^{\beta }u_{1}(\xi ))|^{2}\in
L^{1}(\Omega ),
\end{equation*}%
it follows from Dominated Convergence Theorem that in the limit as $%
t\rightarrow 0^{+}$,%
\begin{equation*}
\left\Vert E_{\alpha ,1}(-m(\cdot )t^{\alpha })u_{1}(\cdot )-u_{1}(\cdot
)\right\Vert _{V_{-\beta }}=\left( \int_{\Omega }|(m(\xi ))^{-\beta
}(E_{\alpha ,1}(-m(\xi )t^{\alpha })u_{1}(\xi )-u_{1}(\xi ))|^{2}d\mu
\right) ^{1/2}\rightarrow 0.
\end{equation*}%
Collecting the preceding three estimates yields the second statement of (\ref%
{ini}).

\textbf{Step 6}. Finally we show uniqueness. Let $u,v$ be any two weak
solutions of \eqref{EQ-LI} with the same initial data $u_{0},u_{1}$ and
source $f=f\left( x,t\right) $ and then set $w=u-v$. Then $w$ is a weak
solution of the system%
\begin{equation}
\begin{cases}
\mathbb{D}_{t}^{\alpha }w=-Aw,\;\; & \mbox{ in }\;\Omega \times (0,T), \\ 
w(\cdot ,0)=0,\;\partial _{t}w(\cdot ,0)=0\; & \mbox{ in }\;\Omega .%
\end{cases}
\label{diff}
\end{equation}%
This system has only a unique solution that is the null solution $w=0$ (see,
e.g., \cite{Ba01,KW}). Hence $u=v$. 
The proof of the theorem is finished.\emph{\ }
\end{proof}

\begin{remark}
{\em
Under the conditions of the previous theorem,  we have that if $\tilde{\gamma}\geq \frac{1}{\alpha}$, then  $u\in C^{1}([0,T];L^2(\Omega))$.   On the other hand,  if $\theta\in (\frac{1-\alpha\tilde{\gamma}}{\alpha},1-\tilde{\gamma})$, then $\partial_tu\in C([0,T];V_{-\theta}))$, compare with \cite[Theorem 1.1]{Lo-Sf21}.}
\end{remark}

\begin{remark}
{\em
An interesting class of solutions is that of $V_{1/2}\times L^{2}\left( \Omega
\right) $-energy solution which is obtained using the values $\gamma =%
\widetilde{\gamma }=1/2$ previously. In that case, the estimates (\ref{EST-1}%
)-(\ref{EST-1-3}) read as follows:%
\begin{align}
\Vert u(\cdot ,t)\Vert _{V_{1/2}}& \leq C\left( \Vert u_{0}\Vert
_{V_{1/2}}+t^{1-\frac{\alpha }{2}}\Vert u_{1}\Vert _{L^{2}(\Omega )}+t^{%
\frac{\alpha }{2}-1+\frac{1}{p}}\Vert f\Vert _{L^{q}((0,T);L^{2}(\Omega
))}\right) ,  \label{p1} \\
\Vert \partial _{t}u(\cdot ,t)\Vert _{L^{2}(\Omega )}& \leq C\left( t^{\frac{%
\alpha }{2}-1}\Vert u_{0}\Vert _{V_{1/2}}+\Vert u_{1}\Vert _{L^{2}(\Omega
)}+t^{\frac{1}{p}+\alpha -2}\Vert f\Vert _{L^{q}((0,T);L^{2}(\Omega
))}\right)  \label{p2}
\end{align}%
and%
\begin{align}
\Vert \mathbb{D}_{t}^{\alpha }u(\cdot ,t)\Vert _{V_{-1/2}}& \leq C\Vert
u_{0}\Vert _{V_{1/2}}+Ct^{1-\frac{\alpha }{2}}\Vert u_{1}\Vert
_{L^{2}(\Omega )}  \label{p3} \\
& +Ct^{\frac{\alpha }{2}-1+1/p}\Vert f\Vert _{L^{q}((0,T);L^{2}(\Omega
))}+\left\Vert f\left( \cdot ,t\right) \right\Vert _{V_{-1/2}}.  \notag
\end{align}%
Also we can show that, if $2/p+\alpha -3>0,$%
\begin{equation*}
g_{2-\alpha }\ast \left\Vert \partial _{t}u\right\Vert _{L^{2}\left( \Omega
\right) }^{2}\left( t\right) \leq Ct^{\frac{2}{p}+\alpha -3}\Vert f\Vert
_{L^{q}((0,T);L^{2}(\Omega ))}^{2}+Cg_{2-\alpha }\ast \Vert u_{1}\Vert
_{L^{2}\left( \Omega \right) }^{2}+Cg_{2-\alpha }\ast t^{\alpha -2}\Vert
u_{0}\Vert _{V_{1/2}}^{2}.
\end{equation*}
}
\end{remark}

\section{Weak energy solutions for the semilinear problem}

Consider the semilinear problem 
\begin{equation}
\begin{cases}
\mathbb{D}_{t}^{\alpha }u\left( x,t\right) +Au\left( x,t\right) =f(u\left(
x,t\right) ), & \mbox{ in }\;\Omega \times \left( 0,T\right) , \\ 
u(\cdot ,0)=u_{0},\;\;\partial _{t}u(\cdot ,0)=u_{1} & \text{ in }\Omega.%
\end{cases}
\label{EQ-NL}
\end{equation}%
Fix $0<T^{\star}\leq T$.  To solve this problem, we introduce the Banach space%
\begin{equation*}
X_{s,T^{\star }}=\left\{ u\in C\left( \left[ 0,T^{\star }\right]
;V_{1/2}\right) \cap C^{1}\left( (0,T^{\star }];L^{2}\left( \Omega \right)
\right) :\left\Vert u\right\Vert _{s,T^{\star }}<\infty \right\} ,
\end{equation*}%
where the norm is defined as (for a properly chosen $s>0,$ depending on the
physical parameters)%
\begin{equation*}
\left\Vert u\right\Vert _{s,T^{\star }}^{2}:=\sup_{t\in \left[ 0,T^{\star }%
\right] }\left( \left\Vert u\left( \cdot ,t\right) \right\Vert
_{V_{1/2}}^{2}+t^{2s}\left\Vert \partial _{t}u\left( \cdot ,t\right)
\right\Vert _{L^{2}\left( \Omega \right) }^{2}+g_{2-\alpha }\ast \left\Vert
\partial _{t}u\right\Vert _{L^{2}\left( \Omega \right) }^{2}\left( t\right)
\right) .
\end{equation*}%
Note that, for the latter, $g_{2-\alpha }\ast \left\Vert
\partial _{t}u\right\Vert _{L^{2}\left( \Omega \right) }^{2}\left( t\right)\to \left\Vert
\partial _{t}u\right\Vert _{L^{2}\left( \Omega \right) }^{2}$ as 
$\alpha\to 2^{-}$,where the limit is also the corresponding kinetic energy associated with a
weak solution in the case of the classical wave equation.  This  is the reason why we refer to this term as kinetic energy.

Next, define the energy  space $\mathcal{B}_{s,T^{\star },R^{\star }}\subset X_{s,T^{\star }},$
where%
\begin{equation*}
\mathcal{B}_{s,T^{\star },R^{\star }}:=\left\{ u\in X_{s,T^{\star
}}:\left\Vert u\right\Vert _{s,T^{\star }}\leq R^{\star }\text{ and }u\left(
0\right) =u_{0},\text{ }\partial _{t}u\left( 0\right) =u_{1}\right\} .
\end{equation*}%
Let $\Phi: \mathcal{B}_{s,T,R}\to \mathcal{B}_{s,T,R}$ defined by 
\begin{align}
\Phi \left( v\left( \cdot \right) \right) \left( t\right) =& E_{\alpha
,1}(-m(\cdot )t^{\alpha })u_{0}(\cdot )+tE_{\alpha ,2}(-m(\cdot )t^{\alpha
})u_{1}(\cdot )  \label{p4} \\
& +\int_{0}^{t}f(v\left( \cdot ,\tau \right) )(t-\tau )^{\alpha -1}E_{\alpha
,\alpha }(-m(\cdot )(t-\tau )^{\alpha })d\tau, \quad v\in \mathcal{B}_{s,T,R}.  \notag
\end{align}%
We say that a function $u$ is \textit{local weak energy solution} to \eqref{EQ-NL} if $\Phi(u(t))=u(t)$.

%
%
%

\begin{enumerate}
\item[(\textbf{Hf2})] We assume that there exists $C_{f}>0$ such that 
\begin{equation}
f(0)=0\quad \mbox{and}\quad |f^{\prime }(\sigma )|\leq C_{f}|\sigma
|^{r-1},\quad r>1,\,\sigma \in \mathbb{R},  \label{embed-supnorm}
\end{equation}%
and 
\begin{equation*}
V_{1/2}\hookrightarrow L^{2r}(\Omega ).
\end{equation*}
\end{enumerate}
Note that \eqref{embed-supnorm} implies that 
\begin{equation*}
|f(\sigma )|\leq C_{f}|\sigma |^{r},\quad \sigma \in \mathbb{R}.
\end{equation*}%

The next theorem is the main result of this section.  Through the Banach fixed point theorem we establish the existence of local weak solutions of \ref{EQ-NL} (on some interval $[0,T]$).  Then we show that the solution can be extended to a larger interval.($[0,T_{\ast}]$)  Finally,  we prove the maximal interval $[0,T_{\max})$  with the property that either $T_{\max}=\infty$ or $T_{\max}<\infty$.  In the last case the norm blows-up in the sense of \eqref%
{t-maximal} below.

\begin{theorem}
\label{main-weak}Assume (\textbf{Hf2}). The problem (\ref{EQ-NL}) has a
unique weak energy solution on $X_{s,T}$, for all $T<T_{\max }$, where
either $T_{\max }=\infty $ or $T_{\max }<\infty $, and in that case, 
\begin{equation}
\lim_{t\rightarrow T_{\max }}\left( \left\Vert u\left( t\right) \right\Vert
_{V_{1/2}}^{2}+t^{2s}\left\Vert \partial _{t}u\left( t\right) \right\Vert
_{L^{2}\left( \Omega \right) }^{2}+g_{2-\alpha }\ast \left\Vert \partial
_{t}u\right\Vert _{L^{2}\left( \Omega \right) }^{2}\left( t\right) \right)
=\infty .  \label{t-maximal}
\end{equation}
holds.
\end{theorem}

\begin{proof}
We will the following:
\begin{itemize}
\item (Contraction argument) $\Phi $ is a contraction from $\mathcal{B}_{s,T^{\star},R^{\star}}\rightarrow \mathcal{%
B}_{s,T^{\star},R^{\star}}$, for some $T^{\star},R^{\star},s>0.$

\item (Extension argument) The solution $u$ can be extended by continuity from $\left[ 0,T^{\ast }%
\right] ,$ for some $T_{\ast }>0,$ to the interval $\left[ 0,T^{\ast
}+\varsigma \right] ,$ for some $\varsigma >0.$

\item (Global vs blow-up) Finally,  we show that that either $T_{\max }=\infty $ or else
if $T_{\max }<\infty $, one has \eqref%
{t-maximal}.

\end{itemize}

\textbf{Step 1 (The contraction argument)}. Since $f$ is continuously
differentiable, we have that $\Phi (u(\cdot ))(t)$ is continuously
differentiable on $(0,T^{\star }]$. We will show that by an appropriate
choice of $T^{\star },R^{\star }>0$, $\Phi :\mathcal{B}_{s,T^{\star
},R^{\star }}\rightarrow \mathcal{B}_{s,T^{\star },R^{\star }}$ is a
contraction with respect to the metric induced by the norm $\left\Vert \cdot
\right\Vert _{s,T^{\star }}$. The appropriate choice of $T^{\star },R^{\star
}>0$ will be specified below. We first show that $\Phi $ maps $\mathcal{B}%
_{s,T^{\star },R^{\star }}$ into $\mathcal{B}_{s,T^{\star },R^{\star }}$.
Indeed, let $u\in \mathcal{B}_{s,T^{\star },R^{\star }}$. Then, by
assumption (\textbf{Hf2}) and the fact that $u\in C([0,T^{\ast }],V_{1/2})$,
for every $t\in \lbrack 0,T^{\star }]$, we have that
\begin{equation}
\Vert f(u(\cdot ,t))\Vert _{L^{2}(\Omega )}\leq C_{f}\Vert u(\cdot ,t)\Vert
_{V_{1/2}}^{r},  \label{C1}
\end{equation}%
for some $C_{f}>0$. Using the estimate \eqref{C1} we get
\begin{align*}
& \left\Vert \int_{0}^{t}f(u(\xi ,\tau ))(t-\tau )^{\alpha -1}E_{\alpha
,\alpha }((-m(\xi )(t-\tau )^{\alpha })d\tau \right\Vert _{V_{1/2}} \\
& \leq C\int_{0}^{t}(t-\tau )^{\frac{\alpha }{2}-1}\Vert u(\cdot ,\tau
)\Vert _{V_{1/2}}^{r}d\tau \\
& \leq C(T^{\star })^{\frac{\alpha }{2}}(R^{\star })^{r}.
\end{align*}%
Thus, proceeding as the proof of Theorem \ref{theo-weak}, we get that there
is a constant $C>0$ such that for every $t\in \lbrack 0,T^{\star }]$, 
\begin{equation}
\Vert \Phi (u)(\cdot ,t)\Vert _{V_{1/2}}\leq C\left( \Vert u_{0}\Vert
_{V_{1/2}}+\Vert u_{1}\Vert _{L^{2}(\Omega )}+(T^{\star })^{\frac{\alpha }{2}%
}(R^{\star })^{r}\right) .  \label{C2}
\end{equation}%
Next, we show that $\Phi (u)\in C([0,T^{\star }];V_{1/2})$. Indeed, since $%
t\mapsto E_{\alpha ,1}(-m(\xi )t^{\alpha })v$ is continuous for every $t\in
\lbrack 0,T^{\star }]$ and for all $v\in \mathcal{H}$, 
\begin{equation*}
|(E_{\alpha ,1}(-m(\xi )(t+h)^{\alpha })-E_{\alpha ,1}(-m(\xi )t^{\alpha
}))(m(\xi ))^{1/2}u_{0}(\xi )|^{2}\leq C^{2}|(m(\xi ))^{1/2}u_{0}(\xi )|^{2}
\end{equation*}%
and $u_{0}\in V_{1/2}$, it follows from the Dominated Convergence Theorem that the function $%
E_{\alpha ,1}(-m(\cdot )t^{\alpha })u_{0}(\cdot )\in C([0,T^{\star
}];V_{1/2})$. On the other hand, we see that $tE_{\alpha ,2}(-m(\cdot
)t^{\alpha })u_{1}(\cdot )$ belongs to $C([0,T^{\star }];V_{1/2})$. The
continuity in $t=0$ is direct. Let $t>0$ and $h>0$. Using \eqref{2}, we get%
\begin{align*}
& \left\vert (m(\xi ))^{1/2}\left[ (t+h)E_{\alpha ,2}(-m(\xi )(t+h)^{\alpha
})-tE_{\alpha ,2}(-m(\xi )t^{\alpha })\right] \right\vert \\
& =\left\vert (m(\xi ))^{1/2}\int_{t}^{t+h}E_{\alpha ,1}(-m(\xi )\sigma
^{\alpha })d\sigma \right\vert \\
& \leq C(m(\xi ))^{-1/2}\int_{t}^{t+h}\frac{m(\xi )\sigma ^{\alpha }}{%
1+m(\xi )\sigma ^{\alpha }}\sigma \,d\sigma \\
& \leq \frac{Cm_{0}^{-1/2}}{\alpha -1}\left( \frac{1}{t^{\alpha -1}}-\frac{1%
}{(t+h)^{\alpha -1}}\right) .
\end{align*}%
Therefore, we deduce 
\begin{align*}
& \Vert \left[ (t+h)E_{\alpha ,2}(-m(\cdot )(t+h)^{\alpha })-tE_{\alpha
,2}(-m(\cdot )t^{\alpha })\right] u_{1}(\cdot )\Vert _{V_{1/2}} \\
& \leq \frac{Cm_{0}}{\alpha -1}\left( \frac{1}{t^{\alpha -1}}-\frac{1}{%
(t+h)^{\alpha -1}}\right) \Vert u_{1}(\cdot )\Vert _{L^{2}(\Omega
)}\rightarrow 0\quad \mbox{as}\quad h\rightarrow 0^{+}.
\end{align*}%
Next, 
\begin{align*}
& \Vert \int_{0}^{t+h}(t+h-\tau )^{\alpha -1}E_{\alpha ,\alpha }(-m(\cdot
)(t+h-\tau )^{\alpha })f(u(\cdot ,\tau ))d\tau \\
& -\int_{0}^{t}(t-\tau )^{\alpha -1}E_{\alpha ,\alpha }(-m(\cdot )(t-\tau
)^{\alpha })f(u(\cdot ,\tau ))\;d\tau \Vert _{V_{1/2}} \\
& \leq \int_{0}^{t}\Vert \lbrack (t+h-\tau )^{\alpha -1}E_{\alpha ,\alpha
}(-m(\cdot )(t+h-\tau )^{\alpha })-(t-\tau )^{\alpha -1}E_{\alpha ,\alpha
}(-m(\cdot )(t-\tau )^{\alpha })]f(u(\cdot ,\tau ))\Vert _{V_{1/2}}d\tau \\
& +\int_{t}^{t+h}\Vert (t+h-\tau )^{\alpha -1}E_{\alpha ,\alpha }(-m(\cdot
)(t+h-\tau )^{\alpha })f(u(\cdot ,\tau ))\Vert _{V_{1/2}}d\tau \\
& =:I_{1}+I_{2}.
\end{align*}%
Let us see that $I_{1}\rightarrow 0$ as $h\rightarrow 0^{+}$. First, setting 
\begin{equation*}
Q_{\alpha ,\beta }\left( x\right) :=x^{\beta -1}E_{\alpha ,\beta }(-m(\cdot
)x^{\alpha }).
\end{equation*}%
We need to prove that%
\begin{equation}
\lim_{h\rightarrow 0^{+}}\Vert \left[ Q_{\alpha ,\alpha }\left( t+h-\tau
\right) -Q_{\alpha ,\alpha }\left( t-\tau \right) \right] f(u(\cdot ,\tau
))\Vert _{V_{1/2}}=0.  \label{dct}
\end{equation}%
Indeed, since $t\mapsto t^{\alpha -1}E_{\alpha ,\alpha }(-m(\xi )t^{\alpha
})f(u(\xi ,t))$ is continuous for every $t\in \lbrack 0,T^{\star }]$, 
\begin{align*}
& |(m(\xi ))^{1/2}[Q_{\alpha ,\alpha }\left( t+h-\tau \right) -Q_{\alpha
,\alpha }\left( t-\tau \right) ]f(u(\xi ,\tau ))| \\
& \leq 2C(t-\tau )^{\frac{\alpha }{2}-1}|f(u(\xi ,\tau ))|,
\end{align*}%
and $f(u(\cdot ,\tau ))\in L^{2}(\Omega )$, \eqref{dct} follows from
Dominated Convergence Theorem. Moreover, it holds%
\begin{align*}
\Vert \left[ Q_{\alpha ,\alpha }\left( t+h-\tau \right) -Q_{\alpha ,\alpha
}\left( t-\tau \right) \right] f(u(\cdot ,\tau ))\Vert _{V_{1/2}}& \leq
C(t-\tau )^{\frac{\alpha }{2}-1}\Vert u(\cdot ,\tau )\Vert _{V_{1/2}}^{r} \\
& \leq C(t-\tau )^{\frac{\alpha }{2}-1}(R^{\star })^{r}.
\end{align*}%
Since $\tau \mapsto C(t-\tau )^{\frac{\alpha }{2}-1}(R^{\star })^{r}$ is
integrable, then $I_{1}\rightarrow 0$ as $h\rightarrow 0^{+}$ by the
Dominated Convergence Theorem. On the other hand, observe that 
\begin{equation*}
\Vert Q_{\alpha ,\alpha }\left( t+h-\tau \right) f(u(\cdot ,\tau ))\Vert
_{V_{1/2}}\leq C(t+h-\tau )^{\frac{\alpha }{2}-1}(R^{\star })^{r}
\end{equation*}%
implies 
\begin{equation*}
I_{2}\leq C(R^{\star })^{r}\int_{t}^{t+h}(t+h-\tau )^{\alpha -2}d\tau
=C(R^{\star })^{r}h^{\alpha -1}\rightarrow 0,\,\,h\rightarrow 0^{+}.
\end{equation*}%
Thus, the claim $\Phi (u)\in C([0,T^{\star }];V_{1/2})$ follows. Similarly,
we have that there is a constant $C>0,$ such that for every $t\in \lbrack
0,T^{\star }]$, 
\begin{align}
\Vert \Phi (u)^{\prime }(t)\Vert _{L^{2}(\Omega )}\leq & C\left( t^{\frac{%
\alpha }{2}-1}\Vert u_{0}\Vert _{V_{1/2}}+\Vert u_{1}\Vert _{L^{2}(\Omega
)}+(T^{\star })^{\alpha -1}\left\Vert u\right\Vert _{C\left( \left[
0,T^{\ast }\right] ;V_{1/2}\right) }^{r}\right)  \notag \\
\leq & C\left( t^{\frac{\alpha }{2}-1}\Vert u_{0}\Vert _{V_{1/2}}+\Vert
u_{1}\Vert _{L^{2}(\Omega )}+(T^{\star })^{\alpha -1}\left( R^{\star
}\right) ^{r}\right) . \label{C3-1}
\end{align}%
From here%
\begin{align}
t^{s}\Vert \Phi (u)^{\prime }(t)\Vert _{L^{2}(\Omega )}\leq & C\left( t^{s+%
\frac{\alpha }{2}-1}\Vert u_{0}\Vert _{V_{1/2}}+t^{s}\Vert u_{1}\Vert
_{L^{2}(\Omega )}+t^{s}(T^{\star })^{\alpha -1}\left\Vert u\right\Vert
_{C\left( \left[ 0,T^{\ast }\right] ;V_{1/2}\right) }^{r}\right) \notag\\
\leq & C\left( (T^{\star })^{s+\frac{\alpha }{2}-1}\Vert u_{0}\Vert
_{V_{1/2}}+(T^{\star })^{s}\Vert u_{1}\Vert _{L^{2}(\Omega )}+(T^{\star
})^{s+\alpha -1}\left( R^{\star }\right) ^{r}\right) .  \label{C-3}
\end{align}%
Clearly, $\Phi ^{\prime }(u)\in C((0,T^{\star }];L^{2}(\Omega ))$. Indeed,
since $t\mapsto m(\xi )t^{\alpha -1}E_{\alpha ,\alpha }(-m(\xi )t^{\alpha
})u_{0}(\xi )$ is continuous for every $t\in \lbrack 0,T^{\star }]$, 
\begin{align*}
& |m(\xi )(t+h)^{\alpha -1}E_{\alpha ,\alpha }(-m(\xi )(t+h)^{\alpha
})u_{0}(\xi )-m(\xi )t^{\alpha -1}E_{\alpha ,\alpha }(-m(\xi )t^{\alpha
})u_{0}(\xi )| \\
& \leq Ct^{\frac{\alpha }{2}-1}|(m(\xi ))^{1/2}u_{0}(\xi )|,
\end{align*}%
and $u_{0}\in V_{1/2}$, we have that by the Dominated Convergence Theorem
that  
$$m(\cdot )t^{\alpha -1}E_{\alpha ,\alpha }(-m(\xi )t^{\alpha
})u_{0}(\cdot )\in C([0,T^{\star }];L^{2}(\Omega )).$$
Analogously,
it can be shown that $E_{\alpha ,1}(-m(\cdot )t^{\alpha })u_{1}(\cdot )$
belongs to $C([0,T^{\star }];L^{2}(\Omega ))$. Notice that 
\begin{equation*}
\lim_{h\rightarrow 0^{+}}\left\vert \left( Q_{\alpha ,\alpha -1}(t+h-\tau
)-Q_{\alpha ,\alpha -1}(t-\tau )\right) f(u(\xi ,\tau ))\right\vert ^{2}=0,
\end{equation*}%
for a.e. $t>0,$ and%
\begin{align*}
& \left\vert \left( Q_{\alpha ,\alpha -1}(t+h-\tau )-Q_{\alpha ,\alpha
-1}(t-\tau )\right) f(u(\xi ,\tau ))\right\vert ^{2} \\
& \leq C^{2}(t-\tau )^{2(\alpha -2)}|f(u(\xi ,\tau ))|^{2}\in L^{1}(\Omega ).
\end{align*}%
Owing to \eqref{C1}, we then have%
\begin{equation*}
\lim_{h\rightarrow 0^{+}}\left\Vert [Q_{\alpha ,\alpha -1}(t+h-\tau
)-Q_{\alpha ,\alpha -1}(t-\tau )]f(u(\xi ,\tau ))\right\Vert _{L^{2}(\Omega
)}=0.
\end{equation*}%
Furthermore, 
\begin{align*}
& \left\Vert [Q_{\alpha ,\alpha -1}(t+h-\tau )-Q_{\alpha ,\alpha -1}(t-\tau
)]f(u(\xi ,\tau ))\right\Vert _{L^{2}(\Omega )} \\
& \leq C(t-\tau )^{\alpha -2}\Vert u(\cdot ,\tau )\Vert _{L^{2r}(\Omega
)}^{r}\in L^{1}(0,T).
\end{align*}%
The Dominated Convergence Theorem implies that 
\begin{equation*}
J_{1}:=\int_{0}^{t}\left\Vert [Q_{\alpha ,\alpha -1}(t+h-\tau )-Q_{\alpha
,\alpha -1}(t-\tau )]f(u(\xi ,\tau ))\right\Vert _{L^{2}(\Omega )}d\tau
\end{equation*}%
goes to zero as $h\rightarrow 0^{+}$. As before, we can prove that 
\begin{equation*}
J_{2}=\int_{t}^{t+h}\Vert (t+h-\tau )^{\alpha -2}E_{\alpha ,\alpha
-1}(-m(\cdot )(t+h-\tau )^{\alpha })f(u(\cdot ,\tau ))\Vert _{L^{2}(\Omega
)}d\tau \rightarrow 0,\quad h\rightarrow 0^{+}.
\end{equation*}%
Therefore, 
\begin{align*}
&\Big \Vert \int_{0}^{t+h}\left( Q_{\alpha ,\alpha -1}(t+h-\tau )-Q_{\alpha
,\alpha -1}(t-\tau )\right) f(u(\cdot ,\tau ))d\tau\Big \Vert _{L^{2}(\Omega )}
\\
& =:J_{1}+J_{2}\rightarrow 0,\quad h\rightarrow 0^{+}.
\end{align*}%
Henceforth, $\Phi ^{\prime }(u)\in C((0,T^{\star }];L^{2}(\Omega ))$. Next,
by \eqref{C3-1} and using the formula 
\begin{equation*}
\int_{0}^{t}g_{2-\alpha }(t-\sigma )\sigma ^{p}\,d\sigma =\frac{%
t^{p+2-\alpha }}{\Gamma (2-\alpha )}\mathbf{B}(p+1,2-\alpha ),
\end{equation*}%
where $\mathbf{B}(a,b)=\int_{0}^{1}\tau ^{a-1}(1-\tau )^{b-1}\,d\tau $ is
the standard Beta function for $a,b>0$, we obtain 
\begin{align*}
& g_{2-\alpha }\ast \Vert \partial _{t}\Phi \left( u\right) \Vert
_{L^{2}(\Omega )}^{2}(t) \\
& \leq C^{2}\int_{0}^{t}g_{2-\alpha }(t-\sigma )\left[ \sigma ^{\frac{\alpha 
}{2}-1}\Vert u_{0}(\cdot )\Vert _{V_{1/2}}+\Vert u_{1}(\cdot )\Vert
_{L^{2}(\Omega )}+(T^{\star })^{\alpha -1}(R^{\star })^{r}\right] ^{2}d\sigma
\\
& =C^{2}\left[ \mathbf{B}\left( \alpha -1,2-\alpha \right) \Vert u_{0}(\cdot
)\Vert _{V_{1/2}}^{2}\right. \\
& +\left. \mathbf{B}\left( \frac{\alpha }{2},2-\alpha \right) t^{1-\frac{%
\alpha }{2}}\Vert u_{0}(\cdot )\Vert _{V_{1/2}}\left( \Vert u_{1}(\cdot
)\Vert _{L^{2}(\Omega )}+(T^{\star })^{\alpha -1}(R^{\star })^{r}\right)
\right. \\
& +\left. t^{2-\alpha }\left( \Vert u_{1}(\cdot )\Vert _{L^{2}(\Omega
)}^{2}+(T^{\star })^{2(\alpha -1)}(R^{\star })^{2r}+\Vert u_{1}(\cdot )\Vert
_{L^{2}(\Omega )}(T^{\star })^{\alpha -1}(R^{\star })^{r}\right) \right] \\
& =C^{2}\left( \Vert u_{0}(\cdot )\Vert _{V_{1/2}}^{2}+t^{2-\alpha }\Vert
u_{1}(\cdot )\Vert _{L^{2}(\Omega )}^{2}+t^{2-\alpha }(T^{\star })^{2(\alpha
-1)}(R^{\star })^{2r}\right. \\
& +\left. t^{1-\frac{\alpha }{2}}\Vert u_{0}(\cdot )\Vert _{V_{1/2}}\Vert
u_{1}(\cdot )\Vert _{L^{2}(\Omega )}+t^{1-\frac{\alpha }{2}}\Vert
u_{0}(\cdot )\Vert _{V_{1/2}}(T^{\star })^{\alpha -1}(R^{\star })^{r}\right.
\\
& +\left. t^{2-\alpha }\Vert u_{1}(\cdot )\Vert _{L^{2}(\Omega )}(T^{\star
})^{\alpha -1}(R^{\star })^{r}\right) .
\end{align*}%
Hence, 
\begin{align}
g_{2-\alpha }\ast \Vert \partial _{t}\Phi \left( u\right) \Vert
_{L^{2}(\Omega )}^{2}(t)& \leq C^{2}\left( \Vert u_{0}(\cdot )\Vert
_{V_{1/2}}^{2}+(T^{\star })^{2-\alpha }\Vert u_{1}(\cdot )\Vert
_{L^{2}(\Omega )}^{2}+(T^{\star })^{\alpha }(R^{\star })^{2r}\right.
\notag\\
& +\left. (T^{\star })^{1-\frac{\alpha }{2}}\Vert u_{0}(\cdot )\Vert
_{V_{1/2}}\Vert u_{1}(\cdot )\Vert _{L^{2}(\Omega )}+\Vert u_{0}(\cdot
)\Vert _{V_{1/2}}(T^{\star })^{\frac{\alpha }{2}}(R^{\star })^{r}\right. 
\notag \\
& +\left. \Vert u_{1}(\cdot )\Vert _{L^{2}(\Omega )}T^{\star }(R^{\star
})^{r}\right) . \label{C3-2} 
\end{align}%
It also follows from \eqref{C2}, \eqref{C3} and \eqref{C3-2} that%
\begin{align*}
\Vert \Phi \left( u\right) (\cdot ,t)\Vert _{s,T^{\star }}^{2} =&\Vert \Phi
\left( u\right) (\cdot ,t)\Vert _{V_{1/2}}^{2}+t^{2s}\Vert \partial _{t}\Phi
\left( u\right) (\cdot ,t)\Vert _{L^{2}(\Omega )}^{2}+g_{2-\alpha }\ast
\Vert \partial _{t}\Phi \left( u\right) (\xi ,\cdot )\Vert _{L^{2}(\Omega
)}^{2}(t) \\
 \leq &C^{2}\left( 2+(T^{\star })^{2s+\alpha -2})\Vert u_{0}\Vert
_{V_{1/2}}^{2}+(1+(T^{\star })^{2s}+(T^{\star })^{2-\alpha })\Vert
u_{1}\Vert _{L^{2}(\Omega )}^{2}\right. \\
& +\left. \left( (T^{\star })^{\alpha }+(T^{\star })^{2(s+\alpha -1)}\right)
(R^{\star })^{2r}\right. \\
& +\left. (1+(T^{\star })^{2s+\frac{\alpha }{2}-1}+(T^{\star })^{1-\frac{%
\alpha }{2}})\Vert u_{0}\Vert _{V_{1/2}}\Vert u_{1}\Vert _{L^{2}(\Omega
)}\right. \\
& +\left. \left( 2(T^{\star })^{\frac{\alpha }{2}}+(T^{\star })^{2s+\frac{%
3\alpha }{2}-2}\right) \Vert u_{0}\Vert _{V_{1/2}}(R^{\star })^{r}\right. \\
& +\left. \left( (T^{\star })^{\frac{\alpha }{2}}+(T^{\star })^{2s+\alpha
-1}+T^{\star }\right) \Vert u_{1}\Vert _{L^{2}(\Omega )}(R^{\star
})^{r}\right) .
\end{align*}%
Doing some computations, we arrive to 
\begin{align*}
\Vert \Phi \left( u\right) (\cdot ,t)\Vert _{s,T^{\star }}& \leq C\left(
(2+(T^{\star })^{2s+\alpha -2})^{1/2}\Vert u_{0}\Vert
_{V_{1/2}}+((1+(T^{\star })^{2s}+(T^{\star })^{2-\alpha })^{1/2}\Vert
u_{1}\Vert _{L^{2}(\Omega )}\right. \\
& +\left. ((T^{\star })^{\alpha }+(T^{\star })^{2(s+\alpha
-1)})^{1/2}(R^{\star })^{r}\right) .
\end{align*}%
Moreover, using that $\sqrt{x^{2}+y^{2}}\leq x+y$ for all $x,y\geq 0$, we
have that
\begin{align*}
\Vert \Phi \left( u\right) (\cdot ,t)\Vert _{s,T^{\star }} \leq &C(\Vert
u_{0}\Vert _{V_{1/2}}+\Vert u_{1}\Vert _{L^{2}(\Omega )}) \\
& +C\left( (T^{\star })^{s+\frac{\alpha }{2}-1}\Vert u_{0}\Vert
_{V_{1/2}}+((T^{\star })^{s}+(T^{\star })^{1-\frac{\alpha }{2}})\Vert
u_{1}\Vert _{L^{2}(\Omega )}\right. \\
& +\left. ((T^{\star })^{\frac{\alpha }{2}}+(T^{\star })^{s+\alpha
-1})(R^{\star })^{r}\right) \\
& \leq C(\Vert u_{0}\Vert _{V_{1/2}}+\Vert u_{1}\Vert _{L^{2}(\Omega )}) \\
& +C\max \{(T^{\star })^{s+\frac{\alpha }{2}-1},(T^{\star })^{s}+(T^{\star
})^{1-\frac{\alpha }{2}}\}\left( \Vert u_{0}\Vert _{V_{1/2}}+\Vert
u_{1}\Vert _{L^{2}(\Omega )}\right) \\
& +C\left( ((T^{\star })^{\frac{\alpha }{2}}+(T^{\star })^{s+\alpha
-1})(R^{\star })^{r}\right) .
\end{align*}%
Letting now 
\begin{equation*}
R^{\star }\geq 3C(\Vert u_{0}\Vert _{V_{1/2}}+\Vert u_{1}\Vert
_{L^{2}(\Omega )}),
\end{equation*}%
we can find a sufficiently small time $T^{\star }>0$ such that 
\begin{align*}
& C\max \{(T^{\star })^{s+\frac{\alpha }{2}-1},(T^{\star })^{s}+(T^{\star
})^{1-\frac{\alpha }{2}}\}\left( \Vert u_{0}\Vert _{V_{1/2}}+\Vert
u_{1}\Vert _{L^{2}(\Omega )}\right) \\
& +C\left( ((T^{\star })^{\frac{\alpha }{2}}+(T^{\star })^{s+\alpha
-1})(R^{\star })^{r}\right) \leq \frac{2R^{\star }}{3},
\end{align*}%
that is, 
\begin{equation}
\Vert \Phi \left( u\right) \Vert _{s,T^{\star }}\leq R^{\star },  \label{T1}
\end{equation}%
{\color{black} provided that $0<s<1$ is such that 
\begin{equation*}
s>1-\frac{\alpha }{2}.
\end{equation*}%
} In this case, it follows that $\Phi (u)\in \mathcal{B}_{s,T^{\star
},R^{\star }},$ for all $u\in \mathcal{B}_{s,T^{\star },R^{\star }}$.

Next, we show that by choosing a possibly smaller $T^{\star }>0$, $\Phi $ is
a contraction on $\mathcal{B}_{s,T^{\star },R^{\star }}$. Indeed, let $%
u,v\in \mathcal{B}_{s,T^{\star },R^{\star }}$. Using the mean value theorem
it is easy to show that%
\begin{equation}
\Vert f(u(\cdot ,t))-f(v(\cdot ,t))\Vert _{L^{2}(\Omega )}\leq C\Vert
u(\cdot ,t)-v(\cdot ,t)\Vert _{V_{1/2}}\left( \Vert u(\cdot ,t)\Vert
_{V_{1/2}}+\Vert v(\cdot ,t)\Vert _{V_{1/2}}\right) ^{r-1}.  \label{EQ.4.10}
\end{equation}%
Similarly to the foregoing estimates, we can exploit%
\begin{align*}
&\Vert \Phi (u)\left( t\right) -\Phi (v)\left( t\right) \Vert _{V_{1/2}}\\
\leq &C\int_{0}^{t}(t-\tau )^{\frac{\alpha }{2}-1}\Vert f(u(\cdot ,\tau
))-f(v(\cdot ,\tau ))\Vert _{L^{2}(\Omega )}d\tau \\
 \leq &C\int_{0}^{t}(t-\tau )^{\frac{\alpha }{2}-1}\Vert u(\cdot ,\tau
)-v(\cdot ,\tau )\Vert _{V_{1/2}}\left( \Vert u(\cdot ,\tau )\Vert
_{V_{1/2}}+\Vert v(\cdot ,\tau )\Vert _{V_{1/2}}\right) ^{r-1}d\tau \\
 \leq &C\Vert u-v\Vert _{C([0,T^{\star }];V_{1/2})}(R^{\star
})^{r-1}(T^{\star })^{\frac{\alpha }{2}}.
\end{align*}%
Analogously, we have 
\begin{equation*}
\Vert \Phi (u)^{\prime }\left( t\right) -\Phi (v)^{\prime }\left( t\right)
\Vert _{L^{2}(\Omega )}\leq C\Vert u-v\Vert _{C([0,T^{\star
}];V_{1/2})}(R^{\star })^{r-1}(T^{\star })^{\alpha -1},
\end{equation*}%
and 
\begin{equation*}
t^{s}\Vert \Phi (u)^{\prime }\left( t\right) -\Phi (v)^{\prime }\left(
t\right) \Vert _{L^{2}(\Omega )}\leq C\Vert u-v\Vert _{C([0,T^{\star
}];V_{1/2})}(R^{\star })^{r-1}(T^{\star })^{s+\alpha -1}.
\end{equation*}%
Finally,%
\begin{equation*}
g_{2-\alpha }\ast \Vert (\Phi (u)-\Phi (v))^{\prime }\Vert _{L^{2}(\Omega
)}^{2}(t)\leq C^{2}\Vert u-v\Vert _{C([0,T^{\star }];V_{1/2})}^{2}(R^{\star
})^{2(r-1)}(T^{\star })^{\alpha }.
\end{equation*}%
Then, for each $t\in \lbrack 0,T^{\star }],$ 
\begin{align*}
& \Vert \Phi (u)\left( t\right) -\Phi (v)\left( t\right) \Vert _{V_{\gamma
}}^{2}+t^{2s}\Vert \Phi (u)^{\prime }(t)-\Phi (v)^{\prime }(t)\Vert
_{L^{2}(\Omega )}^{2}+g_{2-\alpha }\ast \Vert (\Phi (u)-\Phi (v))^{\prime
}\Vert _{L^{2}(\Omega )}^{2}(t) \\
& \leq C^{2}(R^{\star })^{2(r-1)}\left( 2(T^{\star })^{\alpha }+(T^{\star
})^{2(s+\alpha -1)}\right) \Vert u-v\Vert _{C([0,T^{\star }];V_{1/2})}^{2}.
\end{align*}%
\newline
Namely, it holds%
\begin{equation*}
\Vert \Phi (u)-\Phi (v)\Vert _{s,T^{\star }}\leq C(R^{\star })^{r-1}\left(
2(T^{\star })^{\alpha }+(T^{\star })^{2(s+\alpha -1)}\right) ^{1/2}\Vert
u-v\Vert _{C([0,T^{\star }];V_{1/2})}.
\end{equation*}%
Thus, for a sufficiently small $T^{\ast }\leq 1,$ the previous estimate
shows that $\Phi $ is a contraction.

\textbf{Step 2 (The extension argument)}. Let 
\begin{equation}
S_{1}(t)u_{0}(\xi ):=E_{\alpha ,1}(-m(\xi )t^{\alpha })u_{0}(\xi )\;%
\mbox{
and }\;S_{2}(t)u_{1}(\xi ):=tE_{\alpha ,2}(-m(\xi )t^{\alpha })u_{1}(\xi ),
\label{S12}
\end{equation}%
and 
\begin{equation}
(S_{3}\ast f(u(\xi ,\cdot )))(t):=\int_{0}^{t}f(u(\cdot ,\tau ))(t-\tau
)^{\alpha -1}E_{\alpha ,\alpha }((-m(\cdot )(t-\tau )^{\alpha })d\tau
\label{S33}
\end{equation}%
so that 
\begin{equation*}
u(\xi ,t)=S_{1}(t)u_{0}(\xi )+S_{2}(t)u_{1}(\xi )+(S_{3}\ast f(u(\xi ,\cdot
)))(t),\quad \xi \in \Omega .
\end{equation*}%
Let also 
\begin{equation}
S_{1}^{\prime }(t)u_{0}(\cdot ):=-m(\cdot )t^{\alpha -1}E_{\alpha
,1}(-m(\cdot )t^{\alpha })u_{0}(\cdot )\;\mbox{
and }\;S_{2}^{\prime }(t)u_{1}(\cdot ):=E_{\alpha ,1}(-m(\cdot )t^{\alpha
})u_{1}(\cdot ),  \label{S12p}
\end{equation}%
and 
\begin{equation}
(S_{3}^{\prime }\ast f(u(\xi ,\cdot )))(t):=\int_{0}^{t}f(u(\cdot ,\tau
))(t-\tau )^{\alpha -2}E_{\alpha ,\alpha -1}(-m(\cdot )(t-\tau )^{\alpha
})d\tau  \label{S3p}
\end{equation}%
so that 
\begin{equation*}
\partial _{t}u(\xi ,t)=S_{1}^{\prime }(t)u_{0}(\xi )+S_{2}^{\prime
}(t)u_{1}(\xi )+(S_{3}^{\prime }\ast f(u(\xi ,\cdot )))(t),\quad \xi \in
\Omega .
\end{equation*}%
Let $T^{\star }>0$ be the time of the contraction argument. Fix $\tau >0$
and consider the space 
\begin{align*}
\mathcal{K}_{s,T^{\star },R}:=& \Big\{v\in C([0,T^{\star }+\tau
];V_{1/2})\cap C^{1}((0,T^{\star }+\tau ];L^{2}(\Omega ))\text{:} \\
& v(\cdot ,t)=u(\cdot ,t),\;\qquad \forall \;t\in \lbrack 0,T^{\star }], \\
& \Vert v(\cdot ,t)-u(\cdot ,T^{\star })\Vert _{V_{1/2}}^{2}+%
\textcolor{black}{(t-T^{\star})^{2s}}\Vert \partial _{t}v(\cdot ,t)-\partial
_{t}u(\cdot ,T^{\star })\Vert _{L^{2}(\Omega )}^{2} \\
& +g_{2-\alpha }\ast \Vert \partial _{t}v(\cdot ,t)-\partial _{t}u(\cdot
,T^{\star })\Vert _{L^{2}(\Omega )}^{2}\leq R^{2},\;\;\forall \;t\in \lbrack
T^{\star },T^{\star }+\tau ]\Big\}.
\end{align*}%
Define the mapping $\Phi $ on $\mathcal{K}_{s,T^{\star },R}$ by 
\begin{align}
\Phi (v)(\cdot ,t)=& E_{\alpha ,1}(-m(\cdot )t^{\alpha })u_{0}(\cdot
)+tE_{\alpha ,2}(-m(\cdot )t^{\alpha })u_{1}(\cdot )\   \label{AA1} \\
& +\int_{0}^{t}f(v(\cdot ,\tau ))(t-\tau )^{\alpha -1}E_{\alpha ,\alpha
}(-m(\cdot )(t-\tau )^{\alpha })d\tau .  \notag
\end{align}%
Note that $\mathcal{K}_{s,T^{\star },R}$ when endowed with the norm of $%
\Vert \cdot \Vert _{s,T^{\star }}$ is a closed subspace of $\mathcal{B}%
_{s,T^{\star },R}$. We show that $\Phi $ has a fixed point in $\mathcal{K}%
_{s,T^{\star },R}$. Since $f$ is continuously differentiable, we have that
the mapping $t\mapsto \Phi (v(\cdot ))(t)$ is continuously differentiable on 
$[0,T^{\star }+\tau ]$. We will show that by properly choosing $\tau ,R>0$, $%
\Phi :\mathcal{K}_{s,T^{\star },R}\rightarrow \mathcal{K}_{s,T^{\star },R}$
is a contraction mapping with respect to the metric induced by the norm $%
\Vert \cdot \Vert _{s,T^{\star }}$. The appropriate choice of $\tau ,R>0$
will be specified below. First, We show that $\Phi $ maps $\mathcal{K}%
_{s,T^{\star },R}$ into $\mathcal{K}_{s,T^{\star },R}$. Indeed, let $v\in 
\mathcal{K}_{s,T^{\star },R}$. If $t\in \lbrack 0,T^{\star }]$, then $%
v(\cdot ,t)=u(\cdot ,t)$. Hence $\Phi (v(\cdot ))(t)=\Phi (u(\cdot
))(t)=u(\cdot ,t)$ and there is nothing to prove. If $t\in \lbrack T^{\star
},T^{\star }+\tau ]$, then%
\begin{align*}
& \Vert \Phi (v)(\cdot ,t)-u(\cdot ,T^{\star })\Vert _{V_{1/2}} \\
\leq & \Vert S_{1}(t)u_{0}(\cdot )-S_{1}(T^{\star })u_{0}(\cdot )\Vert
_{V_{1/2}}+\Vert S_{2}(t)u_{1}(\cdot )-S_{2}(T^{\star })u_{1}(\cdot )\Vert
_{V_{1/2}} \\
& +\int_{T^{\star }}^{t}\left\Vert (t-\sigma )^{\alpha -1}E_{\alpha ,\alpha
}(-m(\cdot )(t-\sigma )^{\alpha })f(v(\cdot ,\sigma ))\right\Vert
_{V_{1/2}}d\sigma \\
& +\int_{0}^{T^{\star }}\left\Vert \left[ (t-\sigma )^{\alpha -1}-(T^{\star
}-\sigma )^{\alpha -1}\right] E_{\alpha ,\alpha }(-m(\cdot )(t-\sigma
)^{\alpha })f(u(\cdot ,\sigma ))\right\Vert _{V_{1/2}}d\sigma \\
& +\int_{0}^{T^{\star }}\left\Vert (T^{\star }-\sigma )^{\alpha -1}\left[
E_{\alpha ,\alpha }(-m(\cdot )(t-\sigma )^{\alpha })-E_{\alpha ,\alpha
}(-m(\cdot )(T^{\star }-\sigma )^{\alpha })\right] f(u(\cdot ,\sigma
))\right\Vert _{V_{1/2}}d\sigma \\
=& \mathcal{N}_{1}+\mathcal{N}_{2}+\mathcal{N}_{3}+\mathcal{N}_{4}.
\end{align*}%
According to the proof of the contraction argument, we have that for every $%
T\geq 0$, the mappings $t\mapsto S_{1}(t)u_{0}(\xi )$ and $t\mapsto
S_{2}(t)u_{1}(\xi )$ belong to $C([0,T],V_{1/2})$. Hence we can choose $\tau
>0$ small such that for $t\in \lbrack T^{\star },T^{\star }+\tau ]$, we have 
\begin{equation}
\mathcal{N}_{1}:=\Vert S_{1}(t)u_{0}(\cdot )-S_{1}(T^{\star })u_{0}(\cdot
)\Vert _{V_{1/2}}+\Vert S_{2}(t)u_{1}(\cdot )-S_{2}(T^{\star })u_{1}(\cdot
)\Vert _{V_{1/2}}\leq \frac{R}{4\sqrt{3}}.  \label{N1}
\end{equation}%
Proceeding as the proofs above, we can choose $\tau >0$ small such that for $%
t\in \lbrack T^{\star },T^{\star }+\tau ]$, we have 
\begin{align}
\mathcal{N}_{2}:=& \int_{T^{\star }}^{t}\left\Vert (t-\sigma )^{\alpha
-1}E_{\alpha ,\alpha }(-m(\cdot )(t-\sigma )^{\alpha })f(v(\cdot ,\sigma
))\right\Vert _{V_{1/2}}d\sigma  \label{N2} \\
\leq & C\tau ^{\frac{\alpha }{2}}\left( \Vert v(\cdot ,t)\Vert
_{V_{1/2}}^{r}\right) \leq 2C\tau ^{\frac{\alpha }{2}}(R^{\star })^{r}\leq 
\frac{R}{4\sqrt{3}}.  \notag
\end{align}%
For the third norm we have that 
\begin{equation}
\mathcal{N}_{3}:=\int_{0}^{T^{\star }}\left\Vert \left[ (t-\sigma )^{\alpha
-1}-(T^{\star }-\sigma )^{\alpha -1}\right] E_{\alpha ,\alpha }(-m(\cdot
)(t-\sigma )^{\alpha })f(u(\cdot ,\sigma ))\right\Vert _{V_{1/2}}d\sigma .
\label{N3}
\end{equation}%
Since 
\begin{equation*}
\lim_{t\rightarrow T^{\star }}\left\vert m(\xi )^{1/2}\left[ (t-\sigma
)^{\alpha -1}-(T^{\star }-\sigma )^{\alpha -1}\right] E_{\alpha ,\alpha
}(-m(\xi )(t-\sigma )^{\alpha })f(u(\xi ,\sigma )\right\vert =0
\end{equation*}%
and%
\begin{align*}
& |(m(\xi )^{1/2}\left[ (t-\sigma )^{\alpha -1}-(T^{\star }-\sigma )^{\alpha
-1}\right] E_{\alpha ,\alpha }(-m(\xi )(t-\sigma )^{\alpha })f(u(\xi ,\sigma
))| \\
& \leq C(T^{\star }-\sigma )^{\frac{\alpha }{2}-1}|f(u(\xi ,\sigma ))|,
\end{align*}%
owing to (\textbf{Hf2}), then 
\begin{equation*}
\left\Vert \left[ (t-s)^{\alpha -1}-(T^{\star }-s)^{\alpha -1}\right]
E_{\alpha ,\alpha }(-m(\cdot )(t-s)^{\alpha })f(u(\cdot ,s))\right\Vert
_{V_{1/2}}\;\rightarrow 0\;\mbox{ as }\;t\rightarrow T^{\star }
\end{equation*}%
by the Dominated Convergence Theorem. This implies that we can choose $\tau
>0$ small enough, such that for $t\in \lbrack T^{\star },T^{\star }+\tau ]$, 
\begin{equation}
\mathcal{N}_{3}=\int_{0}^{T^{\star }}\left\Vert \left[ (t-s)^{\alpha
-1}-(T^{\star }-s)^{\alpha -1}\right] E_{\alpha ,\alpha }(-m(\cdot
)(t-s)^{\alpha })f(u(\cdot ,s))\right\Vert _{V_{1/2}}ds\leq \frac{R}{4\sqrt{3%
}}.  \label{N-3-2}
\end{equation}%
With the same argument as for $\mathcal{N}_{3}$, we can choose $\tau >0$
small such that for every $t\in \lbrack T^{\star },T^{\star }+\tau ]$ we have%
\begin{align}
\mathcal{N}_{4}:=\int_{0}^{T^{\star }}& \left\Vert (T^{\star }-\sigma
)^{\alpha -1}\left[ E_{\alpha ,\alpha }(-m(\cdot )(t-\sigma )^{\alpha
})-E_{\alpha ,\alpha }(-m(\cdot )(T^{\star }-\sigma )^{\alpha })\right]
f(u(\cdot ,\sigma ))\right\Vert _{V_{1/2}}d\sigma  \notag \\
\leq & \frac{R}{4\sqrt{3}}.  \label{N4}
\end{align}%
Thus, for $\tau >0$ small enough, we get 
\begin{equation*}
\Vert \Phi (v)(\cdot ,t)-u(\cdot ,T^{\star })\Vert _{V_{1/2}}^{2}\leq \frac{%
R^{2}}{3}.
\end{equation*}%
For the time derivative, we also have that 
\begin{align*}
& \Vert \Phi (v(\cdot ))^{\prime }(t)-\partial _{t}u(\cdot ,T^{\star })\Vert
_{L^{2}(\Omega )} \\
\leq & \Vert S_{1}^{\prime }(t)u_{0}(\cdot )-S_{1}^{\prime }(T^{\star
})u_{0}(\cdot )\Vert _{L^{2}(\Omega )}+\Vert S_{2}^{\prime }(t)u_{1}(\cdot
)-S_{2}^{\prime }(T^{\star })u_{1}(\cdot )\Vert _{L^{2}(\Omega )} \\
& +\int_{T^{\star }}^{t}\left\Vert (t-\sigma )^{\alpha -2}E_{\alpha ,\alpha
-1}(-m(\cdot )(t-\sigma )^{\alpha })f(v(\cdot ,\sigma ))\right\Vert
_{L^{2}(\Omega )}d\sigma \\
& +\int_{0}^{T^{\star }}\left\Vert \left[ (t-\sigma )^{\alpha -2}-(T^{\star
}-\sigma )^{\alpha -2}\right] E_{\alpha ,\alpha -2}(-m(\cdot )(t-\sigma
)^{\alpha })f(u(\cdot ,\sigma ))\right\Vert _{L^{2}(\Omega )}d\sigma \\
& +\int_{0}^{T^{\star }}\left\Vert (T^{\star }-\sigma )^{\alpha -2}\left[
E_{\alpha ,\alpha -1}(-m(\cdot )(t-\sigma )^{\alpha })-E_{\alpha ,\alpha
-1}(-m(\cdot )(T^{\star }-\sigma )^{\alpha })\right] f(u(\cdot ,\sigma
))\right\Vert _{L^{2}(\Omega )}d\sigma \\
=& \mathcal{M}_{1}+\mathcal{M}_{2}+\mathcal{M}_{3}+\mathcal{M}_{4}.
\end{align*}%
Using the same argument as the corresponding terms above, we can choose $%
\tau \in \left( 0,1\right) $ small such that for every $t\in \lbrack
T^{\star },T^{\star }+\tau ]$, 
\begin{equation}
\mathcal{M}_{1}:=\Vert S_{1}^{\prime }(t)u_{0}(\cdot )-S_{1}^{\prime
}(T^{\star })u_{0}(\cdot )\Vert _{V_{1/2}}+\textcolor{black}{\tau^s}\Vert
S_{2}^{\prime }(t)u_{1}(\cdot )-S_{2}^{\prime }(T^{\star })u_{1}(\cdot
)\Vert _{L^{2}(\Omega )}\leq \frac{R}{4\sqrt{3}},  \label{M1}
\end{equation}%
and 
\begin{equation}
\mathcal{M}_{2}:=\int_{T^{\star }}^{t}\left\Vert (t-\sigma )^{\alpha
-2}E_{\alpha ,\alpha -1}(-m(\cdot )(t-\sigma )^{\alpha })f(v(\cdot ,\sigma
))\right\Vert _{L^{2}(\Omega )}d\sigma \leq \frac{R}{4\sqrt{3}},  \label{M2}
\end{equation}%
and 
\begin{equation}
\mathcal{M}_{3}:=\int_{0}^{T^{\star }}\left\Vert \left[ (t-\sigma )^{\alpha
-2}-(T^{\star }-\sigma )^{\alpha -2}\right] E_{\alpha ,\alpha -2}(-m(\cdot
)(t-\sigma )^{\alpha })f(u(\cdot ,\sigma ))\right\Vert _{L^{2}(\Omega
)}d\sigma \leq \frac{R}{4\sqrt{3}},  \label{M3}
\end{equation}%
and 
\begin{align}
\mathcal{M}_{4}:=\int_{0}^{T^{\star }}& \left\Vert (T^{\star }-\sigma
)^{\alpha -2}\left[ E_{\alpha ,\alpha -1}(-m(\cdot )(t-\sigma )^{\alpha
})-E_{\alpha ,\alpha -1}(-m(\cdot )(T^{\star }-\sigma )^{\alpha })\right]
f(u(\cdot ,\sigma ))\right\Vert _{L^{2}(\Omega )}d\sigma  \notag \\
\leq & \frac{R}{4\sqrt{3}}.  \label{M4}
\end{align}%
Inequalities \eqref{M1}-\eqref{M4} imply that 
\begin{equation}
\Vert \Phi (v(\cdot ))^{\prime }(t)-\partial _{t}u(\cdot ,T^{\star })\Vert
_{L^{2}(\Omega )}^{2}\leq \frac{R^{2}}{3},  \label{M5-0}
\end{equation}%
so that 
\begin{equation}
\textcolor{black}{(t-T^{\star})^{2s}}\Vert \Phi (v(\cdot ))^{\prime
}(t)-\partial _{t}u(\cdot ,T^{\star })\Vert _{L^{2}(\Omega )}^{2}\leq %
\textcolor{black}{\tau^{2s}}\frac{R^{2}}{3}\leq \frac{R^{2}}{3}.  \label{N5}
\end{equation}%
Next, using \eqref{M5-0}, we get 
\begin{equation*}
\left( g_{2-\alpha }\ast |\Vert \Phi (v(\cdot ))^{\prime }(t)-\partial
_{t}u(\cdot ,T^{\star })\Vert _{L^{2}(\Omega )}\Vert ^{2}\right) (t)\leq 
\frac{R^{2}}{3}\tau ^{2-\alpha }\leq \frac{R^{2}}{3}.
\end{equation*}%
It follows from \eqref{N5}-\eqref{M5-0} that there exists $\tau >0$ small such
that for every $t\in \lbrack T^{\star },T^{\star }+\tau ]$, 
\begin{align*}
& \Vert \Phi (v)(\cdot ,t)-u(\cdot ,T^{\star })\Vert _{V_{1/2}}^{2}+%
\textcolor{black}{(t-T^{\star})^{2s}}\Vert \Phi (v)^{\prime }(\cdot
,t)-\partial _{t}u(\cdot ,T^{\star })\Vert _{L^{2}(\Omega )}^{2} \\
& +g_{2-\alpha }\ast \Vert \Phi (v(\cdot ))^{\prime }(t)-\partial
_{t}u(\cdot ,T^{\star })\Vert _{L^{2}(\Omega )}^{2}(t)\leq R^{2}.
\end{align*}%
We have shown that $\Phi $ maps $\mathcal{K}_{s,T^{\star },R}$ into $%
\mathcal{K}_{s,T^{\star },R}$. Next, we show that $\Phi $ is a contraction
on $\mathcal{K}_{s,T^{\star },R}$. Let $v,w\in \mathcal{K}_{s,T^{\star },R}$%
. Then 
\begin{equation*}
\Phi (v(\cdot ))(t)-\Phi (w(\cdot ))(t)=\int_{0}^{t}(f(v(\cdot
,s))-f(w(\cdot ,s)))(t-s)^{\alpha -1}E_{\alpha ,\alpha }(-m(\cdot
)(t-s)^{\alpha })ds.
\end{equation*}%
If $t\in \lbrack 0,T^{\star }]$, then it follows that%
\begin{equation*}
\Vert \Phi (v)-\Phi (w)\Vert _{s,T^{\star }}\leq C(R^{\star })^{r-1}\left(
(T^{\star })^{\alpha }+(T^{\star })^{2(s+\alpha -1)}+(T^{\star })^{\alpha
+2}\right) ^{1/2}\Vert v-w\Vert _{C([0,T^{\star }];V_{1/2})}.
\end{equation*}%
If $t\in \lbrack T^{\star },T^{\star }+\tau ]$, then proceeding as in the
contraction estimates, we get that there is a constant $C>0$ such that%
\begin{align*}
& \Vert \Phi (v)\left( t\right) -\Phi (w)\left( t\right) \Vert _{V_{1/2}} \\
& \leq C\int_{T^{\star }}^{T^{\star }+\tau }(t-\tau )^{\frac{\alpha }{2}%
-1}\Vert v(\cdot ,\tau )-w(\cdot ,\tau )\Vert _{V_{1/2}}\left( \Vert v(\cdot
,\tau )\Vert _{V_{1/2}}+\Vert w(\cdot ,\tau )\Vert _{V_{1/2}}\right)
^{r-1}d\tau \\
& \leq C\Vert v-w\Vert _{C([T^{\star },T^{\star }+\tau
];V_{1/2})}R^{r-1}\tau ^{\frac{\alpha }{2}}.
\end{align*}%
In a similar way, we have that there is a constant $C>0$ such that 
\begin{equation}\label{EE2}
\Vert \Phi (v)^{\prime }(\cdot ,t)-\Phi (w)^{\prime }(\cdot ,t)\Vert
_{L^{2}(\Omega )}\leq C\Vert v-w\Vert _{C([T^{\star },T^{\star }+\tau
];V_{1/2})}R^{r-1}\tau ^{\alpha -1}.
\end{equation}%
This in turn implies that 
\begin{equation*}
\textcolor{black}{(t-T^{\star})^{s}}\Vert \Phi (v)^{\prime }(\cdot ,t)-\Phi
(w)^{\prime }(\cdot ,t)\Vert _{L^{2}(\Omega )}\leq C\Vert v-w\Vert
_{C([T^{\star },T^{\star }+\tau ];V_{1/2})}R^{r-1}\tau ^{\textcolor{black}{s}%
+\alpha -1}.
\end{equation*}%
Next, by \eqref{EE2}, we get 
\begin{equation*}
g_{2-\alpha }\ast \Vert (\Phi (v)-\Phi (w))^{\prime }\Vert _{L^{2}(\Omega
)}^{2}(t)\leq C\Vert v-w\Vert _{C([T^{\star },T^{\star }+\tau
];V_{1/2})}^{2}R^{2(r-1)}\tau ^{\alpha }.
\end{equation*}%
It follows from the preceding estimates that there is a constant $C>0$ such
that 
\begin{equation*}
\Vert \Phi (v)-\Phi (w)\Vert _{s,T^{\ast }}\leq CR^{r-1}(2\tau ^{\alpha
}+\tau ^{2(\textcolor{black}{s}+\alpha -1)})^{1/2}\Vert v-w\Vert _{s,T^{\ast
}}.
\end{equation*}%
Then choosing $\tau >0$ small enough so that%
\begin{equation*}
CR^{r-1}(2\tau ^{\alpha }+\tau ^{2(\textcolor{black}{s}+\alpha -1)})^{1/2}<1,
\end{equation*}%
we deduce once again that $\Phi $ is a contraction on $\mathcal{K}%
_{s,T^{\star },R}$. Hence, $\Phi $ has a unique fixed point $v$ on $\mathcal{%
K}_{s,T^{\star },R}$. This completes the extension argument.

\textbf{Step 3 (Blowup-vs Global solution).}To complete the proof of the
main theorem we need the following statement whose proof is standard and is
omitted for the sake of brevity.

\begin{lemma}
\label{lem-34} Let $T\in (0,\infty )$ and $u:\mathcal{H}\times \lbrack
0,T)\rightarrow L^{2}(\Omega )$ be such that $u(\xi ,\cdot )$ is
continuously differentiable for a.e. $\xi \in \mathcal{H}$ with 
\begin{equation*}
\sup_{t\in \lbrack 0,T)}\left( \Vert u(\cdot ,t)\Vert
_{V_{1/2}}^{2}+t^{2s}\Vert \partial _{t}u(\cdot ,t)\Vert _{L^{2}(\Omega
)}^{2}+g_{2-\alpha }\ast \Vert \partial _{t}u(\cdot ,t)\Vert _{L^{2}(\Omega
)}^{2}(t)\right) <\infty .
\end{equation*}
Let%
\begin{equation*}
\mathbb{E}(\cdot ,t):=t^{\alpha -1}E_{\alpha ,\alpha }(-m(\cdot )t^{\alpha
})\;\mbox{
and }\;\mathbb{E}^{\prime }(\cdot ,t):=t^{\alpha -2}E_{\alpha ,\alpha
-1}(-m(\cdot )t^{\alpha }).
\end{equation*}%
Let $t_{n}\in \lbrack 0,T)$ be a sequence such that $\lim_{n\rightarrow
\infty }t_{n}=T$. Then 
\begin{equation}
\lim_{n\rightarrow \infty }\int_{0}^{t_{n}}\left\Vert \left[ \mathbb{E}%
(t_{n}-\sigma )-\mathbb{E}(T-\sigma )\right] f(u(\cdot ,\sigma ))\right\Vert
_{V_{1/2}}d\sigma =0,  \label{331}
\end{equation}%
and 
\begin{equation}
\lim_{n\rightarrow \infty }\int_{0}^{t_{n}}\left\Vert \left[ \mathbb{E}%
^{\prime }(t_{n}-\sigma )-\mathbb{E}^{\prime }(T-\sigma )\right] f(u(\cdot
,\sigma ))\right\Vert _{L^{2}(\Omega )}d\sigma =0.  \label{332}
\end{equation}
\end{lemma}

Let now 
\begin{equation*}
\mathcal{T}:=\Big\{T\in \lbrack 0,\infty ):\;\exists \;u:[0,T]\rightarrow
L^{2}(\Omega )\;\mbox{
unique local solution to \eqref{EQ-NL} in }\;[0,T]\Big\},
\end{equation*}%
and set $T_{\max }:=\sup \mathcal{T}$. Then we have a continuously
differentiable function (in the second variable) $u:\Omega\times
\lbrack 0,T_{\max })\rightarrow L^{2}(\Omega)$ which is the local
solution to \eqref{EQ-NL} on $[0,T_{\max })$. If $T_{\max }=\infty $, then $%
u $ is a global solution. Now if $T_{\max }<\infty $ we shall show that we
have \eqref{t-maximal}. Assume that there exists $K_{0}<\infty $ such that 
\begin{equation}
\Vert u(\cdot ,t)\Vert _{V_{1/2}}^{2}+t^{2s}\Vert \partial _{t}u(\cdot
,t)\Vert _{L^{2}(\Omega )}^{2}+g_{2-\alpha }\ast \Vert \partial _{t}u(\cdot
,t)\Vert _{L^{2}(\Omega )}^{2}(t)\leq K_{0}^{2},\;\;\forall \;t\in \lbrack
0,T_{\max }).  \label{AS-bd}
\end{equation}%
Let $(t_{n})\subset \lbrack 0,T_{\max })$ be a sequence that converges to $%
T_{\max }$. Let $t_{n}>t_{m}$ and 
\begin{equation*}
K:=\sup_{t\in \lbrack 0,T_{\max })}\Vert f(u(\cdot ,t))\Vert _{L^{2}(\Omega
)}.
\end{equation*}%
Then using the assumption \eqref{AS-bd}, we get from Lemma \ref{lem-34} that 
\begin{equation*}
\left\Vert \int_{t_{m}}^{t_{n}}\mathbb{E}(T_{\max }-\sigma )f(u(\cdot
,\sigma ))d\sigma \right\Vert _{V_{1/2}}\leq \frac{2CK}{\alpha }\left[
(T_{\max }-t_{n})^{\frac{\alpha }{2}}-(T_{\max }-t_{m})^{\frac{\alpha }{2}}%
\right] \rightarrow 0,
\end{equation*}%
$\mbox{as }n,m\rightarrow \infty .$

We use the notations of $S_{j}$ and $S_{j}^{\prime }$ ($j=1,2,3$) given in %
\eqref{S12}- \eqref{S3p}. Then, by Lemma \ref{lem-34} we get 
\begin{equation*}
\Vert u(\cdot ,t_{n})-u(\cdot ,t_{m})\Vert _{V_{1/2}}\rightarrow 0\quad %
\mbox{as }n,m\rightarrow \infty ,
\end{equation*}%
Analogously, for $t_{n}>t_{m}$ we have that 
\begin{equation*}
\left\Vert \int_{t_{m}}^{t_{n}}\mathbb{E}^{\prime }(T_{\max }-\sigma
)f(u(\cdot ,\sigma ))d\sigma \right\Vert _{L^{2}(\Omega )}\leq \frac{CK}{%
\alpha -1}\left[ (T_{\max }-t_{n})^{\alpha -1}-(T_{\max }-t_{m})^{\alpha -1}%
\right] \rightarrow 0,
\end{equation*}%
$\mbox{as }n,m\rightarrow \infty $. Thus, by Lemma \ref{lem-34}, we obtain%
\begin{equation*}
\Vert \partial _{t}u(\cdot ,t_{n})-\partial _{t}u(\cdot ,t_{m})\Vert
_{L^{2}(\Omega )}\rightarrow 0,
\end{equation*}%
as $n,m\rightarrow \infty $. In particular, 
\begin{equation*}
t^{s}\Vert \partial _{t}u(\cdot ,t_{n})-\partial _{t}u(\cdot ,t_{m})\Vert
_{L^{2}(\Omega )}\rightarrow 0,\quad \mbox{as}\quad m,n\rightarrow \infty .
\end{equation*}%
Next, for $\sigma <t_{n}<T_{\max }$,%
\begin{align*}
\Vert \partial _{\sigma }u(\cdot ,\sigma )\Vert _{L^{2}(\Omega )}& \leq
C\left( \sigma ^{\frac{\alpha }{2}-1}\Vert u_{0}\Vert _{V_{1/2}}+\Vert
u_{1}\Vert _{L^{2}(\Omega )}+K\sigma ^{\alpha -1}\right) \\
& \leq C\left( \sigma ^{\frac{\alpha }{2}-1}\Vert u_{0}\Vert
_{V_{1/2}}+\Vert u_{1}\Vert _{L^{2}(\Omega )}+K(T_{\max })^{\alpha
-1}\right) .
\end{align*}%
Proceeding in a similar way as in \eqref{C3-2} (substituting $T^{\star }$ by 
$T_{\max }$ and $(R^{\star })^{r}$ by $K$), we have 
\begin{align*}
g_{2-\alpha }\ast \Vert \partial _{t}u\Vert _{L^{2}(\Omega )}^{2}(t)& \leq
C^{2}\left( \Vert u_{0}(\cdot )\Vert _{V_{1/2}}^{2}+t^{2-\alpha }\Vert
u_{1}(\cdot )\Vert _{L^{2}(\Omega )}^{2}+t^{2-\alpha }(T_{\max })^{2(\alpha
-1)}K^{2}\right. \\
& +\left. t^{1-\frac{\alpha }{2}}\Vert u_{0}(\cdot )\Vert _{V_{1/2}}\Vert
u_{1}(\cdot )\Vert _{L^{2}(\Omega )}+t^{1-\frac{\alpha }{2}}\Vert
u_{0}(\cdot )\Vert _{V_{1/2}}(T_{\max })^{\alpha -1}K\right. \\
& +\left. t^{2-\alpha }\Vert u_{1}(\cdot )\Vert _{L^{2}(\Omega )}(T_{\max
})^{\alpha -1}K\right) .
\end{align*}%
This in turn implies 
\begin{align*}
& g_{2-\alpha }\ast \Vert \partial _{t}u\Vert _{L^{2}(\Omega
)}^{2}(t_{n})-g_{2-\alpha }\ast \Vert \partial _{t}u\Vert _{L^{2}(\Omega
)}^{2}(t_{m}) \\
& \leq (t_{n}^{2-\alpha }-t_{m}^{2-\alpha })\left( \Vert u_{1}(\cdot )\Vert
_{L^{2}(\Omega )}^{2}+(T_{\max })^{2(\alpha -1)}K^{2}+\Vert u_{1}(\cdot
)\Vert _{L^{2}(\Omega )}(T_{\max })^{\alpha -1}K\right) \\
& +(t_{n}^{1-\frac{\alpha }{2}}-t_{m}^{1-\frac{\alpha }{2}})\left( \Vert
u_{0}(\cdot )\Vert _{V_{1/2}}\Vert u_{1}(\cdot )\Vert _{L^{2}(\Omega
)}+\Vert u_{0}(\cdot )\Vert _{V_{1/2}}(T_{\max })^{\alpha -1}K\right)
\rightarrow 0\quad \mbox{as}\quad m,n\rightarrow \infty .
\end{align*}%
It follows that $(u(\cdot ,t_{n}))_{n\in {\mathbb{N}}}$, $(\partial
_{t}u(\cdot ,t_{n}))_{n\in {\mathbb{N}}}$ and $g_{2-\alpha }\ast \Vert
\partial _{t}u\Vert _{L^{2}(\Omega )}^{2}(t_{n})$ are Cauchy sequences, and
therefore, they have limits $u_{T_{\max }}\left( \cdot ,t\right) $, $%
\partial _{t}u_{T_{\max }}\left( \cdot ,t\right) $ and $g_{2-\alpha }\ast
\Vert \partial _{t}u_{T_{\max }}\Vert _{L^{2}(\Omega )}^{2}$ such that $%
u_{T_{\max }}\left( \cdot ,t\right) \in V_{1/2}$ and $\partial
_{t}u_{T_{\max }}\left( \cdot ,t\right) \in L^{2}(\Omega )$. Thus, we can
extend $u$ over $[0,T_{\max }]$ to obtain the equality 
\begin{equation*}
u(\xi ,t)={S}_{1}(t)u_{0}(\xi )+{S}_{2}(t)u_{1}(\xi )+({S}_{3}\ast f(u(\xi
,\cdot ))(t),
\end{equation*}%
for all $t\in \lbrack 0,T_{\max }]$. By the extension argument above, we can
extend the solution to some larger interval. This is a contradiction with
the definition of $T_{\max }>0$. The proof is finished
\end{proof}

\section{Strong energy solutions for the linear case}

In this section we prove the existence  of strong solutions  in the linear case. The nonlinear problem will be considered in the next section.

\begin{definition}
\label{def-strong}
Let $\frac{1}{2}\leq \gamma\leq 1$ and $0\leq \widetilde{\gamma}\leq 1$. We say that a function $u$ is a strong solution of %
\eqref{EQ-LI} on $(0,T)$, for some $T>0$, if it is a weak solution in the
sense of Definition \ref{def-weak} and, additionally,%
\begin{equation}
\{\mathbb{D}_{t}^{\alpha }u,Au\}\in L^{l}(0,T;L^{2}\left( \Omega \right) )\,,%
\text{ }u\in W^{1,\zeta }\left( 0,T;V_{\widetilde{\gamma }}\right) \cap
C^{1}((0,T];V_{\widetilde{\gamma }})\cap C\left( \left[ 0,T\right]
;V_{\gamma }\right) ,  \label{reg-strong}
\end{equation}%
where $\zeta \in \left[ 1,\infty \right]$, is such that $\alpha \left(
\gamma -\widetilde{\gamma }\right) +\zeta ^{-1}>1,$ and%
\begin{equation*}
\left\{ 
\begin{array}{ll}
1\leq l<\min \{\frac{1}{\alpha \left( 1-\gamma \right) },\frac{1}{\alpha
\left( 1-\widetilde{\gamma }\right) -1}\}, & \text{if }\widetilde{\gamma }%
<1-\alpha ^{-1}, \\ 
1\leq l<\frac{1}{\alpha \left( 1-\gamma \right) }, & \text{if }\widetilde{%
\gamma }\geq 1-\alpha ^{-1}.%
\end{array}%
\right.
\end{equation*}%
In this case the main equation of \eqref{EQ-LI} is satisfied pointwise a.e.
in $\Omega \times \left( 0,T\right) .$
\end{definition}

\begin{remark}
When $l=1$ we recover  \cite[Definition 3.3]{AGKW}.  It means that our definition of strong solution is more general that the previous one.
\end{remark}

We have the following fundamental existence result for strong solutions.

\begin{theorem}
\label{theo-strong} Let $\frac{1}{2}\leq \gamma\leq 1$ and $0\leq \widetilde{\gamma}\leq 1$ and assume%
\begin{equation*}
f\in W^{1,\frac{p}{p-1}}(0,T;L^{2}(\Omega )),\text{ for }1\leq p\leq \infty .
\end{equation*}%
Then the following assertions hold.

\begin{enumerate}
\item[(a)] If $u_{0}\in V_{\gamma }$ and $u_{1}\in V_{\widetilde{\gamma }}$,
there is a constant $C>0$ such that for every $t\in (0,T]$,%
\begin{align}
& \left\Vert \mathbb{D}_{t}^{\alpha }u(\cdot ,t)\right\Vert _{L^{2}\left(
\Omega \right) }+\left\Vert Au(\cdot ,t)\right\Vert _{L^{2}\left( \Omega
\right) }  \label{est-mathbbd} \\
\leq & C\left( t^{\alpha \left( \gamma -1\right) }\Vert u_{0}\Vert
_{V_{\gamma }}+t^{1-\alpha \left( 1-\widetilde{\gamma }\right) }\Vert
u_{1}\Vert _{V_{\widetilde{\gamma }}}\right)  \notag \\
& +C\left( \Vert f(\cdot ,t)\Vert _{L^{2}(\Omega )}+\Vert f(\cdot ,0)\Vert
_{L^{2}(\Omega )}+Ct^{\frac{1}{p}}\Vert \partial _{t}f\Vert _{L^{p/\left(
p-1\right) }((0,T);L^{2}(\Omega ))}\right)  \notag
\end{align}%
and, if $\gamma -\widetilde{\gamma }\in (0,1],$ we have
\begin{align}
\left\Vert \partial _{t}u(\cdot ,t)\right\Vert _{V_{\widetilde{\gamma }}}&
\leq C\left( t^{ \alpha \left( \gamma -\widetilde{\gamma }\right)
-1 }\Vert u_{0}\Vert _{V_{\gamma }}+\Vert u_{1}\Vert _{V_{\widetilde{%
\gamma }}}\right.  \notag \\
&\left.+t^{\frac{1}{p}+\alpha \left( 1-\widetilde{\gamma }\right) -2}\Vert
f\Vert _{L^{p/\left( p-1\right) }((0,T);L^{2}(\Omega ))}\right) ,\text{ if }%
\frac{1}{p}+\alpha \left( 1-\widetilde{\gamma }\right) -2>0.
\label{est-mathbbd2} \\
\Vert u(\cdot ,t)\Vert _{V_{\gamma }}& \leq C\left( \Vert u_{0}\Vert
_{V_{\gamma }}+t^{1-\alpha \left( \gamma -\widetilde{\gamma }\right) }\Vert
u_{1}\Vert _{V_{\widetilde{\gamma }}}+t^{\frac{1}{p}+\alpha \left( 1-\gamma
\right) -1}\Vert f\Vert _{L^{p/\left( p-1\right) }((0,T);L^{2}(\Omega
))}\right) .  \label{est-mathbbd3}
\end{align}%

In particular, the system \eqref{EQ-LI} has a unique strong solution in the
sense of Definition \ref{def-strong}, that is given exactly by %
\eqref{sol-spec}.

\item[(b)] Let $0\leq \theta <1$ and assume that $u_{0}\in V_{\gamma },$ $%
u_{1}\in V_{\widetilde{\gamma }}$ with both $\gamma -\theta ,$ $\widetilde{%
\gamma }-\theta \in \left[ 0,1\right] $. Then the strong solution of %
\eqref{EQ-LI} also satisfies the following estimate:%
\begin{align}
\left\Vert \partial _{t}^{2}u(\cdot ,t)\right\Vert _{V_{\theta }}\leq &
C\left( t^{\alpha (\gamma -\theta )-2}\Vert u_{0}\Vert _{V_{\gamma
}}+t^{\alpha (\widetilde{\gamma }-\theta -1)}\Vert u_{1}\Vert _{V_{%
\widetilde{\gamma }}}\right. \notag \\
& \qquad \left. +t^{\frac{1}{p}+\alpha \left( 1-\theta \right) -2}\Vert
\partial _{t}f\Vert _{L^{p/\left( p-1\right) }(0,T;L^{2}(\Omega
))}+t^{\alpha \left( 1-\theta \right) -2}\Vert f(\cdot ,0)\Vert
_{L^{2}(\Omega )}\right) .   \label{est-strong}
\end{align}%
\newline

\item[(c)] The initial conditions are also satisfied in the following
(strong) sense: 
\begin{equation}  \label{ini-strong}
\lim_{t\rightarrow 0^{+}}\left\Vert u(\cdot ,t)-u_{0}\right\Vert _{V_{\sigma
}}=0,\lim_{t\rightarrow 0^{+}}\left\Vert \partial _{t}u(\cdot
,t)-u_{1}\right\Vert _{V_{\beta }}=0,
\end{equation}%
for $\sigma \in \lbrack 0,\gamma ],$ $\sigma -\widetilde{\gamma }\in \left[
0,1\right] $ with $1-\alpha \left( \sigma -\widetilde{\gamma }\right) >0$
only if $\sigma >\widetilde{\gamma },$ \textbf{and} $\beta \in \lbrack 0,%
\widetilde{\gamma }]$ is such that $\gamma -\beta \in \left[ 0,1\right] $
and $\gamma -\beta >1/\alpha .$
\end{enumerate}
\end{theorem}

\begin{proof}
(a) First, since $f\in W^{1,p/\left( p-1\right) }(0,T;L^{2}(\Omega ))$, then 
$f\in C([0,T];L^{2}(\Omega ))\subset C([0,T];V_{-\gamma })$. Second, let $%
u_{0}\in V_{\gamma }$, $u_{1}\in V_{\widetilde{\gamma }}$ and $u$ the weak
solution of \eqref{EQ-LI}. Recall that $\mathbb{D}^{\alpha }u=-Au+f$ is
given by \eqref{dt-al}. Using \eqref{IN-L2}, we get the following estimates: 
\begin{equation}\label{D2-SS0}
\Vert m(\cdot )E_{\alpha ,1}(-m(\cdot )t^{\alpha })u_{0}(\cdot )\Vert
_{L^{2}(\Omega )}^{2}\leq C^{2}t^{2\alpha \left( \gamma -1\right) }\Vert
u_{0}\Vert _{V_{\gamma }},
\end{equation}%
where we require that $\gamma >1-\frac{1}{\alpha }$ so that the coefficient $%
t^{\alpha \left( \gamma -1\right) }$ is (at least) integrable near the origin%
\footnote{%
But this holds since $\gamma \geq \frac{1}{2}>1-\frac{1}{\alpha }.$}.
Similarly, we get%
\begin{equation}
\left\Vert m(\cdot )tE_{\alpha ,2}(-m(\cdot )t^{\alpha })u_{1}(\cdot
)\right\Vert _{L^{2}(\Omega )}\leq Ct^{1-\alpha \left( 1-\widetilde{\gamma }%
\right) }\Vert u_{1}\Vert _{V_{\widetilde{\gamma }}},  \label{D2-SS}
\end{equation}%
provided that $\widetilde{\gamma }\in \left[ 0,1\right] $. Using %
\eqref{Est-MLF2} and integrating by parts, we deduce 
\begin{align}
& \left\Vert \int_{0}^{t}m(\cdot )(t-\tau )^{\alpha -1}E_{\alpha ,\alpha
}(-m(\cdot )(t-\tau )^{\alpha })f(\cdot ,\tau )d\tau \right\Vert
_{L^{2}(\Omega )}  \notag\\
=& \Vert -f(\cdot ,t)+E_{\alpha ,1}(-m(\cdot )t^{\alpha })f(\cdot ,0)  \notag
\\
& +\int_{0}^{t}f^{\prime }(\cdot ,\tau )E_{\alpha ,1}(-m(\cdot )(t-\tau
)^{\alpha })d\tau \Vert _{L^{2}(\Omega )}  \notag \\
\leq & \Vert f(\cdot ,t)\Vert _{L^{2}(\Omega )}+C\Vert f(\cdot ,0)\Vert
_{L^{2}(\Omega )}+Ct^{\frac{1}{p}}\Vert \partial _{t}f\Vert _{L^{p/\left(
p-1\right) }((0,T);L^{2}(\Omega ))},  \label{D3-SS} 
\end{align}%
for $p\in \left[ 1,\infty \right] $. It follows from \eqref{D2-SS0}, %
\eqref{D2-SS}, \eqref{D3-SS} that $\mathbb{D}_{t}^{\alpha }u\in
L^{l}((0,T);L^{2}(\Omega ))$ since%
\begin{equation*}
f\in W^{1,p/\left( p-1\right) }\left( (0,T);L^{2}(\Omega )\right) .
\end{equation*}%
The estimate \eqref{est-mathbbd} then follows from \eqref{D2-SS0}, %
\eqref{D2-SS}, \eqref{D3-SS} on account of \eqref{dt-al}. Since $Au=-\mathbb{%
D}_{t}^{\alpha }u+f$, we have also shown that $u\in L^{l}((0,T);D(A))$.
Thus, $u$ is a unique strong solution of \eqref{EQ-LI}. Next, we exploit the
formula (\ref{D1-1}) to derive the estimate (\ref{est-mathbbd2}). By token
of (\ref{IN-L1}), we first observe that%
\begin{equation*}
\Vert S_{1}^{\prime }(t)u_{0}(\cdot )\Vert _{V_{\widetilde{\gamma }%
}}^{2}\leq C^{2}t^{2\left( \alpha \left( \gamma -\widetilde{\gamma }\right)
-1\right) }\Vert u_{0}\Vert _{V_{\gamma }}^{2},
\end{equation*}%
provided that $0<\gamma -\widetilde{\gamma }\leq 1$. Moreover,%
\begin{equation*}
\Vert S_{2}^{\prime }(t)u_{1}(\cdot )\Vert _{V_{\widetilde{\gamma }%
}}^{2}\leq C^{2}\Vert u_{1}\Vert _{V_{\widetilde{\gamma }}}^{2},
\end{equation*}%
and%
\begin{align*}
\Vert (S_{3}^{\prime }\ast f(\xi ,\cdot ))(t)\Vert _{V_{\widetilde{\gamma }%
}}& \leq C\int_{0}^{t}(t-\tau )^{\alpha -2-\alpha \widetilde{\gamma }}\Vert
f(\cdot ,\tau )\Vert _{L^{2}(\Omega )}d\tau \\
& \leq Ct^{\frac{1}{p}+\alpha \left( 1-\widetilde{\gamma }\right) -2}\Vert
f\Vert _{L^{p/\left( p-1\right) }((0,T);L^{2}(\Omega ))},
\end{align*}%
provided that $\frac{1}{p}+\alpha \left( 1-\widetilde{\gamma }\right) >2,$
for $p\in \left[ 1,\infty \right] .$ Collecting all these estimates together
yields (\ref{est-mathbbd2}). The estimate (\ref{est-mathbbd3}) follows from
the proof of Theorem \ref{theo-weak}\ with some minor (obvious)
modifications.

(b) Let $0\leq \theta <1$ and assume that $u_{0}\in V_{\gamma }$ and $%
u_{1}\in V_{\widetilde{\gamma }}$. Let $u$ be the strong solution of %
\eqref{EQ-LI}. From a simple calculation, for a.e. $t\in (0,T)$, we have%
\begin{align}
\partial _{t}^{2}u(\cdot ,t)=& m(\cdot )t^{\alpha -2}E_{\alpha ,\alpha
-1}(-m(\cdot )t^{\alpha })u_{0}(\cdot )-m(\cdot )E_{\alpha ,\alpha
}(-m(\cdot )t^{\alpha })u_{1}(\cdot ) \notag \\
& +\int_{0}^{t}(t-\tau )^{\alpha -2}E_{\alpha ,\alpha -1}(-m(\cdot )(t-\tau
)^{\alpha })f^{\prime }(\cdot ,\tau )d\tau  \notag \\
& +t^{\alpha -2}E_{\alpha ,\alpha -1}(-m(\cdot )t^{\alpha })f(\cdot ,0) 
\notag \\
=:& S_{1}^{\prime \prime }(t)u_{0}(\cdot )+S_{2}^{\prime \prime
}(t)u_{1}(\cdot )+(S_{3}^{\prime }\ast f^{^{\prime }}(\xi ,\cdot
))(t)+t^{\alpha -2}E_{\alpha ,\alpha -1}(-m(\cdot )t^{\alpha })f(\cdot ,0).   \label{DD-2}
\end{align}%
Note that 
\begin{equation*}
\Vert S_{1}^{\prime \prime }(t)u_{0}(\cdot )\Vert _{V_{\theta
}}^{2}=\int_{\Omega }\left\vert (m(\xi ))^{1-\gamma +\theta }t^{\alpha
-2}E_{\alpha ,\alpha -1}(-m(\xi )t^{\alpha })(m(\xi ))^{\gamma }u_{0}(\xi
)\right\vert ^{2}d\mu .
\end{equation*}%
Since $0\leq \theta <1$, we have from \eqref{IN-L1} that%
\begin{align*}
\Vert S_{1}^{\prime \prime }(t)u_{0}(\cdot )\Vert _{V_{\theta }}& \leq
Ct^{\left( \alpha -2-\alpha \left( 1-\gamma +\theta \right) \right) }\left(
\int_{\Omega }\left\vert (m(\xi ))^{\gamma }u_{0}(\xi )\right\vert ^{2}d\mu
\right) ^{\frac{1}{2}} \\
& =Ct^{\alpha (\gamma -\theta )-2}\Vert u_{0}\Vert _{V_{\gamma }},
\end{align*}%
provided that $\gamma -\theta \in \left[ 0,1\right] $. Next,%
\begin{equation*}
\Vert S_{2}^{\prime \prime }(t)u_{1}(\cdot )\Vert _{V_{\theta
}}^{2}=\int_{\Omega }\left\vert (m(\xi ))^{1-\widetilde{\gamma }+\theta
}E_{\alpha ,\alpha }(-m(\xi )t^{\alpha })(m(\xi ))^{\widetilde{\gamma }%
}u_{1}(\xi )\right\vert ^{2}d\mu ,
\end{equation*}%
and since $\widetilde{\gamma }-\theta \in \left[ 0,1\right] $, using %
\eqref{IN-L2} we deduce that%
\begin{equation*}
\Vert S_{2}^{\prime \prime }(t)u_{1}(\cdot )\Vert _{V_{\theta }}\leq
Ct^{\alpha (\widetilde{\gamma }-\theta -1)}\Vert u_{1}\Vert _{V_{\widetilde{%
\gamma }}}.
\end{equation*}%
For the third term, we have%
\begin{align*}
& \Vert (S_{3}^{\prime }\ast f^{\prime }(\xi ,\cdot ))(t)\Vert _{V_{\theta }}
\\
& \leq \int_{0}^{t}\left( \int_{\Omega }|(t-\tau )^{\alpha -2}(m(\xi
))^{\theta }E_{\alpha ,\alpha -1}(m(\xi )(t-\tau )^{\alpha })f^{\prime }(\xi
,\tau )|^{2}d\mu \right) ^{1/2}d\tau .
\end{align*}%
Using \eqref{IN-L1}, the H\"{o}lder inequality, the fact that $\frac{1}{p}%
+\alpha \left( 1-\theta \right) >2$ and $\theta\leq \tilde{\gamma}$, we find 
\begin{equation*}
\Vert (S_{3}^{\prime }\ast f^{\prime }(\xi ,\cdot ))(t)\Vert _{V_{\theta
}}\leq Ct^{\frac{1}{p}+\alpha \left( 1-\theta \right) -2}\Vert \partial
_{t}f\Vert _{L^{p/\left( p-1\right) }((0,T);L^{2}(\Omega ))}.
\end{equation*}%
For the fourth term, using once again \eqref{IN-L1}, we get 
\begin{equation*}
\left\Vert t^{\alpha -2}E_{\alpha ,\alpha -1}(-m(\cdot )t^{\alpha })f(\cdot
,0)\right\Vert _{V_{\theta }}\leq Ct^{\alpha \left( 1-\theta \right)
-2}\Vert f(\cdot ,0)\Vert _{L^{2}(\Omega )}.
\end{equation*}%

(c) Finally, let us verify the initial conditions. For this argument, we
shall exploit once again (\ref{diff-u0}) and (\ref{diff-u1}). Now, since 
$1-\frac{1}{\alpha }<\frac{1}{2}\leq \gamma$, then the condition $1\leq p<%
\frac{1}{1-\alpha (1-\gamma )}$ make sense as before. Thus, from \eqref{S3},
we get 
\begin{equation*}
\left\Vert (S_{3}\ast f(\xi ,\cdot ))(t)\right\Vert _{V_{\gamma }}\leq
Ct^{\alpha \left( 1-\gamma \right) -1+\frac{1}{p}}\Vert f\Vert _{L^{p/\left(
p-1\right) }((0,T);L^{2}(\Omega ))}\rightarrow 0,\text{ as }%
t\rightarrow 0^{+}.
\end{equation*}

Moreover, for $\sigma \leq \gamma ,$ we have that
\begin{equation*}
\Vert S_{2}(t)u_{1}(\cdot )\Vert _{V_{\sigma }}\leq Ct^{1-\alpha \left(
\sigma -\widetilde{\gamma }\right) }\Vert u_{1}\Vert _{V_{\gamma
}},\rightarrow 0,\quad t\rightarrow 0^{+}.
\end{equation*}%
Furthermore, for $u_{0}\in V_{\gamma }$, $E_{\alpha ,1}(m(\xi )t^{\alpha
})u_{0}(\xi )\rightarrow u_{0}(\xi )$ as $t\rightarrow 0^{+}$, and $\sigma
<\gamma ,$%
\begin{equation*}
|(m(\xi ))^{\sigma }(E_{\alpha ,1}(-m(\xi )t^{\alpha })u_{0}(\xi )-u_{0}(\xi
))|^{2}\leq 4C^{2}m_{0}^{2(\sigma -\gamma )}|(m(\xi ))^{\gamma }u_{0}(\xi
))|^{2}\in L^{1}(\Omega ),
\end{equation*}%
one has from the Dominated Convergence Theorem that 
\begin{align*}
& \left\Vert E_{\alpha ,1}(-m(\cdot )t^{\alpha })u_{0}(\cdot )-u_{0}(\cdot
)\right\Vert _{V_{\sigma }} \\
& =\left( \int_{\Omega }|(m(\xi ))^{\sigma }(E_{\alpha ,1}(-m(\xi )t^{\alpha
})u_{0}(\xi )-u_{0}(\xi ))|^{2}d\mu \right) ^{1/2}\rightarrow 0,
\end{align*}%
as $t\rightarrow 0^{+}$. Then the first limit of (\ref{ini-strong}) is a
simple corollary of these estimates and the identity (\ref{diff-u0}).
Analogously, for the second statement of (\ref{ini-strong}) we may exploit
the formula (\ref{diff-u1}). From \eqref{IN-L1} and the implied condition $%
1\leq p<\frac{1}{2-\alpha \left( 1-\widetilde{\gamma }\right) }$ and $\beta
\leq \widetilde{\gamma }$, we get 
\begin{equation*}
\left\Vert (S_{3}^{\prime }\ast f(\xi ,\cdot ))(t)\right\Vert _{V_{\beta
}}\leq Ct^{\alpha \left( 1-\beta \right) -2+\frac{1}{p}}\Vert f\Vert
_{L^{p/\left( p-1\right) }((0,T);L^{2}(\Omega ))}\rightarrow 0\text{ as }%
t\rightarrow 0^{+}.
\end{equation*}%
In addition, we have that
\begin{equation*}
\Vert S_{1}^{\prime }(t)u_{0}(\cdot )\Vert _{V_{\beta }}^{2}=\int_{\Omega
}|(m(\xi ))^{1-\gamma +\beta }t^{\alpha -1}E_{\alpha ,1}(-m(\xi )t^{\alpha
})(m(\xi ))^{\gamma }u_{0}(\xi )|^{2}d\mu .
\end{equation*}%
Since $1-\gamma +\beta \in \left[ 0,1\right] $,  it follows that
\begin{equation}
\Vert S_{1}^{\prime }(t)u_{0}(\cdot )\Vert _{V_{\beta }}\leq Ct^{\alpha
\left( \gamma -\beta \right) -1}\Vert u_{0}\Vert _{V_{\gamma }}  \label{dd-3}
\end{equation}%
by the assumption $\alpha \left( \gamma -\beta \right) >1$, we get the
desired vanishing convergence as $t\rightarrow 0^{+},$ of the right hand
side of (\ref{dd-3}). Additionally, the Dominated Convergence Theorem shows%
\begin{align*}
& \left\Vert E_{\alpha ,1}(-m(\cdot )t^{\alpha })u_{1}(\cdot )-u_{1}(\cdot
)\right\Vert _{V_{\beta }} \\
& \leq m_{0}^{-\left( \widetilde{\gamma }-\beta \right) }\left( \int_{\Omega
}|E_{\alpha ,1}(-m(\xi )t^{\alpha })(m(\xi ))^{\widetilde{\gamma }}u_{1}(\xi
)-(m(\xi ))^{\widetilde{\gamma }}u_{1}(\xi )|^{2}d\mu \right) ^{1/2}
\end{align*}%
converges to 0, as $t \rightarrow 0^{+}.$ The proof of the theorem is
complete.
\end{proof}

\begin{remark}
A comparison with \cite[Theorem 3.4]{AGKW} lead us a better regularity by the range of the parameters $\gamma,\widetilde{\gamma}$.  This includes the case $\gamma=1/\alpha$ and $\widetilde{\gamma}=0$.
Moreover, 
If $\gamma>\theta+\frac{1}{\alpha}$, $\tilde{\gamma}>\theta+1-\frac{1}{\alpha}$ and $\theta<1-\frac{1}{\alpha}$, then $u\in W^{2,1}((0,T);V_{\theta})$.  Moreover, 
\begin{align}
\left\Vert \partial _{t}^{2}u\right\Vert _{L^1(0,T;V_{\theta })}\leq &
C\left( T^{\alpha (\gamma -\theta )-1}\Vert u_{0}\Vert _{V_{\gamma
}}+T^{\alpha (\widetilde{\gamma }-\theta -1)+1}\Vert u_{1}\Vert _{V_{%
\widetilde{\gamma }}}\right. \notag \\
& \qquad \left. +T^{\frac{1}{p}+\alpha \left( 1-\theta \right) -1}\Vert
\partial _{t}f\Vert _{L^{p/\left( p-1\right) }(0,T;L^{2}(\Omega
))}+T^{\alpha \left( 1-\theta \right) -1}\Vert f(\cdot ,0)\Vert
_{L^{2}(\Omega )}\right) .   \label{est-strong1}
\end{align}
In particular, if $\theta=0$, we recover the estimate (3.30) of  \cite[Theorem 3.3, part (b)]{AGKW}.
\end{remark}

\section{Strong energy solutions for the semilinear case}

From now on, we assume the same conditions for the parameters given in Definition \ref{def-strong}.\\

In order to show the existence of strong energy solutions to the semilinear
problem \eqref{EQ-NL}, we introduce the Banach space%
\begin{equation*}
X_{\delta _{1},\delta _{2},T}=\left\{ u\in C\left( \left[ 0,T\right]
;V_{1/2}\right) \cap C^{1}\left( [0,T];L^{2}\left( \Omega \right) \cap
C((0,T]:D(A)\right) :\left\Vert u\right\Vert _{\delta _{1},\delta
_{2},T}<\infty \right\} ,
\end{equation*}%
where the norm is defined as (for properly chosen $\delta _{1},\delta
_{2}>0, $ depending on the physical parameters)%
\begin{equation*}
\left\Vert u\right\Vert _{\delta _{1},\delta _{2},T}:=\sup_{t\in \left[ 0,T%
\right] }\left( \left\Vert u\left( \cdot ,t\right) \right\Vert _{V_{\gamma
}}+t^{\delta _{1}}\left\Vert \partial _{t}u\left( \cdot ,t\right)
\right\Vert _{L^{2}\left( \Omega \right) }+t^{\delta _{2}}\Vert \mathbb{D}%
_{t}^{\alpha }u(\cdot ,t)\Vert _{L^{2}\left( \Omega \right) }+t^{\delta
_{2}}\left\Vert Au(\cdot ,t)\right\Vert _{L^{2}\left( \Omega \right)
}\right) .
\end{equation*}%
Next, define the space $\mathcal{B}_{\delta _{1},\delta _{2},T,R}\subset X_{\delta
_{1},\delta _{2},T},$ where%
\begin{equation*}
\mathcal{B}_{\delta _{1},\delta _{2},T,R}:=\left\{ u\in X_{\delta
_{1},\delta _{2},T}:\left\Vert u\right\Vert _{\delta _{1},\delta _{2},T}\leq
R\text{ and }u\left( 0\right) =u_{0},\text{ }\partial _{t}u\left( 0\right)
=u_{1}\right\} .
\end{equation*}
The \textit{strong energy solution} is defined in a similar way that the weak energy solution for the semilinear case.


We will impose the following assumption on the nonlinearity $f$.

\begin{enumerate}
\item[(\textbf{Hf3})] There exist two monotone increasing (real-valued)
functions $Q_{1},Q_{2}\geqslant 0$ such that%
\begin{equation*}
f(0)=0\;\text{and }\left\vert f^{^{\prime }}\left( s\right) \right\vert \leq
Q_{1}\left( \left\vert s\right\vert \right) ,\text{ }\left\vert f\left(
s\right) \right\vert \leq Q_{2}\left( \left\vert s\right\vert \right) \text{%
, for all }s\in {\mathbb{R}}.
\end{equation*}
\end{enumerate}

and 
\begin{equation*}
V_{\gamma}\hookrightarrow L^{\infty}(\Omega ).
\end{equation*}

The main result of the section is the following.

\begin{theorem}
\label{main-weak2}Assume (\textbf{Hf3}). The problem (\ref{EQ-NL0}) has a
unique strong energy solution in $X_{\delta _{1},\delta _{2},T}$, for all $%
T<T_{\max }$, where either $T_{\max }=\infty $ or $T_{\max }<\infty $, and
in that case, 
\begin{equation}
\underset{t\rightarrow T_{\max }^{-}}{\lim \sup }\left( \left\Vert u\left(
\cdot ,t\right) \right\Vert _{V_{\gamma }}+t^{\delta _{1}}\left\Vert
\partial _{t}u\left( \cdot ,t\right) \right\Vert _{L^{2}\left( \Omega
\right) }+t^{\delta _{2}}\Vert \mathbb{D}_{t}^{\alpha }u(\cdot ,t)\Vert
_{L^{2}\left( \Omega \right) }+t^{\delta _{2}}\left\Vert Au(\cdot
,t)\right\Vert _{L^{2}\left( \Omega \right) }\right) =\infty .  \label{bl}
\end{equation}
\end{theorem}

\begin{proof}
Since $f$ is continuously differentiable, we have that $\Phi (u(\cdot ))(t)$
is continuously differentiable on $(0,T]$. We will show that by an
appropriate choice of $T,R>0$, $\Phi :\mathcal{B}_{\delta _{1},\delta
_{2},T,R}\rightarrow \mathcal{B}_{\delta _{1},\delta _{2},T,R}$ is a
contraction with respect to the metric induced by the norm $\left\Vert \cdot
\right\Vert _{\delta _{1},\delta _{2},T}$. The appropriate choice of $T,R>0$
will be specified below. We first show that $\Phi $ maps $\mathcal{B}%
_{\delta _{1},\delta _{2},T,R}$ into $\mathcal{B}_{\delta _{1},\delta
_{2},T,R}$. Indeed, let $u\in \mathcal{B}_{\delta _{1},\delta _{2},T,R}$.
Let 
\textcolor{black}{
\begin{align*}
\max\{0,1-\alpha\gamma\}\leq \delta_1<1\quad \mbox{and}\quad  \max\{\alpha(1-\gamma),\alpha-1\}\leq \delta_2.
\end{align*}} By (\textbf{Hf3},) we have that 
\begin{equation*}
\Vert w\Vert _{L^{\infty }(\Omega )}\leq C\Vert w\Vert _{V_{\gamma }}.
\end{equation*}%
Set 
\begin{equation*}
S_{1}(t)u_{0}(\xi ):=E_{\alpha ,1}(-m(\xi )t^{\alpha })u_{0}(\xi )\;%
\mbox{
and }\;S_{2}(t)u_{1}(\xi ):=tE_{\alpha ,2}(-m(\xi )t^{\alpha })u_{1}(\xi ),
\end{equation*}%
and 
\begin{equation*}
(S_{3}\ast f(u(\xi ,\cdot )))(t):=\int_{0}^{t}f(u(\cdot ,\tau ))(t-\tau
)^{\alpha -1}E_{\alpha ,\alpha }((-m(\cdot )(t-\tau )^{\alpha })d\tau
\end{equation*}%
so that 
\begin{equation*}
\Phi (u(\xi ,t))=S_{1}(t)u_{0}(\xi )+S_{2}(t)u_{1}(\xi )+(S_{3}\ast f(u(\xi
,\cdot )))(t),\quad t\in \lbrack 0,T],\,\,\xi \in \Omega .
\end{equation*}%
Next, by assumption (\textbf{Hf3}), \eqref{embed-supnorm} and the facts that 
$\mu(\Omega)<\infty $ and $u\in C([0,T],V_{\gamma })$, for every $t\in \lbrack
0,T]$, 
\begin{equation*}
\Vert f(u(\cdot ,t))\Vert _{L^{2}(\Omega )}\leq \left( \int_{\Omega
}(Q_{2}(C_{1}\Vert u(\cdot ,t)\Vert _{V_{\gamma }})^{2}\;d\mu \right)
^{1/2}\leq C_{f}Q_{2}(\Vert u(\cdot ,t)\Vert _{V_{\gamma }}),
\end{equation*}%
that is, 
\begin{equation}
\Vert f(u(\cdot ,t))\Vert _{L^{2}(\Omega )}\leq C_{f}Q_{2}\left( \Vert
u(\cdot ,t)\Vert _{V_{\gamma }}\right) ,  \label{C1-1}
\end{equation}%
for some $C_{f}>0$. 
Using the estimate \eqref{C1-1}  we get
\begin{align*}
\Vert (S_{3}\ast f(\xi ,\cdot ))(\cdot )\Vert _{V_{\gamma }}\leq &
\int_{0}^{t}\left( \int_{\Omega }C^{2}(t-\tau )^{2(\alpha -1-\alpha \gamma
)}|f(u(\xi ,\tau ))|^{2}\;d\mu \right) ^{1/2}\;d\tau \\
& \leq CQ_{2}(R)T^{\alpha (1-\gamma )}.
\end{align*}%
Thus, proceeding as the proof of Theorem \ref{theo-weak}, we get that there
is a constant $C>0$ such that for every $t\in \lbrack 0,T]$, 
\begin{equation*}
\Vert \Phi (u(\cdot ))(t)\Vert _{V_{\gamma }}\leq C\left( \Vert u_{0}\Vert
_{V_{\gamma }}+\Vert u_{1}\Vert _{L^{2}(\Omega )}+Q_{2}\left( R\right)
T^{\alpha (1-\gamma )}\right) .
\end{equation*}%
Let us see that $\Phi u\in C([0,T];V_{\gamma })$. We start proving that $%
S_{1}(t)u_{0}(\xi )$ is continuous on $[0,T]$. Indeed, since $t\mapsto
S_{1}(t)x$ is continuous for each $t\in \lbrack 0,T]$,  for $h>0$ we have 
\begin{align*}
& \Vert S_{1}(t+h)u_{0}(\cdot )-S_{1}(t)u_{0}(\cdot )\Vert _{V_{\gamma }} \\
& =\int_{\Omega }|(E_{\alpha ,1}(-m(\xi )(t+h)^{\alpha })-E_{\alpha
,1}(-m(\xi )t^{\alpha }))(m(\xi ))^{\gamma }u_{0}(\xi )|^{2}\;d\mu .
\end{align*}%
Since $t\mapsto E_{\alpha ,1}(-m(\xi )t^{\alpha })u_{0}(\xi )$ is continuous
for every $t\in \lbrack 0,T]$, 
\begin{equation*}
|(E_{\alpha ,1}(-m(\xi )(t+h)^{\alpha })-E_{\alpha ,1}(-m(\xi )t^{\alpha
}))(m(\xi ))^{\gamma }u_{0}(\xi )|^{2}\leq C^{2}|(m(\xi ))^{\gamma
}u_{0}(\xi )|^{2}
\end{equation*}%
and $u_{0}\in V_{\gamma }$, it follows from Dominated Convergence Theorem
that $S_{1}(t)u_{0}(\xi )$ is continuous for all $t\in \lbrack 0,T]$. The
proof of the fact that $S_{2}(t)u_{1}(\xi )$ is continuous on $[0,T]$ is
similar. Now, 
\begin{align*}
& \Bigg\Vert \int_{0}^{t+h}(t+h-\tau )^{\alpha -1}E_{\alpha ,\alpha }(-m(\cdot
)(t+h-\tau )^{\alpha })f(u(\cdot ,\tau ))\;d\tau \\
& -\int_{0}^{t}(t-\tau )^{\alpha -1}E_{\alpha ,\alpha }(-m(\cdot )(t-\tau
)^{\alpha })f(u(\cdot ,\tau ))\;d\tau \Big\Vert _{V_{\gamma }} \\
& \leq \int_{0}^{t}\Vert \lbrack (t+h-\tau )^{\alpha -1}E_{\alpha ,\alpha
}(-m(\cdot )(t+h-\tau )^{\alpha })-(t-\tau )^{\alpha -1}E_{\alpha ,\alpha
}(-m(\cdot )(t-\tau )^{\alpha })]f(u(\cdot ,\tau ))\Vert _{V_{\gamma
}}\;d\tau \\
& +\int_{t}^{t+h}\Vert (t+h-\tau )^{\alpha -1}E_{\alpha ,\alpha }(-m(\cdot
)(t+h-\tau )^{\alpha })f(u(\cdot ,\tau ))\Vert _{V_{\gamma }}\;d\tau
=:I_{1}+I_{2}.
\end{align*}%
Let us see that $I_{1}\rightarrow 0$ as $h\rightarrow 0^{+}$. First, we need
to prove 
\begin{equation}
\lim_{h\rightarrow 0^{+}}\Vert \lbrack Q_{\alpha ,\alpha }(t+h-\tau
)-Q_{\alpha ,\alpha }(t-\tau )]f(u(\cdot ,\tau ))\Vert _{V_{\gamma }}=0.
\label{Dct}
\end{equation}%
Indeed, since $t\mapsto t^{\alpha -1}E_{\alpha ,\alpha }(-m(\xi )t^{\alpha
})f(u(\xi ,t))$ is continuous for every $t\in \lbrack 0,T]$, 
\begin{equation*}
|(m(\xi ))^{\gamma }[Q_{\alpha ,\alpha }(t+h-\tau )-Q_{\alpha ,\alpha
}(t-\tau )]f(u(\cdot ,\tau ))|\leq 2C(t-\tau )^{\alpha (1-\gamma
)-1}|f(u(\xi ,\tau ))|
\end{equation*}%
and $f(u(\cdot ,\tau ))\in L^{2}(\Omega )$, \eqref{Dct} follows from
Dominated Convergence Theorem. Moreover, 
\begin{equation*}
\Vert \lbrack Q_{\alpha ,\alpha }(t+h-\tau )-Q_{\alpha ,\alpha }(t-\tau
)]f(u(\cdot ,\tau ))\Vert _{V_{\gamma }}\leq C(t-\tau )^{\alpha (1-\gamma
)-1}Q_{2}(R).
\end{equation*}%
Since $\tau \mapsto C(t-\tau )^{\alpha (1-\gamma )-1}Q_{2}(R)$ is
integrable, then $I_{1}\rightarrow 0$ as $h\rightarrow 0^{+}$ by the
Dominated Convergence Theorem.

On the other hand, observe that 
\begin{equation*}
\Vert (t+h-\tau )^{\alpha -1}E_{\alpha ,\alpha }(-m(\cdot )(t+h-\tau
)^{\alpha })f(u(\cdot ,\tau ))\Vert _{V_{\gamma }}\leq C(t+h-\tau )^{\alpha
(1-\gamma )-1}Q_{2}(R)
\end{equation*}%
implies 
\begin{equation*}
I_{2}\leq CQ_{2}(R)\int_{t}^{t+h}(t+h-\tau )^{\alpha (1-\gamma )-1}\;d\tau
\leq CQ_{2}(R)h^{\alpha (1-\gamma )}\rightarrow 0,\qquad h\rightarrow 0^{+}.
\end{equation*}%
We have shown that $\Phi u\in C([0,T];V_{\gamma })$. Similarly, we have that
there is a constant $C>0$ such that for every $t\in \lbrack 0,T]$, 
\begin{align}
& t^{\delta _{1}}\Vert \Phi (u(\cdot ))^{\prime }(t)\Vert _{L^{2}(\Omega )} 
\notag \\
\leq & C\left( t^{\delta _{1}+\alpha \gamma -1}\Vert u_{0}\Vert _{V_{\gamma
}}+t^{\delta _{1}}\Vert u_{1}\Vert _{L^{2}(\Omega )}+T^{\delta _{1}+\alpha
-1}Q_{2}\left( \left\Vert u\right\Vert _{C\left( \left[ 0,T\right]
;V_{\gamma }\right) }\right) \right)  \notag \\
\leq & C\left( T^{\delta _{1}+\alpha \gamma -1}\Vert u_{0}\Vert _{V_{\gamma
}}+T^{\delta _{1}}\Vert u_{1}\Vert _{L^{2}(\Omega )}+T^{\delta _{1}+\alpha
-1}Q_{2}\left( R\right) \right) .  \label{C3}
\end{align}%
Let us see that $\Phi (u)\in C^{1}([0,T];L^{2}(\Omega ))$. From the
Dominated Convergence Theorem one gets $m(\cdot )t^{\alpha -1}E_{\alpha
,\alpha }(-m(\xi )t^{\alpha })u_{0}(\cdot )\in C([0,T];L^{2}(\Omega ))$ and $%
E_{\alpha ,1}(-m(\cdot )t^{\alpha })u_{1}(\cdot )\in C([0,T];L^{2}(\Omega ))$%
. Now, since 
\begin{equation*}
\lim_{h\rightarrow 0^{+}}|[Q_{\alpha ,\alpha -1}(t+h-\tau )-Q_{\alpha
,\alpha -1}(t-\tau )]f(u(\xi ,\tau ))|=0\quad a.e.,
\end{equation*}%
\begin{align*}
& |[Q_{\alpha ,\alpha -1}(t+h-\tau )-Q_{\alpha ,\alpha -1}(t-\tau )]f(u(\xi
,\tau ))| \\
& \leq C(t-\tau )^{\alpha -2}|f(u(\xi ,\tau ))|,
\end{align*}%
and \eqref{C1-1}, we have that 
\begin{equation*}
\lim_{h\rightarrow 0^{+}}\left\Vert [Q_{\alpha ,\alpha -1}(t+h-\tau
)-Q_{\alpha ,\alpha -1}(t-\tau )]f(u(\xi ,\tau ))\right\Vert _{L^{2}(\Omega
)}=0.
\end{equation*}%
The Dominated Convergence Theorem implies that 
\begin{equation*}
J_{1}:=\int_{0}^{t}\left\Vert [Q_{\alpha ,\alpha -1}(t+h-\tau )-Q_{\alpha
,\alpha -1}(t-\tau )]f(u(\xi ,\tau ))\right\Vert _{L^{2}(\Omega )}\;d\tau
\end{equation*}%
goes to zero as $h\rightarrow 0^{+}$. As before, we can prove that 
\begin{equation*}
J_{2}:=\int_{t}^{t+h}\Vert (t+h-\tau )^{\alpha -2}E_{\alpha ,\alpha
-1}(-m(\cdot )(t+h-\tau )^{\alpha })f(u(\cdot ,\tau ))\Vert _{L^{2}(\Omega
)}\;d\tau \rightarrow 0,\quad h\rightarrow 0^{+}.
\end{equation*}%
Therefore 
\begin{align*}
& \Big\Vert \int_{0}^{t+h}(t+h-\tau )^{\alpha -2}E_{\alpha ,\alpha -1}(-m(\cdot
)(t+h-\tau )^{\alpha })f(u(\cdot ,\tau ))\;d\tau \\
& -\int_{0}^{t}(t-\tau )^{\alpha -2}E_{\alpha ,\alpha -1}(-m(\cdot )(t-\tau
)^{\alpha })f(u(\cdot ,\tau ))\;d\tau \Big\Vert _{L^{2}(\Omega )}\rightarrow
0,\quad h\rightarrow 0^{+}.
\end{align*}%
Hence $\Phi (u)\in C^{1}([0,T];L^{2}(\Omega ))$. Let us see that $\Phi
(u)\in C((0,T];D(A))$. Let $0<t\leq T$. First, from the proof of the linear
part (see \eqref{S1} and \eqref{D2-SS} with $\widetilde{\gamma }=0$) we
have 
\begin{equation*}
\Vert m(\cdot )E_{\alpha ,1}(-m(\cdot )t^{\alpha })u_{0}(\cdot )\Vert
_{L^{2}(\Omega )}^{2}\leq C^{2}t^{2\alpha \left( \gamma -1\right) }\Vert
u_{0}\Vert _{V_{\gamma }},
\end{equation*}%
and 
\begin{equation*}
\left\Vert m(\cdot )tE_{\alpha ,2}(-m(\cdot )t^{\alpha })u_{1}(\cdot
)\right\Vert _{L^{2}(\Omega )}\leq Ct^{1-\alpha }\Vert u_{1}\Vert
_{L^{2}(\Omega )}.
\end{equation*}%
Note that 
\begin{equation*}
\int_{0}^{t}\Vert f^{\prime }(u(\cdot ,\tau ))\partial _{\tau }u(\cdot ,\tau
)\Vert _{L^{2}(\Omega )}\;d\tau \leq Q_{1}(R)\int_{0}^{t}\Vert \partial
_{\tau }u(\cdot ,\tau ))\Vert _{L^{2}(\Omega )}\,d\tau .
\end{equation*}%
Since $\sup_{t\in \lbrack 0,T]}t^{\delta _{1}}\Vert \partial _{t}u(\cdot
,t)\Vert _{L^{2}(\Omega )}=T^{\delta _{1}}\sup_{t\in \lbrack 0,T]}\Vert
\partial _{t}u(\cdot ,t)\Vert _{L^{2}(\Omega )}\leq R$, then 
\begin{equation*}
\int_{0}^{t}\Vert f^{\prime }(u(\cdot ,\tau ))\partial _{\tau }u(\cdot ,\tau
)\Vert _{L^{2}(\Omega )}\;d\tau \leq Q_{1}(R)T^{1-\delta _{1}}R.
\end{equation*}%
Using the previous inequality, \eqref{C1-1} and integration by parts, we
obtain 
\begin{align*}
& \left\Vert \int_{0}^{t}m(\cdot )(t-\tau )^{\alpha -1}E_{\alpha ,\alpha
}(-m(\cdot )(t-\tau )^{\alpha })f(u(\cdot ,\tau ))\;d\tau \right\Vert
_{L^{2}(\Omega )} \\
& \leq C[Q_{2}(R)+Q_{1}(R)T^{1-\delta _{1}}R].
\end{align*}%
Hence 
\begin{align}
t^{\delta _{2}}\left\Vert A\Phi (u(\cdot ))(t)\right\Vert _{L^{2}\left(
\Omega \right) }& \leq C\left( T^{\alpha (\gamma -1)+\delta _{2}}\Vert
u_{0}\Vert _{V_{\gamma }}+T^{1-\alpha +\delta _{2}}\Vert u_{1}\Vert
_{L^{2}(\Omega )}\right. \notag\\
& +T^{\delta _{2}}Q_{2}(R)+T^{1+\delta _{2}-\delta _{1}}Q_{1}(R)R\left.
{}\right) .  \label{EQ-5.17} 
\end{align}%
We continue with the proof that $\Phi (u)\in C((0,T];D(A))$. Observe that
for $0<t\leq T$ 
\begin{align*}
& \Vert S_{1}(t+h)u_{0}(\cdot )-S_{1}(t)u_{0}(\cdot )\Vert _{D(A)}^{2} \\
& =\int_{\Omega }\left\vert m(\xi )[E_{\alpha ,1}(-m(\xi )(t+h)^{\alpha
})-m(\xi )E_{\alpha ,1}(-m(\xi )t^{\alpha })]u_{0}(\xi )\right\vert
^{2}\,d\mu ,
\end{align*}%
and by \eqref{IN-L2} we get 
\begin{align*}
& \left\vert m(\xi )E_{\alpha ,1}(-m(\xi )(t+h)^{\alpha })u_{0}(\xi )-m(\xi
)E_{\alpha ,1}(-m(\xi )t^{\alpha })u_{0}(\xi )\right\vert ^{2} \\
& \leq C^{2}t^{2\alpha (\gamma -1)}((m(\xi ))^{\gamma }u_{0}(\xi ))^{2}\in
L^{1}(\Omega ).
\end{align*}%
By the continuity of $t\mapsto S_{1}(t)x$ ($x\in \Omega $), we obtain from
the Dominated Convergence Theorem that 
\begin{equation*}
\Vert S_{1}(t+h)u_{0}(\cdot )-S_{1}(t)u_{0}(\cdot )\Vert
_{D(A)}^{2}\rightarrow 0\qquad (h\rightarrow 0^{+}).
\end{equation*}%
Analogously, we prove that 
\begin{equation*}
\Vert S_{2}(t+h)u_{1}(\cdot )-S_{2}(t)u_{1}(\cdot )\Vert
_{D(A)}^{2}\rightarrow 0\qquad (h\rightarrow 0^{+}).
\end{equation*}%
Next, 
\begin{align*}
& \Big\Vert \int_{0}^{t+h}(t+h-\tau )^{\alpha -1}E_{\alpha ,\alpha }(-m(\cdot
)(t+h-\tau )^{\alpha })f(u(\cdot ,\tau ))\;d\tau \\
& -\int_{0}^{t}(t-\tau )^{\alpha -1}E_{\alpha ,\alpha }(-m(\cdot )(t-\tau
)^{\alpha })f(u(\cdot ,\tau ))\;d\tau\Big \Vert _{D(A)} \\
& \leq \left\Vert m(\cdot )\int_{0}^{t}[Q_{\alpha ,\alpha }(t+h-\tau
)-Q_{\alpha ,\alpha }(t-\tau )]f(u(\cdot ,\tau ))\;d\tau \right\Vert
_{L^{2}(\Omega )} \\
& +\left\Vert m(\cdot )\int_{t}^{t+h}(t+h-\tau )^{\alpha -1}E_{\alpha
,\alpha }(-m(\cdot )(t+h-\tau )^{\alpha })f(u(\cdot ,\tau ))\;d\tau
\right\Vert _{L^{2}(\Omega )}=:I_{1}+I_{2}.
\end{align*}%
Let us see that $I_{1}\rightarrow 0$ as $h\rightarrow 0^{+}$. First, we need
to prove 
\begin{equation}
\lim_{h\rightarrow 0^{+}}\Vert \lbrack Q_{\alpha ,\alpha }(t+h-\tau
)-Q_{\alpha ,\alpha }(t-\tau )]f(u(\cdot ,\tau ))\Vert _{D(A)}=0.
\label{dct-1}
\end{equation}%
Indeed, since $t\mapsto Q_{\alpha ,\alpha }(t)f(u(\xi ,t))$ is continuous
for every $t\in \lbrack 0,T]$ and 
\begin{equation*}
|m(\xi )[Q_{\alpha ,\alpha }(t+h-\tau )-Q_{\alpha ,\alpha }(t-\tau
)]f(u(\cdot ,\tau ))|\leq C(t-\tau )^{-1}|f(u(\xi ,\tau ))|\in L^{2}(\Omega
),
\end{equation*}%
\eqref{dct-1} follows from Dominated Convergence Theorem. Moreover, 
\begin{align*}
I_{1} =&\left\Vert \int_{0}^{t}m(\cdot )[Q_{\alpha ,\alpha }(t+h-\tau
)-Q_{\alpha ,\alpha }(t-\tau )]f(u(\cdot ,\tau ))\;d\tau \right\Vert
_{L^{2}(\Omega )} \\
\leq & \Vert -f(u(\cdot ,t))[E_{\alpha ,1}(-m(\cdot )h^{\alpha })-1]\Vert
_{L^{2}(\Omega )}\\
&+\Vert f(u(\cdot ,0))[E_{\alpha ,1}(-m(\cdot )(t+h)^{\alpha
})-E_{\alpha ,1}(-m(\cdot )t^{\alpha })]\Vert _{L^{2}(\Omega )} \\
& +\Vert \int_{0}^{t}[E_{\alpha ,1}(-m(\cdot )(t-\tau +h)^{\alpha
})-E_{\alpha ,1}(-m(\cdot )(t-\tau )^{\alpha })]f^{\prime }(u(\cdot ,\tau
))\partial _{\tau }u(\cdot ,\tau )\;d\tau \Vert _{L^{2}(\Omega )} \\
& =:I_{11}+I_{12}+I_{13}.
\end{align*}%
Since $|f(\xi ,t)|\in L^{2}(\Omega )$ for $\xi \in \Omega $ and $0\leq t\leq
T$, then $I_{11}$ and $I_{12}$ converge to zero as $h\rightarrow 0^{+}$ by
the Dominated Convergence Theorem. On the other hand, we have that 
\begin{align*}
& |[E_{\alpha ,1}(-m(\cdot )(t-\tau +h)^{\alpha })-E_{\alpha ,1}(-m(\cdot
)(t-\tau )^{\alpha })]f^{\prime }(u(\cdot ,\tau ))\partial _{\tau }u(\cdot
,\tau )| \\
& \leq 2C|f^{\prime }(u(\cdot ,\tau ))\partial _{\tau }u(\cdot ,\tau )|\leq
2CQ_{1}(r)|\partial _{\tau }u(\cdot ,\tau )|\in L^{2}(\Omega ).
\end{align*}%
Then, the Dominated Convergence Theorem implies that 
\begin{equation*}
\Vert \lbrack E_{\alpha ,1}(-m(\cdot )(t-\tau +h)^{\alpha })-E_{\alpha
,1}(-m(\cdot )(t-\tau )^{\alpha })]f^{\prime }(u(\cdot ,\tau ))\partial
_{\tau }u(\cdot ,\tau )\Vert _{L^{2}(\Omega )}\rightarrow 0\quad
(h\rightarrow 0^{+}).
\end{equation*}%
Thus 
\begin{align*}
& \Vert \lbrack E_{\alpha ,1}(-m(\cdot )(t-\tau +h)^{\alpha })-E_{\alpha
,1}(-m(\cdot )(t-\tau )^{\alpha })]f^{\prime }(u(\cdot ,\tau ))\partial
_{\tau }u(\cdot ,\tau )\Vert _{L^{2}(\Omega )} \\
& \leq 2CQ_{1}(R)\Vert \partial _{\tau }u(\cdot ,\tau )\Vert _{L^{2}(\Omega
)}\leq 2CQ_{1}(R)T^{-\delta _{1}}R.
\end{align*}%
The Dominated Convergence Theorem guarantees that $I_{13}$ goes to zero as $%
h\rightarrow 0^{+}$. Next, we see that $I_{2}\rightarrow 0$ as $t\rightarrow
0^{+}$. Integration by parts gives 
\begin{align*}
I_{2}& \leq \Vert -f(u(\cdot ,t+h))+f(u(\cdot ,t))E_{\alpha ,1}(-m(\cdot
)h^{\alpha })\Vert _{L^{2}(\Omega )} \\
& +\left\Vert \int_{t}^{t+h}f^{\prime }(u(\cdot ,\tau ))\partial _{\tau
}u(\cdot ,\tau )E_{\alpha ,1}(-m(\cdot )(t-\tau +h)^{\alpha })\,d\tau
\right\Vert _{L^{2}(\Omega )} \\
& =:I_{21}+I_{22}.
\end{align*}%
By continuity we have that $I_{21}\rightarrow 0$ as $h\rightarrow 0^{+}$. On
the other hand, 
\begin{align*}
I_{22}& \leq \int_{t}^{t+h}\left\Vert f^{\prime }(u(\cdot ,\tau ))\partial
_{\tau }u(\cdot ,\tau )E_{\alpha ,1}(-m(\cdot )(t-\tau +h)^{\alpha
})\right\Vert _{L^{2}(\Omega )}\,d\tau \\
& \leq CQ_{1}(R)T^{-\delta _{1}}Rh\rightarrow 0\qquad (h\rightarrow 0^{+}).
\end{align*}%
The conclusion follows.

On the other hand, for $u_{1}(\cdot ,t),u_{2}(\cdot ,t)\in V_{\gamma }$, the
mean value theorem implies%
\begin{equation}
\Vert f(u_{1}(\cdot ,t))-f(u_{2}(\cdot ,t))\Vert _{L^{2}(\Omega )}\leq
CQ_{1}(R)\Vert u_{1}(\cdot ,t)-u_{2}(\cdot ,t)\Vert _{V_{\gamma }}.
\label{EQ.5.19}
\end{equation}%
Similarly to the foregoing estimates, we can exploit%
\begin{equation*}
\Vert \Phi (u_{1}(\cdot ))(t)-\Phi (u_{2}(\cdot ))(t)\Vert _{V_{\gamma
}}\leq CQ_{1}(R)\int_{0}^{t}(t-\tau )^{\alpha (1-\gamma )-1}\Vert
u_{1}(\cdot ,\tau )-u_{2}(\cdot ,\tau )\Vert _{V_{\gamma }}d\tau .
\end{equation*}%
Hence 
\begin{equation}
\Vert \Phi (u_{1}(\cdot ))(t)-\Phi (u_{2}(\cdot ))(t)\Vert _{V_{\gamma
}}\leq CQ_{1}(R)\Vert u_{1}-u_{2}\Vert _{C([0,T];V_{\gamma })}T^{\alpha
(1-\gamma )}.  \label{Eq-5-20}
\end{equation}%
In particular, since $v=0$ belongs to $V_{\gamma }$, $\Phi v=0$ and $\Phi
(u(\cdot ))(t)\in V_{\gamma }$ (because $u(\cdot ,t)\in V_{\gamma }$), we
get 
\begin{equation*}
\Vert f\Phi (u(\cdot ))(t)\Vert _{L^{2}(\Omega )}\leq CQ_{1}(R)RT^{\alpha
(1-\gamma )},
\end{equation*}%
where we used \eqref{Eq-5-20}. Then, from the identity $\mathbb{D}%
_{t}^{\alpha }\Phi u=-A\Phi u+f\Phi u$ and \eqref{EQ-5.17}, we have 
\begin{align}
t^{\delta _{2}}\left\Vert \mathbb{D}_{t}^{\alpha }\Phi (u(\cdot
))(t)\right\Vert _{L^{2}(\Omega )}& \leq t^{\delta _{2}}\left\Vert A\Phi
(u(\cdot ))(t)\right\Vert _{L^{2}\left( \Omega \right) }+t^{\delta
_{2}}\Vert f\Phi (u(\cdot ))(t))\Vert _{L^{2}(\Omega )}  \notag \\
& \leq C\left( T^{\alpha (\gamma -1)+\delta _{2}}\Vert u_{0}\Vert
_{V_{\gamma }}+T^{1-\alpha +\delta _{2}}\Vert u_{1}\Vert _{L^{2}(\Omega
)}\right.  \notag \\
& +\left. T^{\delta _{2}}Q_{2}(R)+(T^{1+\delta _{2}-\delta _{1}}+T^{\delta
_{2}+\alpha (1-\gamma )})Q_{1}(R)R\right) .  \label{EQU-5.21}
\end{align}%
From \eqref{C2}, \eqref{C3}, \eqref{EQ-5.17} and \eqref{EQU-5.21}, we obtain 
\begin{align}
& \Vert \Phi (u(\cdot ))(t)\Vert _{V_{\gamma }}+t^{\delta _{1}}\Vert \Phi
(u(\cdot ))^{\prime }(t)\Vert _{L^{2}(\Omega )}+t^{\delta _{2}}\left\Vert
A\Phi (u(\cdot ))(t)\right\Vert _{L^{2}\left( \Omega \right) }+t^{\delta
_{2}}\left\Vert \mathbb{D}_{t}^{\alpha }\Phi (u(\cdot ))(t)\right\Vert
_{L^{2}(\Omega )} \notag \\
& \leq C\left( \Vert u_{0}\Vert _{V_{\gamma }}+\Vert u_{1}\Vert
_{L^{2}(\Omega )}+Q_{2}\left( R\right) T^{\alpha (1-\gamma )}\right)  \notag
\\
& +C\left( T^{\delta _{1}+\alpha \gamma -1}\Vert u_{0}\Vert _{V_{\gamma
}}+T^{\delta _{1}}\Vert u_{1}\Vert _{L^{2}(\Omega )}+T^{\delta _{1}+\alpha
-1}Q_{2}\left( R\right) \right)  \notag \\
& +C\left( T^{\alpha (\gamma -1)+\delta _{2}}\Vert u_{0}\Vert _{V_{\gamma
}}+T^{1-\alpha +\delta _{2}}\Vert u_{1}\Vert _{L^{2}(\Omega )}+T^{\delta
_{2}}Q_{2}(R)+T^{1+\delta _{2}-\delta _{1}}Q_{1}(R)R\right)  \notag \\
& +C\left( T^{\alpha (\gamma -1)+\delta _{2}}\Vert u_{0}\Vert _{V_{\gamma
}}+T^{1-\alpha +\delta _{2}}\Vert u_{1}\Vert _{L^{2}(\Omega )}+T^{\delta
_{2}}Q_{2}(R)+(T^{1+\delta _{2}-\delta _{1}}+T^{\delta _{2}+\alpha (1-\gamma
)})Q_{1}(R)R\right) .   \label{main-in-non}
\end{align}%
Letting now 
\begin{equation*}
R\geq 2C(\Vert u_{0}\Vert _{V_{\gamma }}+\Vert u_{1}\Vert _{L^{2}(\Omega )}),
\end{equation*}%
we can find a sufficiently small time $T>0$ such that 
\begin{align}
& CQ_{2}(R)(T^{\alpha (1-\gamma )}+T^{\alpha -1+\delta _{1}}+T^{\delta _{2}})
\notag\\
& +C\max \{T^{\delta _{1}+\alpha \gamma -1},T^{\delta _{1}}\}(\Vert
u_{0}\Vert _{V_{\gamma }}+\Vert u_{1}\Vert _{L^{2}(\Omega )})  \notag \\
& +C\max \{T^{\alpha (\gamma -1)+\delta _{2}},T^{1-\alpha +\delta
_{2}}\}(\Vert u_{0}\Vert _{V_{\gamma }}+\Vert u_{1}\Vert _{L^{2}(\Omega )}) 
\notag \\
& +CQ_{1}(R)(T^{1+\delta _{2}-\delta _{1}}+T^{\delta _{2}+\alpha (1-\gamma
)})R\leq \frac{R}{2}.  \label{Eq-519} 
\end{align}%
Thus, taking into account \eqref{main-in-non}, we deduce 
\begin{align*}
\Vert \Phi u\Vert _{\delta _{1},\delta _{2},T}=& \sup_{0\leq t\leq T}\left\{
\Vert \Phi (u(\cdot ))(t)\Vert _{V_{\gamma }}+\Vert \partial _{t}\Phi
(u(\cdot ))(t)\Vert _{L^{2}(\Omega )}+t^{\delta _{2}}\left\Vert A\Phi
(u(\cdot ))(t)\right\Vert _{L^{2}(\Omega )}\right. \\
& \left. +t^{\delta _{2}}\left\Vert \mathbb{D}_{t}^{\alpha }\Phi (u(\cdot
))(t)\right\Vert _{L^{2}(\Omega )}\right\} \leq R.
\end{align*}%
From here we deduce that $\Phi u\in \mathcal{B}_{\delta _{1},\delta
_{2},T,R} $. Next, we show that by choosing a possibly smaller $T>0$, $\Phi $
is a contraction on $\mathcal{B}_{\delta _{1},\delta _{2},T,R}$. Indeed, let 
$u,v\in \mathcal{B}_{\delta _{1},\delta _{2},T,R}$. Analogously to %
\eqref{Eq-5-20}, we have 
\begin{equation*}
\Vert \Phi (u(\cdot ))^{\prime }\left( t\right) -\Phi (v(\cdot ))^{\prime
}\left( t\right) \Vert _{L^{2}(\Omega )}\leq CQ_{1}(R)\Vert u-v\Vert
_{C([0,T];V_{\gamma })}T^{\alpha -1},
\end{equation*}%
and 
\begin{equation*}
t^{\delta _{1}}\Vert \Phi (u(\cdot ))^{\prime }\left( t\right) -\Phi
(v(\cdot ))^{\prime }\left( t\right) \Vert _{L^{2}(\Omega )}\leq
CQ_{1}(R)\Vert u-v\Vert _{C([0,T];V_{\gamma })}T^{\delta _{1}+\alpha -1}.
\end{equation*}%
Now, integration by parts gives 
\begin{align*}
& \Vert A\Phi (u(\cdot ))(t)-A\Phi (v(\cdot ))(t)\Vert _{L^{2}(\Omega )} \\
& \leq CQ_{1}(R)\Vert u-v\Vert _{C([0,T];V_{\gamma
})}+CQ_{1}(R)\int_{0}^{t}\tau ^{-\delta _{1}}\tau ^{\delta _{1}}\Vert
\partial _{\tau }u(\cdot ,\tau )-\partial _{\tau }v(\cdot ,\tau )\Vert
_{L^{2}(\Omega )}\,d\tau \\
& \leq CQ_{1}(R)\Vert u-v\Vert _{C([0,T];V_{\gamma })}+CQ_{1}(R)T^{1-\delta
_{1}}\Vert u-v\Vert _{\delta _{1},\delta _{2},T}.
\end{align*}%
Hence 
\begin{equation*}
t^{\delta _{2}}\Vert A\Phi (u(\cdot ))(t)-A\Phi (v(\cdot ))(t)\Vert
_{L^{2}(\Omega )}\leq CQ_{1}(R)(T^{\delta _{2}}+T^{1+\delta _{2}-\delta
_{1}})\Vert u-v\Vert _{\delta _{1},\delta _{2},T}.
\end{equation*}%
On the other hand, since $\Phi u,\Phi v\in \mathcal{B}_{\delta _{1},\delta
_{2},T,R}$, we obtain from \eqref{Eq-5-20} that 
\begin{equation*}
\Vert f\Phi (u(\cdot ))(t)-f\Phi (v(\cdot ))(t)\Vert _{L^{2}(\Omega )}\leq
C^{2}(Q_{1}(R))^{2}\Vert u-v\Vert _{C([0,T];V_{\gamma })}T^{\alpha (1-\gamma
)}.
\end{equation*}%
From here we deduce that 
\begin{align*}
& t^{\delta _{2}}\Vert \mathbb{D}_{t}^{\alpha }\Phi (u(\cdot ))(t)-\mathbb{D}%
_{t}^{\alpha }\Phi (v(\cdot ))(t)\Vert _{L^{2}(\Omega )}\leq t^{\delta
_{2}}\Vert A\Phi (u(\cdot ))(t)-A\Phi (v(\cdot ))(t)\Vert _{L^{2}(\Omega )}
\\
& \leq CQ_{1}(R)(T^{\delta _{2}}+T^{1+\delta _{2}-\delta
_{1}}+Q_{1}(R)T^{\delta _{2}+\alpha (1-\gamma )})\Vert u-v\Vert _{\delta
_{1},\delta _{2},T}.
\end{align*}%
Hence 
\begin{align*}
& \Vert \Phi u-\Phi v\Vert _{\delta _{1},\delta _{2},T} \\
& \leq CQ_{1}(R)(T^{\alpha (1-\gamma )}+T^{\delta _{1}+\alpha -1}+T^{\delta
_{2}}+T^{1+\delta _{2}-\delta _{1}}+Q(R)T^{\delta _{2}+\alpha (1-\gamma
)})\Vert u-v\Vert _{\delta _{1},\delta _{2},T}.
\end{align*}%
Choosing $T>0$ smaller than the one determined by \eqref{Eq-519} such that 
\begin{equation*}
CQ_{1}(R)(T^{\alpha (1-\gamma )}+T^{\delta _{1}+\alpha -1}+T^{\delta
_{2}}+T^{1+\delta _{2}-\delta _{1}}+Q_{1}(R)T^{\delta _{2}+\alpha (1-\gamma
)})<1,
\end{equation*}%
it follows that the mapping $\Phi $ is a contraction. In a similar way than
before, we can show that $u$ can be extended by continuity from $[0,T_{\ast
}]$ for some $T_{\ast }$, to the interval $\left[ 0,T_{\ast }+\varsigma %
\right] ,$ for some $\varsigma >0$. The rest of the proof (associated with
the extension argument and (\ref{bl}))\ goes in an analogous fashion as in
the proof of Theorem \ref{main-weak}. The proof is complete.
\end{proof}

The following result gives information about the regularity of the second derivative of strong solutions.
\begin{theorem}
\label{cor-str}Let the assumptions \emph{\textbf{(Hf3)}} and $u$ be a weak solution in the sense of Definition  \ref{def-weak}.  Let $1/2\leq \gamma\leq 1$, $0\leq \widetilde{\gamma}\leq 1$, 
$0\leq \theta <1$ and
\begin{align*}
1\leq \zeta'<\frac{1}{2-\alpha(1-\theta)},\quad  \frac{1}{\zeta'}=1-\frac{1}{\zeta}.
\end{align*}
Assume that $u_{0}\in V_{\gamma },$ $%
u_{1}\in V_{\widetilde{\gamma }}$ with both $\gamma -\theta ,$ $\widetilde{%
\gamma }-\theta \in \left[ 0,1\right] $. Then $u$ is a strong solution in the sense of %
\ref{def-strong} and satisfies the following estimate:%
\begin{align*}
&\left\Vert \partial _{t}^{2}u\right\Vert _{V_{\theta}}\\
&\leq  C\left( T^{\alpha (\gamma-\theta)-2}\Vert u_{0}\Vert _{V_{\gamma
}}+T^{\alpha (\tilde{\gamma}-\theta-1 )}\Vert u_{1}\Vert _{V_{\tilde{\gamma} }}\right.
 \\
& \qquad \left. CT^{\frac{1}{\zeta'}+\alpha(1-\theta)-2}\|f'(u)\|_{L^{\infty}((0,T);L^{\infty}(\Omega))}\,\|\partial_tu\|_{L^{\zeta}((0,T);V_{\tilde{\gamma}})}+T^{\alpha-2 -\alpha\theta}\Vert f(u(\cdot ,0))\Vert
_{L^{2}( \Omega)}\right). \notag
\end{align*}%
\end{theorem}
\begin{proof}
Indeed, 
each weak solution turns out to be bounded, namely,%
\begin{equation}
u\in C\left( \left[ 0,T\right] ;V_{\tilde{\gamma} }\right) \subset C\left( \left[ 0,T%
\right] ;L^{\infty }\left(\Omega\right) \right) .  \label{con3}
\end{equation}%
Consequently, one has $f^{^{\prime }}\left( u\right) \in L^{\infty }\left(
(0,T);L^{\infty }\left( \Omega\right) \right) $ from \eqref{con3} and the
assumption (\textbf{Hf3}). 
Indeed, since $u\in C\left( \left[ 0,T%
\right] ;L^{\infty }\left(\Omega\right) \right)$, there exists $K_0>0$ such that  $\mbox{ess sup}_{\xi\in \Omega}\,|u(\xi,t)|<K_0$ for all $t\in [0,T]$. Then, the monotone property of $Q_1$ implies that $Q_1(|u(\xi,t|)\leq Q_1(K_0)$.
The claim follows from
\begin{align*}
\mbox{ess sup}_{t\in[0,T]}\|f^{^{\prime }}(u(\cdot,t)\|_{L^{\infty}(\Omega)}
&=\mbox{ess sup}_{t\in[0,T]}\left[\mbox{ess sup}_{\xi\in \Omega}|f'(u(\xi,t))|\right]\\
&\leq C\,\,\mbox{ess sup}_{t\in[0,T]}\left[\mbox{ess sup}_{\xi\in \Omega}\,Q_1(|u(\xi,t)|)\right].
\end{align*}
Now, we show that $\partial_t u\in L^{\zeta}((0;T);V_{\tilde{\gamma}})$.  Indeed, 
\begin{align*}
&\left(\int_0^T\|-m(\cdot)t^{\alpha-1}E_{\alpha,\alpha}(-m(\cdot)t^{\alpha})u_0(\cdot)\|^{\zeta}_{V_{\tilde{\gamma}}}\;dt\right)^{1/\zeta}\\
&\leq \left(\int_0^T\left(\int_{\Omega}|m(\xi)^{\tilde{\gamma}}t^{\alpha-1}E_{\alpha,\alpha}(-m(\xi)t^{\alpha})u_0(\xi)|^2\;d\mu\right)^{\zeta/2}\;dt\right)^{1/\zeta}\\
&\leq C\left(\int_0^T t^{\zeta(\alpha-1-\alpha(\tilde{\gamma}-\gamma))}\left(\int_{\Omega}|(m(\xi))^{\gamma}u_0(\xi)|^2\;d\mu\right)^{\zeta/2}\;dt\right)^{1/\zeta}\\
&=CT^{\alpha-1+\alpha(\gamma-\tilde{\gamma})+\frac{1}{\zeta}}\|u_0\|_{V_{\gamma}},
\end{align*}
where we have used that $\alpha(\gamma-\tilde{\gamma})+\frac{1}{\zeta}>1$.
Also,
\begin{align*}
&\left(\int_0^t\|E_{\alpha,1}(-m(\cdot)t^{\alpha})u_1(\cdot)\|^{\zeta}_{V_{\tilde{\gamma}}}\;dt\right)^{1/\zeta}\\
&\leq \left(\int_0^T\left(\int_{\Omega}|(m(\xi))^{\tilde{\gamma}}E_{\alpha,1}(-m(\xi)t^{\alpha})u_1(\xi)|^2\;d\mu\right)^{\zeta/2}\;dt\right)^{1/\zeta}\\
&\leq C\left(\int_0^T \left(\int_{\Omega}|(m(\xi))^{\tilde{\gamma}}u_1(\xi)|^2\;d\mu\right)^{\zeta/2}\;dt\right)^{1/\zeta}\\
&=CT^{1/\zeta}\|u_1\|_{V_{\tilde{\gamma}}},
\end{align*}
and
\begin{align*}
&\left(\int_0^T\left\|\int_0^t (t-\tau)^{\alpha-2}E_{\alpha,\alpha-1}(-m(\cdot)(t-\tau)^{\alpha})f(u(\cdot,\tau))\;d\tau \right\|^{\zeta}_{V_{\tilde{\gamma}}}\;dt\right)^{1/\zeta}\\
&\leq C\left(\int_0^T\left(\int_0^t(t-\tau)^{\alpha-2-\alpha\tilde{\gamma}} \left(\int_{\Omega}|f(u(\xi,\tau))|^2\;d\mu\right)^{1/2}\;d\tau\right)^{\zeta}\;dt\right)^{1/\zeta}\\
&\leq C\left(\int_0^T\left(\int_0^t(t-\tau)^{\alpha-2-\alpha\tilde{\gamma}} \left(\int_{\Omega}(Q_2(|u(\xi,\tau)|))^{2}\;d\mu\right)^{1/2}\;d\tau\right)^{\zeta}\;dt\right)^{1/\zeta}\\
&\leq CT^{\alpha(1-\tilde{\gamma})-1+\frac{1}{\zeta}}Q_2(K_0)\mu(\Omega)^{1/2},
\end{align*}
where in the last inequality we have used that $\alpha(1-\tilde{\gamma})+\zeta^{-1}\geq \alpha(\gamma-\tilde{\gamma})+\zeta^{-1}>1$, the boundeness of $u$, the monotone property of $Q_2$ and the fact that $\Omega$ is of finite measure.  
This in turn implies $f'(u),\partial_t u\in L^{\zeta}(0,T;V_{\tilde{\gamma}})$. We claim that 
$\mathbb{D}%
_{t}^{\alpha }u=-Au+f(u)\in L^{l}\left( (0,T);L^{2}\left( \Omega \right) \right) $. Indeed, using \eqref{IN-L2} and the fact that $1\leq l<\frac{1}{\alpha(1-\gamma)}$, we get the following estimates: 
\begin{align*}
\left(\int_0^T\|m(\cdot)E_{\alpha,1}(-m(\cdot)t^{\alpha})u_0(\cdot)\|^l_{L^2(\Omega)}\;dt\right)^{1/l}&\leq C\left(\int_0^Tt^{l(\alpha\gamma-\alpha)}\left( \int_{\Omega}\left|(m(\xi))^{\gamma}u_0(\xi)\right|^2\;d\mu\right)^{l/2}dt\right)^{1/l}\\
\nonumber &\leq CT^{\alpha\gamma-\alpha+\frac{1}{l}}\|u_0\|_{V_{\gamma}}.
\end{align*}
Now, since $1\leq l<\frac{1}{\alpha(1-\tilde{\gamma})-1}$ then 
\begin{equation*}
\left(\int_0^T\left\| m(\cdot)tE_{\alpha,2}(-m(\cdot)t^{\alpha})u_1(\cdot)\right \|^{l} _{L^{2}(\Omega)}dt\right)^{1/l}\leq CT^{1-\alpha+\alpha\tilde{\gamma} +\frac{1}{l}}\Vert
u_{1}\Vert _{V_{\tilde{\gamma}}}.  
\end{equation*}%
Using \eqref{Est-MLF2} and integrating by parts, we get that 
\begin{align*}
& \int_{0}^{t}f(u(\xi,\tau))m(\xi)(t-\tau )^{\alpha -1}E_{\alpha ,\alpha
}(-m(\xi)(t-\tau )^{\alpha })\;d\tau  \\
=& -f(u(\xi,t))+E_{\alpha ,1}(-m(\xi)t^{\alpha
})f(u(\xi,0)) \\
& +\int_{0}^{t}f'(u(\xi,\cdot))\partial_{\tau}u(\xi,\tau)E_{\alpha ,1}(-m(\xi)(t-\tau)^{\alpha })\;d\tau,\quad \xi\in \Omega,\,t>0.
\end{align*}%
It follows from the triangle inequality that
\begin{align*}
&\left\|\int_0^tm(\xi)(t-\tau)^{\alpha-1}E_{\alpha,\alpha}(-m(\xi)(t-\tau)^{\alpha})f(u(\xi,\tau))\;d\tau\right\|_{L^{l}(0,T;L^2(\Omega))}\\
&\leq \left\|f(u)\right\|_{L^{l}(0,T;L^2(\Omega))}+ \left\|E_{\alpha ,1}(-m(\xi)t^{\alpha
})f(u(\xi,0))\right\|_{L^{l}(0,T;L^2(\Omega))}\\
&+\left\|\int_{0}^{t}f'(u(\xi,\tau))\partial_{\tau}u(\xi,\tau)E_{\alpha ,1}(-m(\xi)(t-\tau)^{\alpha })\;d\tau\right\|_{L^{l}(0,T;L^2(\Omega))}
\end{align*}
Now,  
\begin{align*}
&\left\|\int_{0}^{t}f'(u(\xi,\tau))\partial_{\tau}u(\xi,\tau)E_{\alpha ,1}(-m(\xi)(t-\tau)^{\alpha })\;d\tau\right\|_{L^2(\Omega)}\\
&\leq 
\int_{0}^{t}\left\|f'(u(\xi,\tau))\partial_{\tau}u(\xi,\tau)E_{\alpha ,1}(-m(\xi)(t-\tau)^{\alpha })\right\|_{L^2(\Omega)}\;d\tau\\
&\leq C \int_{0}^{t}\left\|f'(u(\xi,\tau))\partial_{\tau}u(\xi,\tau)\right\|_{L^2(\Omega)}\;d\tau\leq C\|f'(u)\partial_{t}u\|_{L^1(0,T;L^2(\Omega))}\\
&\leq C\|f'(u)\|_{L^{\zeta'}(0,T;L^2(\Omega))}\,\|\partial_{t}u\|_{L^{\zeta}(0,T;L^2(\Omega))}\\
&\leq C\|f'(u)\|_{L^{\infty}(0,T;L^{\infty}(\Omega))}\,\|\partial_{t}u\|_{L^{\zeta}(0,T;V_{\tilde{\gamma}})}.
\end{align*}
Therefore
\begin{align*}
\left\|\int_{0}^{t}f'(u(\xi,\tau))\partial_{\tau}u(\xi,\tau)E_{\alpha ,1}(-m(\xi)(t-\tau)^{\alpha })\;d\tau\right\|_{L^{l}(0,T;L^2(\Omega))}\leq  CT^{\frac{1}{l}}\|f'(u)\|_{L^{\infty}(0,T;L^{\infty}(\Omega))}\,\|\partial_{t}u\|_{L^{\zeta}(0,T;V_{\tilde{\gamma}})}.
\end{align*}
Consequently,
\begin{align*}
&\left\|\int_0^tm(\xi)(t-\tau)^{\alpha-1}E_{\alpha,\alpha}(-m(\xi)(t-\tau)^{\alpha})f(u(\xi,\tau))\;d\tau\right\|_{L^{l}(0,T;L^2(\Omega))}
\\&\leq \|f(u)\|_{L^{l}((0,T);L^{2}(\Omega))}+CT\Vert f(u(\cdot ,0))\Vert_{L^{2}(\Omega)}\\
&+CT^{\frac{1}{l}}\|f'(u)\|_{L^{\infty}(0,T;L^{\infty}(\Omega))}\,\|\partial_{t}u\|_{L^{\zeta}(0,T;V_{\tilde{\gamma}})}.
\end{align*}

Thus, the claim is proved.  

Next, we prove the estimated of the second derivative.
\begin{align*}
\partial _{t}^{2}u(\cdot ,t)=& m(\cdot)t^{\alpha -2}E_{\alpha ,\alpha -1}(-m(\cdot)t^{\alpha })u_0(\cdot)-m(\cdot)E_{\alpha ,\alpha }(-m(\cdot)t^{\alpha })u_1(\cdot) \\
& +\int_{0}^{t}(t-\tau
)^{\alpha -2}E_{\alpha ,\alpha -1}(-m(\cdot)(t-\tau )^{\alpha })f'(u(\cdot,\tau))u_{\tau}(\cdot,\tau)d\tau
  \notag \\
&+t^{\alpha-2}E_{\alpha,\alpha-1}(-m(\cdot)t^\alpha)f(u(\cdot,0)).
\end{align*}%
By hypothesis we have that  $\alpha(1-\theta)-2+\frac{1}{\zeta'}>0$. Then
\begin{align*}
&\left\|\int_{0}^{t}(t-\tau
)^{\alpha -2}E_{\alpha ,\alpha -1}(-m(\cdot)(t-\tau )^{\alpha })
f'(u(\cdot,\tau))u_{\tau}(\cdot,\tau)d\tau\right\|_{V_{\theta}}\\
&\leq C \int_{0}^{t}(t-\tau
)^{\alpha -2-\alpha\theta}
\|f'(u(\cdot,\tau))u_{\tau}(\cdot,\tau)\|_{L^2(\Omega)}\,d\tau\\
&\leq C\|f'(u)\|_{L^{\infty}((0,T);L^{\infty}(\Omega))} \int_{0}^{t}(t-\tau
)^{\alpha -2-\alpha\theta}
\|u_{\tau}(\cdot,\tau)\|_{V_{\tilde{\gamma}}}\,d\tau\\
&\leq CT^{\frac{1}{\zeta'}+\alpha(1-\theta)-2}\|f'(u)\|_{L^{\infty}((0,T);L^{\infty}(\Omega))}\,\|\partial_tu\|_{L^{\zeta}((0,T);V_{\tilde{\gamma}})}.
\end{align*}
Moreover, proceeding as in the proof of \eqref{est-strong},  we obtain
\begin{align*}
&\left\Vert \partial _{t}^{2}u\right\Vert _{V_{\theta}}\\
&\leq  C\left( T^{\alpha (\gamma-\theta)-2}\Vert u_{0}\Vert _{V_{\gamma
}}+T^{\alpha (\tilde{\gamma}-\theta-1 )}\Vert u_{1}\Vert _{V_{\tilde{\gamma} }}\right.
 \\
& \qquad \left. CT^{\frac{1}{\zeta'}+\alpha(1-\theta)-2}\|f'(u)\|_{L^{\infty}((0,T);L^{\infty}(\Omega))}\,\|\partial_tu\|_{L^{\zeta}((0,T);V_{\tilde{\gamma}})}+T^{\alpha-2 -\alpha\theta}\Vert f(u(\cdot ,0))\Vert
_{L^{2}( \Omega)}\right) .  \notag
\end{align*}%
The proof is finished.

\end{proof}


\section{Examples}

In \cite{AGKW} the examples presented were: elliptic operators with Dirichlet boundary condition on an arbitrary bounded nonempty open set $\Omega\subset \mathbb R^n,$ the Dirichlet to Neumann operator and the fractional Laplace operator on a bounded open subset of $\mathbb R^n$ with external Dirichlet condition. \\  Those examples are still valid here, but we need to verify that the corresponding operator $A$ is strictly positive.  The last example below corresponds to Schr\"odinger operators. 

 \

\begin{enumerate}

  \item{{\bf Elliptic operators on open subsets of $ \mathbb R^n$.}} Let $\Omega $ be a bounded domain of $\mathbb{R}^{d}$ ($d\geqslant 1$). In
what follows, we define $\mathcal{L}$ by the differential operator 
\begin{equation}\label{EO}
\mathcal{L}u(x)=-\sum_{i,j=1}^{d}\partial _{x_{i}}\left( a_{ij}(x)\partial
_{x_{j}}u\right) ,\quad x\in \Omega ,
\end{equation}%
where $a_{ij}=a_{ji}\in L^{\infty }(\Omega )$, satisfy the ellipticity
condition 
\begin{equation*}
\sum_{i,j=1}^{d}a_{ij}(x)\xi _{i}\xi _{j}\geq c|\xi |^{2},\quad x\in {\Omega 
},\ \xi =(\xi _{1},\ldots ,\xi _{d})\in \mathbb{R}^{d}.
\end{equation*}%
In (\ref{EQ-NL0}), we consider the Dirichlet operator\footnote{%
Of course, the same differential operator $\mathcal{A}$ subject to Neumann
and/or Robin boundary conditions may be also allowed (see \cite{GalWarma2020}). The
results in this section remain also valid in these cases without any
modifications to the main statements.} which is also the case investigated
in detail by \cite{Ki-Ya} (assuming only that $d\leq 3$, and $\Omega $ is of
class $\mathcal{C}^{2}$ and $a_{ij}\in C^{1}(\overline{\Omega })$)\textbf{. }%
Let $A$ be the realization in $L^{2}(\Omega )$ of $\mathcal{L}$ with the
Dirichlet boundary condition $u=0$ in ${\mathbb{\partial }}\Omega $. That
is, $A$ is the self-adjoint operator in $L^{2}(\Omega )$ associated with the
Dirichlet form%
\begin{equation}
\mathcal{E}_{A}(u,v)=\sum_{i,j=1}^{d}\int_{\Omega }a_{ij}(x)\partial
_{x_{j}}u\partial _{x_{i}}vdx,\;\;u,v\in V_{1/2}=W_{0}^{1,2}(\Omega ).
\label{ex-dir}
\end{equation}%

  \item{{\bf The fractional Laplace operator.}} Let $\Omega $ be a bounded domain of $\mathbb{R}^{d}$ ($d\geqslant 1$).  Let $0<s<1$ and the form $\mathcal{E}_{A}$ with $D(\mathcal{E}%
_{A}):=W_{0}^{s,2}(\overline{\Omega })$ be defined by%
\begin{equation*}
\mathcal{E}_{A}(u,v):=\frac{C_{d,s}}{2}\int_{{\mathbb{R}}^{d}}\int_{{\mathbb{%
R}}^{d}}\frac{(u(x)-u(y))(v(x)-v(y))}{|x-y|^{d+2s}}dxdy.
\end{equation*}%
Let $A$ be the self-adjoint operator on $L^{2}(\Omega )$ associated with $%
\mathcal{E}_{A}$.
An integration by parts
argument gives that%
\begin{equation*}
D(A)=\Big\{u\in W_{0}^{s,2}(\overline{\Omega }):\;(-\Delta )^{s}u\in
L^{2}(\Omega )\Big\},\;Au=(-\Delta )^{s}u.
\end{equation*}%
The operator $A$ is the realization in $L^{2}(\Omega )$ of the fractional
Laplace operator $(-\Delta )^{s}$ with the Dirichlet exterior condition $u=0$
in ${\mathbb{R}}^{d}\backslash \Omega $. Here, $(-\Delta )^{s}$ is given by
the following singular integral%
\begin{align*}
(-\Delta )^{s}u(x)& :=C_{d,s}\mbox{P.V.}\int_{{\mathbb{R}}^{d}}\frac{%
u(x)-u(y)}{|x-y|^{d+2s}}\;dy \\
& =C_{d,s}\lim_{\varepsilon \downarrow 0}\int_{\{y\in {\mathbb{R}}%
^{d}\;|x-y|>\varepsilon \}}\frac{u(x)-u(y)}{|x-y|^{d+2s}}\;dy,\;\;x\in {%
\mathbb{R}}^{d},
\end{align*}%
provided that the limit exists, where $C_{d,s}$ is a normalization constant.
We refer to \cite{BBC,Caf3,War} and their references for more information on
the fractional Laplace operator. It has been shown in \cite{SV2} that the
operator $A$ has a compact resolvent and its first eigenvalue $\lambda
_{1}>0 $. In addition, from \cite[Theorem 6.6]{War} and \cite[Example 2.3.5]%
{GalWarma2020} we can deduce that $(\mathcal{E}_{A},W_{0}^{s,2}(\overline{\Omega }))$
is a Dirichlet space.
  
   \item{{\bf Dirichlet to Neumann operator.}} Let $\Omega $ be a bounded domain of $\mathbb{R}^{d}$ ($d\geqslant 1$).  Here we give an example where the metric space $X$ is given by the boundary
of an open set. Let $\Delta _{D}$ be the Laplace operator with Dirichlet boundary conditions (see (a) with $\mathcal{L}=-\Delta$). We denote its spectrum by $\sigma (\Delta _{D})$.
Let $\lambda \in {\mathbb{R}}\backslash \sigma (\Delta _{D})$, $g\in
L^{2}(\partial \Omega )$ and let $u\in W^{1,2}(\Omega )$ be the weak
solution of the following Dirichlet problem 
\begin{equation}
-\Delta u=\lambda u\;\mbox{ in }\;\Omega ,\;\;\;u|_{\partial \Omega }=g.
\label{CDP}
\end{equation}%
The classical Dirichlet-to-Neumann map is the operator $\mathbb{D}%
_{1,\lambda }$ on $L^{2}(\partial \Omega )$ with domain 
\begin{align*}
D(\mathbb{D}_{1,\lambda })=\Big\{g\in L^{2}(\partial \Omega ),\;\exists
\;u\in W^{1,2}(\Omega )\;& \mbox{ solution of }\;\eqref{CDP} \\
& \mbox{ and }\;\partial _{\nu }u\;\mbox{
exists in }\;L^{2}(\partial \Omega )\Big\},
\end{align*}%
and given by 
\begin{equation*}
\mathbb{D}_{1,\lambda }g=\partial _{\nu }u.
\end{equation*}%
It has been shown in \cite{AR-CPAA} that $\mathbb{D}_{1,\lambda }$ is the
self-adjoint operator on $L^{2}(\partial \Omega )$ associated with the
bilinear symmetric and continuous form $\mathcal{E}_{1,\lambda }$ with
domain $W^{\frac{1}{2},2}(\partial \Omega )$ given by 
\begin{equation*}
\mathcal{E}_{1,\lambda }(\varphi ,\psi )=\int_{\Omega }\nabla u\cdot \nabla
vdx-\lambda \int_{\Omega }uvdx,
\end{equation*}%
where $\varphi ,\psi \in W^{\frac{1}{2},2}(\partial \Omega )$ and $u,v\in 
\mathcal{H}^{1,\lambda }(\Omega )$ are such that $u|_{\partial \Omega
}=\varphi $ and $v|_{\partial \Omega }=\psi $. Here,\begin{equation*}
\mathcal{H}^{1,\lambda }(\Omega )=\Big\{u\in W^{1,2}(\Omega ),\;\;-\Delta
u=\lambda u\Big\},
\end{equation*}%
and by $-\Delta u=\lambda u$ we mean that 
\begin{equation*}
\int_{\Omega }\nabla u\cdot \nabla udx=\lambda \int_{\Omega }uvdx,\;\forall
\;v\in W_{0}^{1,2}(\Omega ).
\end{equation*}%

  \item{[{\bf Schr\"odinger operators.}]} We present here examples of 
  Schr\"odinger operators that fit into the framework developed  in the previous sections.   Such operators have been extensively studied due to their preeminent role in mathematical physics, especially quantum mechanics.
 
 The example is from \cite{MaSh} (see also the monographs \cite{Bo},\cite{ReSi4},  \cite{Si15}  and  \cite{EdEv},  Chapter VIII,  and Chapter 18 of \cite{Ma11}).   The results of this paper give necessary and sufficient conditions for the Schr\"odinger operators with a  {\it nonnegative potential } defined on $L^2(\Omega)$ with Dirichlet boundary conditions to be positive (the quadratic form is positive definite).  
 
It also gives necessary and sufficient conditions for the operator to have discrete spectrum. Here, discrete spectrum means that the spectrum consists of isolated eigenvalues, all having finite multiplicity.  Combining the two conditions yield a characterization of those open subsets of $\mathbb R^n$ for which the results of \cite{AGKW} apply. There, we considered only bounded open subsets.

 
  We will not consider the result from this reference in its full generality. We take $\Omega\subset\mathbb R^n,\, n\ge 2,$ open and nonempty (more conditions will be imposed below).
 Let $V\in L^1_{\rm{loc}}(\mathbb R^n)$ be almost everywhere nonnegative, and consider the quadratic form $a$:
 \begin{equation}\label{form-abc}
a(u,u)=\int_\Omega \vert \nabla u\vert^2dx+\int_\Omega \vert  u\vert^2Vdx
\end{equation}
with domain $D(a)=C_0^\infty(\Omega).$
 The form $a$ is closable (see e.g.   (see e.g.  \cite[Theorem 1.8.1]{Dav},  
 \cite[Section 12.4]{Mabook}, \cite[Theorem 7.2]{Bo},  \cite[Theorem 7.6.2]{Si15}).
\end{enumerate}
We denote by $H_V$ the associated self-adjoint operator.   Formally,
$$H_V=-\Delta+V.$$

We define the family $\mathcal G_d.$  Here $\mathcal G$ will be an open subset of $ \mathbb R^n$ satisfying the conditions
\begin{enumerate}
\item $\mathcal G$ is bounded and star-shaped with respect to $B(0,\rho)$ (the open ball of radius $\rho>0$ centered at $0.$ The condition means that $\mathcal G$ is star-shaped with respect to any point of $B(0,\rho).$
\item  $\rm{diam}(\mathcal G)=1.$
\end{enumerate}
 Then if $$\mathcal G_d(0)=\{x:\, d^{-1} x\in \mathcal G\},$$ the set
  $\mathcal G_d$ will be any body which is obtained from  $\mathcal G_d(0)$ by translation.
  
For $\gamma\in (0,1),$ the class (the so-called negligibility class) $\mathcal N_\gamma(\mathcal G_d;\Omega)$ consists of all compact sets $F\subset \overline{\mathcal G_d}$ such that
$$\overline{\mathcal G_d}\setminus \Omega\subset F\subset \overline{\mathcal G_d}$$ and
$$\rm{cap}(F)\le \gamma \rm{cap}(\overline{\mathcal G_d})$$

where $\rm{cap}$ is the Wiener capacity.

We denote $\mu_V$ the measure $V dx$ ($dx$ being the Lebesgue measure on $\mathbb R^n$) Take $\gamma\in (0,1) $ and  consider the following condition:

\noindent {\bf (Cp)} 
There exists $d_0>0$ and $\kappa>0$  such that 
\begin{equation}\label{cns-positive}
\displaystyle d^{-n}\inf_{F\in\mathcal N_\gamma(\mathcal G_d,\Omega)}\mu_V(\overline{\mathcal G_d}\setminus F)\ge\kappa
\end{equation}
for every $d>d_0$ and every ${\mathcal G_d}.$\\

The situation just described includes the case $V\equiv 0$ when $\Omega$ is bounded, and also some cases where $\Omega$ is unbounded.  It also includes for example cases where $\Omega\subset\mathbb R^n$ is unbounded and condition  {\bf (Cp)}  is satisfied.\\

More concretely (covered by the previous description),  when $\Omega=\mathbb R^n,$ we have the following example from \cite[Theorem 7.3]{Bo}.


Suppose the potential $V\in L^1_{\rm{loc}}( \mathbb {R}^n)$ is such that $V\geq 0$ and consider $H=-\Delta+V$, the self-adjoint operator described above. Assume further that 
$\displaystyle \lim_{\vert x\vert\to\infty} V(x) =\infty.$ Then the operator $H$ has compact resolvent and purely discrete spectrum. This result is due to K. Friedriechs.  It follows that if we consider the quantum harmonic oscillator which corresponds to $V(x)=\|x\|^2$
or more generally $\|x\|^2$ and an equivalent norm on $\mathbb{R}^n$, then $H=-\Delta+V$ has compact resolvent and the first eigenvalue is positive.  In fact the spectrum can be 
explicitly described, and so are the corresponding eigenvectors.


The case of magnetic Schr\"odinger operators (which are important in the theory of liquid crystals and superconductivity) can be found in \cite{KoMaSh04} by Kondratiev, Maz'ya and Shubin, which extends the results of  \cite{MaSh}.  For magnetic Schr\"odinger operators, see also \cite{Ko-Ma-Sh09}  and \cite{LeSi81}.  These results are presented in the monograph \cite{Ma11}, Chapter 18. Magnetic Schr\"odinger operators are formally defined on an open subset $\Omega\subset\mathbb R^n$ as 
\begin{equation}\label{magnetic}
H_{a,V}=-\sum_{j=1}^n\left(\frac{\partial }{\partial x^j}+ia_j\right)^2+V.
\end{equation}
Here $a_j=a_j(x),\, x\in\Omega, 1\le j\le n.$ If we set 
$$\nabla_au=\nabla u+iau=(\frac{\partial}{\partial x_1}+ia_1 u,\cdot\cdot\cdot,\frac{\partial}{\partial x_n}+ia_nu)$$ with the associated quadratic form initially defined on $C_c^\infty(\Omega)$, as:
\begin{equation}\label{form-abcm}
a(u,u)=\int_\Omega \vert \nabla_a u\vert^2dx+\int_\Omega \vert  u\vert^2Vdx
\end{equation}
corresponding to Dirichlet boundary conditions. As in the case of Schr\"odinger operators just presented, the necessary and sufficient conditions for discreteness and strict positivity of these operators. Here, discreteness is equivalent to having compact resolvent. In both cases, the conditions involve (the Wiener) capacity. Related results are discussed in \cite[Chapter 18]{Ma11}. An alternative condition equivalent to the discreteness of the spectrum of Schr\"odinger operators is given by Taylor in  \cite{Tay06}. The condition is in terms of scattering length, a concept that is shown by the author to have some relationship with capacity used by Maz'ya and his collaborator in the results mentioned above.



\bibliographystyle{plain}
\bibliography{biblio}

\end{document}